\documentclass[aap]{imsart}

\RequirePackage{amsthm,amsmath,amsfonts,amssymb}
\RequirePackage[numbers,sort&compress]{natbib}
\RequirePackage[colorlinks,citecolor=blue,urlcolor=blue]{hyperref}
\RequirePackage{graphicx}

\startlocaldefs
\theoremstyle{plain}
\newtheorem{axiom}{Axiom}
\newtheorem{claim}[axiom]{Claim}
\newtheorem{theorem}{Theorem}[section]
\newtheorem{lemma}[theorem]{Lemma}
\theoremstyle{remark}
\newtheorem{rem}[theorem]{Remark}
\newtheorem{definition}[theorem]{Definition}
\newtheorem{prop}[theorem]{Proposition}
\newtheorem{cor}[theorem]{Corollary}
\newtheorem*{example}{Example}

\def \a{\alpha}   \def \d{\delta}
\def \t{\theta}   \def \e{\epsilon}
 \def \l{\lambda}  \def \o{\omega}
\def \O{\Omega} \def \D{\Delta}

\def \L{\Lambda}
\endlocaldefs

\begin{document}

\begin{frontmatter}
\title{Ergodicity and Mixing of invariant capacities and applications}
\runtitle{Ergodicity and Mixing of invariant capacities and applications}

\begin{aug}
\author[A]{\fnms{Chunrong}~\snm{Feng}\ead[label=e1]{chunrong.feng@durham.ac.uk}\orcid{0000-0001-5244-4664}},
\author[B]{\fnms{Wen}~\snm{Huang}\ead[label=e2]{wenh@mail.ustc.edu.cn}\orcid{0000-0003-1074-1331}}
\author[C,A,B]{\fnms{Chunlin}~\snm{Liu}\ead[label=e3]{chunlinliu@mail.ustc.edu.cn} \orcid{0000-0001-6277-013X}}

\author[A,D]{\fnms{Huaizhong}~\snm{Zhao}\ead[label=e4]{huaizhong.zhao@durham.ac.uk}\orcid{0000-0002-8873-2040}}
\address[A]{Department of Mathematical Sciences, Durham University, DH1 3LE, United Kingdom
}

\address[B]{School of Mathematical Sciences, University of Science and Technology of China, Hefei, Anhui, 230026, P.R. China
}
\address[C]{School of Mathematical Sciences, Dalian University of Technology, Dalian, 116024, P.R. China
}
\address[D]{Research Center for Mathematics and Interdisciplinary Sciences, Shandong University, Qingdao 266237, China 
}
\address{
	\printead[presep={\ }]{e1,e2,e3,e4}
}
\end{aug}

\begin{abstract}
	We introduce the notion of common conditional expectation to investigate Birkhoff's ergodic theorem and subadditive ergodic theorem for invariant upper probabilities. If in addition, the upper probability is ergodic, we construct an invariant probability to characterize the limit of the ergodic mean.  Moreover, this skeleton probability is the unique ergodic  probability in the core of the upper probability, that is equal to all probabilities in the core on all invariant sets.
We have the following applications of these two theorems:\\
$\bullet$  provide a  strong law of large numbers for ergodic stationary sequence on upper probability spaces;\\
$\bullet$   prove the multiplicative ergodic theorem on  upper probability spaces;\\
$\bullet$ establish a criterion for the ergodicity of upper probabilities in terms of independence.

Furthermore, we introduce and study weak mixing for capacity preserving systems. Using the skeleton idea, we also provide several characterizations of weak mixing for invariant upper probabilities.

Finally, we provide  examples of ergodic and weakly mixing capacity preserving systems. As applications, we obtain new results in the classical ergodic theory. e.g. in characterizing dynamical properties on measure preserving systems, such as weak mixing, periodicity. Moreover, we use our results in the nonlinear theory to obtain the asymptotic independence, Birkhoff's type ergodic theorem, subadditive ergodic theorem, and multiplicative ergodic theorem for non-invariant probabilities. 
\end{abstract}



\end{frontmatter}

	\tableofcontents

	\section{Introduction}
Ergodic theory is a branch of mathematics that focuses on the behavior of a given measure preserving system  $(\O,\mathcal{F},P,T)$, where  $(\O,\mathcal{F},P)$ is a probability space (i.e., $\O$ is a nonempty set, $\mathcal{F}$ is a $\sigma$-algebra on $\O$, and $P$ is a probability on $\mathcal{F}$), and $T:\O\to\O$ is a measurable transformation such that $P(T^{-1}A)=P(A)$ for any $A\in\mathcal{F}$.  One of the key theorems in ergodic theory is Birkhoff's Ergodic Theorem \cite{Birkhoff1}, which provided a rigorous mathematical framework to investigate  the Boltzmann Ergodic Hypothesis, that is, for a closed system, the time averages of a physical quantity over long periods would converge to the ensemble average.  Furthermore, it has many connections with other fields apart from ergodic theory, for example, number theory (c.f. Einsiedler and Ward \cite{Ward2011} and Furstenberg \cite{Furstenberg1981}), stationary process  (c.f. Doob \cite{Doob1953}), and harmonic analysis  (c.f. Rosenblatt and Wierdl \cite{Rosenblatt1995}).  
Afterward, in order to address the conjecture for subadditive stochastic processes raised by Hammersley and Welsh \cite{HammersleyWelsh1965}, Kingman \cite{Kingman1968,Kingman1973} extended Birkhoff's Ergodic Theorem to subadditive sequences, which became known as the subadditive ergodic theorem.  Meanwhile, it also has many other applications, for example, the study of the multiplicative ergodic theorem by Oseledec \cite{Oseledec1968}. 

As research in ergodic theory has progressed, many generalizations and applications of these two theorems have been obtained. However, a majority of this research has focused on measure preserving systems. In real-world scenarios, it is often the case that we cannot find an ideal situation where the probability can be exactly determined. For example, it has been shown that the classical probability theory on measurable space may not be sufficient for modelling such situations, as in economics  (c.f. Billot \cite{BILLOT199275}, Marinacci and  Montrucchio\cite{MM}, and Schmeidler \cite{MR999273}) and statistics (c.f.  Walley \cite{MR1145491}). To address this challenge, capacities (or nonadditive probabilities) and nonlinear expectations are used as a tool for modeling heterogeneous environments, such as financial markets where biased beliefs of future price movements drive the decision of stock-market participants and create ambiguous volatility. 
Moreover, the following example is well-known in  number theory.
\begin{example}\label{ex:number theory}
	Recall the definition of upper density for a subset $A$ of $\mathbb{Z}$, given by $$\bar{d}(A)=\limsup_{n\to\infty}\frac{1}{2n+1}|A\cap [-n,n]|,$$ where $|A|$ denotes the number of elements of $A$. Let $T:\mathbb{Z}\to\mathbb{Z}, x\mapsto x+1$ and $2^\mathbb{Z}$ be the family consisting of all subsets of $\mathbb{Z}$. Then $(\mathbb{Z}, 2^{\mathbb{Z}}, \bar{d}, T)$ is a capacity preserving system, but not a measure preserving system, as $\bar{d}$ is not additive.
\end{example}

Thus, the study of dynamical systems on a capacity space is a natural and necessary extension. Due to the loss of the additivity, many classical results in probability theory and ergodic theory may fail. So far, there are only two main works on invariant capacities  by  Cerreia-Vioglio, Maccheroni and Marinacci \cite{CMM2016} and Feng, Wu and Zhao \cite{FWZ2020}, and one work on invariant sublinear expectations by Feng and Zhao  \cite{FengZhao2021}. In this paper, we focus on invariant capacities.
Following ideas in classical ergodic theory, we call $(\O,\mathcal{F},\mu,T)$ a capacity preserving system if $(\O,\mathcal{F})$ is a measurable space, $T:\O\to \O$ is a measurable transformation and $\mu$ is a $(T\text{-})$invariant capacity.

Firstly, we introduce the concept of common conditional expectation (see Definition \ref{defn:common conditional expectation}) to study ergodic theory on capacity spaces. In particular, given a standard  measurable space $(\Omega,\mathcal{F})$, and a measurable transformation $T:\Omega\to\Omega$, we prove that for any bounded $\mathcal{F}$-measurable function $f$, there exists a  bounded $\mathcal{I}$-measurable  function $g_f:\Omega\to\mathbb{R}$ such that for any $T$-invariant probability $P$ on $(\O,\mathcal{F})$, 
\begin{equation}\label{eq:123}
	\mathbb{E}_P(f\mid \mathcal{I})=g_f,\text{ } P\text{-almost surely},
\end{equation}
where $\mathcal{I}=\{A\in\mathcal{F}:T^{-1}A=A\}$ and $\mathbb{E}_P(f\mid \mathcal{I})$ is the conditional expectation of $f$ with respect to the $\sigma$-algebra $\mathcal{I}$ and the probability $P$. Note that in the nonlinear theory, there is no general definition for conditional expectations apart from some special cases, for example, conditional expectations with respect to $G$-Brownian motions introduced by Peng \cite{Peng2007}. However, the common conditional expectation with respect to $\mathcal{I}$ provides a method to define the conditional expectation for some upper expectations with respect to a special $\sigma$-algebra. More specifically, given a subset $\L$ of invariant probabilities, let $\hat{\mathbb{E}}=\sup_{P\in\L}\mathbb{E}_P$ be the upper expectation with respect to $\L$. Then we can define the nonlinear conditional expectation for $\hat{\mathbb{E}}$ with respect to $\mathcal{I}$ by 
\[\hat{\mathbb{E}}(f\mid \mathcal{I})=g_f\text{ for any bounded $\mathcal{F}$-measurable function }f.\]
It is easy to verify that  $\hat{\mathbb{E}}[1_A\cdot\hat{\mathbb{E}}[f|\mathcal{I}]]=\hat{\mathbb{E}}[1_A\cdot f]$ for any $A\in\mathcal{I}$.
Moreover,  we show that for any upper probability $V$, if $V$ is $T$-invariant then 
\begin{equation*}
	V(\O\setminus\{\o\in\O:\lim_{n\to\infty}\frac{1}{n}\sum_{i=0}^{n-1}f(T^i\o)=g_f(\o)\})=0,
\end{equation*}
where $g_f$ is the common conditional expectation obtained in \eqref{eq:123}  (see Theorem \ref{thm:common}).

Feng, Wu, and Zhao \cite{FWZ2020} introduced a definition of ergodicity of capacities  to describe the inability to decompose the system into disjoint subsystems inspired by the ergodicity of sublinear expectations \cite{FengZhao2021} as follows:
Given a capacity preserving system  $(\O,\mathcal{F},\mu,T)$, $\mu$ is said to be ergodic (with respect to $T$) if for any $B\in\mathcal{I}$ the following
two conditions hold:
\begin{longlist}
	\item 	 $\mu(B)=0$ or $\mu(B)=1$

	\item $\mu(B)=0$ or $\mu(\O\setminus B)=0.$ 
\end{longlist}
They provided a number of equivalent characterizations of the ergodicity, especially, in terms of the spectral properties of transformation operator whose eigenvalue $1$ being simple was proved. This leads to a result   that  an invariant upper probability $V$ is ergodic with respect to $T$ if and only if for any bounded $\mathcal{F}$-measurable function $f$, there exists a constant $c_f\in\mathbb{R}$ such that
$
V(\O\setminus\{\o\in\O:\lim _{n \rightarrow \infty} \frac{1}{n} \sum_{i=0}^{n-1} f(T^i \o)=c_f\})=0.
$
When $V$ is a probability, we know that $c_f=\int f dV$ by  Birkhoff's ergodic theorem for probabilities.
However, the uncertainty of the upper probability $V$ results in the loss of information of this constant $c_f$. In this paper,  we demonstrate what this constant is. In fact, one of the main results can be stated as follows:\\
\\
\textbf{The first main result (Theorem \ref{lem:ergodic of core} and Theorem \ref{thm:main Birk})}.
Let $(\O,\mathcal{F},V,T)$ be a capacity preserving system, where $V$ is an  upper probability.  Then $V$ is ergodic with respect to $T$ if and only if there exists a unique ergodic probability $Q$ on $\mathcal{F}$ such that $Q(A)\le V(A)$ for any $A\in\mathcal{F}$ and $Q(B)=V(B)$ for any $B\in\mathcal{I}$. Moreover, these are equivalent to that 
there exists an ergodic probability $Q$ on $\mathcal{F}$ with respect to $T$ such that for any  $f\in L^1(\O,\mathcal{F},Q)$,
$$
V(\O\setminus\{\o\in\O:\lim _{n \rightarrow \infty} \frac{1}{n} \sum_{i=0}^{n-1} f(T^{i} \o)=\int f dQ\})=0.
$$
This suggests that the above requirement for  ergodicity is not redundant. Meanwhile, the definition of ergodicity that $\mu(\mathcal{I})\in\{0,1\}$ suggested by \cite{CMM2016} cannot imply that the uniqueness of such $Q$ and thus the irreducibility of the dynamical system cannot hold. The invariant probability $Q$ is called invariant skeleton and satisfies $P|_{\mathcal{I}}=Q|_{\mathcal{I}}$ for any $P\in\operatorname{core}(V)$.

Recall that  the strong law of large numbers for processes on probability spaces can be obtained from Birkhoff's ergodic theorem.
We extend this result to upper probability space, that is, for any ergodic stationary process $\{Y_n\}_{n\in\mathbb{N}}$ on an upper probability space $(\O,\mathcal{F},V)$,  there exists a probability $Q$ on $(\O,\mathcal{F})$ such that $\lim_{n\to\infty}\frac{1}{n}\sum_{i=1}^nY_i=\int Y_1 dQ$ almost surely (\textbf{see Theorem \ref{thm:application for process}}).

As another application of the first main result, motivated by the ergodic theory of measure preserving systems,  we provide a characterization for ergodicity of upper probabilities in terms of independence as follows (\textbf{see Theorem \ref{thm:ergodicity of upper}}):
Let $(\O,\mathcal{F},V,T)$ be a capacity preserving system, where $V$ is an upper probability. Then
$V$ is ergodic if and only if there exists an ergodic probability $Q$ on $\mathcal{F}$ such that $V(A)=Q(A)$ for any $A\in\mathcal{I}$, and 
$\lim_{n\to\infty}\int\frac{1}{n}\sum_{i=0}^{n-1} f\cdot(g\circ T^i)dV=\int fdV\int gdQ$
for any bounded $\mathcal{F}$-measurable functions $f,g$ with $g\ge 0$, where the integral with respect to the capacity is the Choquet integral.

In addition to ergodicity, the concept of mixing plays a fundamental role in understanding the behavior of measure preserving systems. Weak mixing, a type of mixing that exhibits a certain level of randomness and unpredictability, is an important tool for understanding the properties of dynamical systems and has connections to the theory of unique ergodicity and rigidity (refer to the book by Glasner \cite{Glanser2003}). Meanwhile, the study of weak mixing has important implications for a range of research fields, including number theory and combinatorics  (refer to books by  Einsiedler and Ward \cite{Ward2011} and Furstenberg \cite{Furstenberg1981}). Therefore, studying weak mixing for capacity preserving systems is also critical for capacities.
In this paper, we introduce the definition of weak mixing for capacity preserving systems and study its properties. It is well known that an invariant probability is weakly mixing if and only if the product probability of itself is ergodic. Naturally, we want to have a similar characterization for capacities. However, Carath\'eodory's extension theorem from an algebra  to a $\sigma$-algebra is not true for capacities (see \cite[Chapter 12]{Denneberg2000} for example). To address this issue, we provide a means to define the product of two upper probabilities and show that the unique invariant skeleton of a weakly mixing upper probability  is a weakly mixing probability (\textbf{see Lemma \ref{lem:core of w.m.}}). Moreover, we prove that weak mixing for an invariant upper probability is equivalent to the ergodicity of the product upper probability of itself (\textbf{see Theorem \ref{thm:ergodic by independ.}}). More specifically, let $(\O,\mathcal{F},V,T)$ be a capacity preserving system, where $V$ is an upper probability. Then the following statements are equivalent:
\begin{longlist}
	\item $V$ is weakly mixing with respect to $T$;
	\medskip
	\item for any capacity preserving system $(\O',\mathcal{F}',V',T')$ with $V'$ being an ergodic upper probability, $V\times V'$ is ergodic with respect to $T\times T'$;
	\medskip
	\item $V\times V$ is ergodic with respect to $T\times T$.
\end{longlist}

As a corollary, we extend  Birkhoff's ergodic theorem along
polynomial subsequences to weakly mixing upper probability spaces (\textbf{see Corollary \ref{thm:Birkhoff along sequence}}). Namely, let $(\O, \mathcal{F}, V, T)$ be an weakly mixing upper probability space and let $p(x)$ be a polynomial with integer coefficients. Then there exists an weakly mixing probability $Q$ on $\mathcal{F}$ such that for any  $f\in L^r(\O,\mathcal{F},Q)$, $r>1$,
$$
V(\O\setminus\{\o\in\O:\lim _{n \rightarrow \infty} \frac{1}{n} \sum_{i=1}^n f(T^{p(i)} \o)=\int fdQ\})=0.
$$

We also provide some examples of ergodic and weak mixing capacity preserving systems by concave distortion of ergodic  and weakly mixing measure preserving systems. Meanwhile, we show that a subadditive weakly mixing capacity  must be ergodic, and provide examples to show that the reverse is not true. As applications, for an ergodic measure preserving system $(\O,\mathcal{F},P,T)$,  we establish lower and upper bounds to the limit $\lim_{n\to\infty}\frac{1}{n}\sum_{i=0}^{n-1}P^{1/2}(B\cap T^{-i}C)$ for any $B,C\in\mathcal{F}$, and utilize this limit to obtain new characterizations of the system to being 
\begin{longlist}
	\item weak mixing (\textbf{see Proposition \ref{prop1}}): $P$ is weakly mixing if and only if \[\lim_{n\to\infty}\frac{1}{n}\sum_{i=0}^{n-1}P^{1/2}(B\cap T^{-i}C)=P^{1/2}(B)P^{1/2}(C)\text{ for any }B,C\in\mathcal{F};\] 
	\item periodic (\textbf{see Proposition \ref{prop2}}): there exists $B\in\mathcal{F}$ with $P(B)>0$ such that for any $C\subset B$,
	\[\lim_{n\to\infty}\frac{1}{n}\sum_{i=0}^{n-1}P^{1/2}(B\cap T^{-i}C)=P^{1/2}(B)P(C)\] if and only if
	there exist $r\in\mathbb{N}$ and distinct points $\o_1,\ldots,\o_r\in\O$ such that $P(\{\o_i\})=\frac{1}{r}$, $i=1,2,\ldots,r$.
\end{longlist}

Further applications of the nonlinear theory are for  non-invariant probabilities (linear), namely, for a probability space $(\O,\mathcal{F},P)$ and an invertible measurable map $T:\O\to\O$. The programme of studying an ergodic theory for non-invariant probabilities was initiated in Hurewicz \cite{Hurewicz1944}. The main result obtained was Birkhoff's law of large  numbers about the convergence of pathwise average of a function along $P$ almost every trajectory. In this paper, under the help of the nonlinear ergodic theory of upper probabilities, we push the study of this problem further, of which the main results are briefly described as follows. Suppose that $\lim_{n \rightarrow \infty}\frac{1}{n}\sum_{i=0}^{n-1}P\circ T^{-i}$ exists,  i.e., the limit $\lim_{n \rightarrow \infty}\frac{1}{n}\sum_{i=0}^{n-1}P(T^{-i}A)$ exists for any $A\in\mathcal{F}$. Denote the limit by $Q(A)$ for any $A\in\mathcal{F}$. By Vitali-Hahn-Saks theorem  (see Lemma \ref{lem:VHStheorem}), $Q$ is an invariant probability. Moreover, if $Q$ is 
\begin{longlist}
	\item  ergodic (\textbf{see Theorem \ref{thm:app1}}) then 
	for any $f\in L^1(\O,\mathcal{F},Q)$,
	\begin{equation}\label{eq1.1}
		\lim_{n\to\infty}\frac{1}{n}\sum_{i=0}^{n-1}f(T^{i}\o)=\int fdQ\text{ for }P\text{-a.s. } \o\in\O,
	\end{equation}
	and \begin{equation}\label{eq1.2}
		\lim_{n\to\infty}\frac{1}{n}\sum_{i=0}^{n-1}P(B\cap T^{-i} C)=P(B)Q(C)\text{ for any }B,C\in\mathcal{F};
	\end{equation}
	\item  weakly mixing (\textbf{see Theorem \ref{thm:app2}}) then 
	\begin{equation}\label{eq1.3}
		\lim_{n\to\infty}\frac{1}{n}\sum_{i=0}^{n-1}|P(B\cap T^{-i} C)-P(B)Q(C)|^2=0\text{ for any }B,C\in\mathcal{F},
	\end{equation}
	and  for any $f\in B(\O,\mathcal{F})$, there exists a subset $J=J_f$ of $\mathbb{N}$ with $D(J)=0$ such that 
	\begin{equation}\label{eq1.4}
		\lim_{n\notin J,n\to\infty}\int f\circ T^ndP=\int fdQ.
	\end{equation}
\end{longlist}
Moreover, we can prove that \eqref{eq1.1} and \eqref{eq1.2} are equivalent, and are both equivalent to $Q$ being ergodic; and \eqref{eq1.3} and \eqref{eq1.4} are equivalent, and are both equivalent to $Q$ being weakly mixing.

Finally, as a further extension of Birkhoff's ergodic theorem for capacities, we extend Kingman's subadditive ergodic theorem to upper probability spaces by taking advantage of the common conditional expectation.\\
\textbf{The second main result (Theorem \ref{thm:main theorem})}.
Let $(\Omega,\mathcal{F})$ be a standard measurable space,  $T:\O\to\O$ be a measurable transformation, and $V$ be an invariant upper probability. Suppose that $\{f_n\}_{n\in\mathbb{N}}$ is a sequence of $\mathcal{F}$-measurable functions satisfying the following conditions:
\begin{longlist}
	\item there exists $\l >0$ such that $-\l n\le f_n(\o)\le \l n$ for any $n\in\mathbb{N}$, and $\o\in\O$;
	\medskip
	\item for each $m,n\in\mathbb{N}$, $V(\O\setminus \{\o\in\O:f_{n+m}(\o)\le f_n(\o)+f_m(T^n\o)\})=0.$ 
\end{longlist} 
Then there exists a bounded $T$-invariant $\mathcal{F}$-measurable function $f^{*}$ such that 
\[V(\O\setminus\{\o\in\O:\lim_{n\to\infty}\frac{1}{n}f_n(\o)=f^{*}(\o)\})=0.\]
Note that $f^*$ can be represented by the common conditional expectations of $\{f_n\}_{n\in\mathbb{N}}$. 

As an application of this theorem, we  extend Furstenberg-Kesten theorem \cite{FurstenbergKesten1960} from probability spaces to upper probability spaces. Moreover, we use this extension to prove the multiplicative ergodic theorem on upper probability spaces (\textbf{see Theorem \ref{thm:MET}}). Meanwhile, we also provide the subadditive ergodic theorem \textbf{(see Theorem \ref{thm:sub for noninvariant})} and  the multiplicative ergodic theorem \textbf{(see Theorem \ref{thm:met for non})} for a class of non-invariant probabilities.
\medskip

The structure of the paper is as follows. In Section \ref{sec:prelimilaries}, we recall some basic notion and prove some basic properties that we use in this paper. In Section \ref{sec:Birk}, we study  Birkhoff's ergodic theorem for upper probabilities, and prove a strong law of large numbers for processes on capacity spaces. 
In Section \ref{sec:concave}, we provide  characterizations for ergodicity of upper probabilities in terms of independence.   In Section \ref{sec:weak mixing}, we introduce the notion of weak mixing for capacity preserving systems, and provide a number of their characterizations. In Section \ref{sec:ex and open}, we will study some examples including some applications on measure preserving systems. 	In Section \ref{sec:subadditive}, we investigate subadditive ergodic theorem for upper probabilities, and prove the multiplicative ergodic theorem on upper probability spaces.

\section{Preliminaries}\label{sec:prelimilaries}
In this paper, we denote by $\mathbb{N}$, $\mathbb{Z}_+$, $\mathbb{Z}$, $\mathbb{R}$, $\mathbb{R}_+$ and $\mathbb{C}$ the set of all natural numbers, natural numbers with $0$, integers, real numbers, positive real numbers and complex numbers, respectively.

Let $(\O,\mathcal{F})$ be a measurable space. 
Denote by $B(\O,\mathcal{F})$ the set of all bounded and $\mathcal{F}$-measurable functions from $\O$ to $\mathbb{R}$. For a subset $A$ of $\O$,  write $\O\setminus A$ as $A^c$. 

If $\O$ is a topological space, then we denote by $\mathcal{B}(\O)$ the Borel $\sigma$-algebra on $\O$.  A measurable space $(\O,\mathcal{F})$ is said to be standard if there exists a complete and  separable metric space $X$ such that $(\O,\mathcal{F})$ is isomorphic to $(X,\mathcal{B}(X))$, that is, there exists a bijection $f:\O\to X$ such that for any $E\subset \O$, we have $E\in\mathcal{F}$ if and only if $f(E)\in \mathcal{B}(X)$. Note that in this paper, unless stated otherwise, we do not require measurable space $(\O,\mathcal{F})$ is standard.
\subsection{Set functions and Choquet integrals}\label{sec:set functions}
Let $(\Omega, \mathcal{F})$ be a measurable space. Recall a set function $\mu: \mathcal{F} \rightarrow[0,1]$ is
\begin{itemize}
	\item a capacity if $\mu(\emptyset)=0, \mu(\Omega)=1$, and $\mu(A) \leq \mu(B) $ for all $A, B \in \mathcal{F}$ such that $A \subset B$;
	\item concave if $\mu(A \cup B)+\mu(A \cap B) \leq \mu(A)+\mu(B) $ for all $A, B \in \mathcal{F}$;
	\item subadditive if $\mu(A \cup B) \leq \mu(A)+\mu(B) $ for all $A, B \in \mathcal{F}$ with $A \cap B=\emptyset$;
	\item additive if $\mu(A \cup B) = \mu(A)+\mu(B) $ for all $A, B \in \mathcal{F}$ with $A \cap B=\emptyset$;
	\item $\sigma$-additive if $\mu(\cup_{n=1}^\infty A_n) = \sum_{n=1}^\infty \mu(A_n)$ for all $\{A_n\}_{n\in\mathbb{N}} \subset\mathcal{F}$ with $A_i \cap A_j=\emptyset$ for any $i\neq j$;
	\item continuous from below if $\lim _{n \rightarrow \infty} \mu\left(A_n\right)=1$ for $A_n \uparrow \O$;
	\item continuous from above if $\lim _{n \rightarrow \infty} \mu\left(A_n\right)=0$ for $A_n \downarrow \emptyset$;
	\item continuous if it is both continuous from below and above;
	\item a probability if it is a $\sigma$-additive capacity.
\end{itemize}
Denote by $\D(\O,\mathcal{F})$ the set of all additive capacities on $(\O,\mathcal{F})$, and by $\D^{\sigma}(\O,\mathcal{F})$ the set of all probabilities on $(\O,\mathcal{F})$. We endow both sets with the weak* topology\footnote{Recall that a net $\left\{P_\alpha\right\}_{\alpha \in I}$ converges to $P$, in the weak*  topology if and only if $P_\alpha(A) \rightarrow$ $P(A)$ for all $A \in \mathcal{F}$ (see \cite[Page 3382]{CMM2016} for more details).}.   	Given a sub-$\sigma$-algebra $\mathcal{F}'$ of $\mathcal{F}$, denote by $\mu|_{\mathcal{F}'}$ the capacity $\mu$ restricted to $\mathcal{F}'$.
Given  $P\in \D^{\sigma}(\O,\mathcal{F})$, we denote 
\[L^r(\O,\mathcal{F},P)=\{f:\O\to\mathbb{R}:f\text{ is }\mathcal{F}\text{-measurable and }\|f\|_{r,P}:=\left(\int |f|^rdP\right)^r<\infty\},\text{ }r\ge 1\]
and 
\[L^\infty(\O,\mathcal{F},P)=\{f:\O\to\mathbb{R}:f\text{ is }\mathcal{F}\text{-measurable and }\|f\|_{\infty,P}<\infty\},\]
where $\|f\|_{\infty,P}:=\inf\{C\ge0:|f(\o)|\le C\text{ for  }P\text{-a.s. }\o\in\O\}$.

A capacity  is called an  upper probability if there exists a compact subset $\Lambda$ of $\D^{\sigma}(\O,\mathcal{F})$ in the weak* topology such that $V(A)=\max_{P\in\Lambda}P(A)\text{ for any }A\in \mathcal{F}$. In this case, for any $A\in\mathcal{F}$, there exists $P\in\L$ such that $V(A)=P(A)$.
Given a sub-$\sigma$-algebra $\mathcal{F}'$ of $\mathcal{F}$,  it is easy to check that $V|_{\mathcal{F}'}$ is also an  upper  probability. 

The core of a capacity $\mu$ is  defined by 
\begin{equation}\label{eq:core}
	\operatorname{core}(\mu)=\{P\in\D(\O,\mathcal{F}):P(A)\le \mu(A)\text{ for any }A\in\mathcal{F}\}.
\end{equation}
We remark that for a general capacity $\mu$, $\operatorname{core}(\mu)$ may be empty. If it is not empty, then $\operatorname{core}(\mu)$ is compact in the weak* topology (see \cite[Proposition 4.2]{MM} for example). 

From \cite[Page 3382 and 3383]{CMM2016} and \cite[Lemma 2.2 (ii)]{FWZ2020}, if $V$ is an upper  probability then $V$ is continuous, and hence $\operatorname{core}(V)\subset \D^{\sigma}(\O,\mathcal{F})$. Conversely, if $V$ is a concave capacity continuous from above, then $V$  is an upper probability. Meanwhile, if $V$ is continuous and  $V=\sup_{P\in\operatorname{core}(V)}P$, then $V$ is an upper probability. Indeed, if $V$ is continuous then $\operatorname{core}(V)\subset\D^\sigma(\O,\mathcal{F})$, and then by the compactness of $\operatorname{core}(V)$ we have that $V$ is an upper probability.

\begin{definition}
	In a capacity space $(\O,\mathcal{F},\mu)$, we call that a statement holds for $\mu$-almost surely ($\mu$-a.s. for short) if it holds on a set $A\in\mathcal{F}$ with $\mu(A^c)=0$.
\end{definition}
The following result can be found in \cite[III. 7.2, Theorem 2 and Corollary 8]{DunfordSchwarz1988}.
\begin{lemma}[Vitali-Hahn-Saks]\label{lem:VHStheorem}
	Let $(\O, \mathcal{F})$ be a measurable space, and $\{P_n\}_{n\in\mathbb{N}}$ be a sequence  of probabilities. Suppose that  for any $A\in \mathcal{F}$ the  limits exist $\lim _{n \rightarrow \infty} P_n(A)=Q(A)$. Then $Q$ is a probability. If we further suppose that each $P_n$ is absolutely continuous with respect to the probability $Q$, then the absolute continuity of the $P_n$ with respect to $Q$ is uniform in $n\in\mathbb{N}$, that is, for any $\epsilon>0$ there exists $\d>0$ such that for  any $A\in\mathcal{F}$, if  $Q(A)<\d$ then $ P_n(A)<\e$ for all $n\in\mathbb{N}$. 
\end{lemma}
Next, we recall  Choquet integral introduced by Choquet \cite{Choquet1955}.
A capacity $\mu$ induces Choquet integral, defined by
\[\int_\O fd\mu=\int_{0}^{\infty}\mu(\{\o\in \O:f(\o)\ge t\})dt+\int_{-\infty}^{0}\left(\mu(\{\o\in \O:f(\o)\ge t\})-1\right)dt\]
for all $\mathcal{F}$-measurable functions $f$, where the integrals on the right-hand side are Lebesgue integrals. If $\mu$ is  additive, then Choquet integral reduces to the standard additive integral.

The following result can be checked by definitions.
\begin{prop}\label{prop:choquet inegral properties}
	Let $(\O,\mathcal{F},\mu)$ be a capacity space, $f$ and $g$ be two $\mathcal{F}$-measurable functions. Then
	\begin{longlist}
		\item (Positive homogeneity): $\int_\O\alpha fd\mu=\alpha \int_\O fd\mu$ for each $\alpha \geq 0$.
		\item (Translation invariance): $\int_\O (f+\alpha 1_{\Omega})d\mu=\int_\O fd\mu+\alpha$ for each $\alpha \in \mathbb{R}$.
		\item (Monotonicity): $\int_\O fd\mu \geq \int_\O gd\mu$ if $f \geq g$.
	\end{longlist}
\end{prop}
If there is no ambiguity, we will omit $\O$, and write $\int_\O$ as $\int$ for simplicity.

The following result provides a dominated convergence theorem in a capacity space with respect to   Choquet integral, which was proved in \cite[Lemma 2.2]{FWZ2020}.
\begin{lemma}\label{lem:dominate convergence}
	Let $\mu$ be a continuous subadditive capacity on $(\O,\mathcal{F})$. For any $\mathcal{F}$-measurable functions $\{f_n\}_{n\in\mathbb{N}},$ $g$ and $h$ with $g\le f_n\le h$ for each $n\in\mathbb{N}$, and $\int gd\mu$, $\int hd\mu$ being finite, if $f_n\to f$ $\mu$-a.s. then 
	\[\lim_{n\to\infty}\int f_nd\mu=\int fd\mu.\]
\end{lemma}

\subsection{Invariant probabilities and capacities}
In this section, we fix  a measurable space $(\O,\mathcal{F})$  and  a measurable transformation $T$ from $\O$ to itself. Denote by $\mathcal{M}(T)$ the set of all ($T$-)invariant probabilities on $(\O,\mathcal{F})$. An invariant probability  $P$ is said to be ergodic if and only if $P(\mathcal{I})=\{0,1\}$.  We denote by $\mathcal{M}^e(T)$ the set of all ergodic probabilities on $(\O,\mathcal{F})$. Furthermore, if $P\times P$ is ergodic with respect to $T\times T$, then $P$ is said to be weak mixing with respect to $T$. Denote by $\mathcal{ M}^{wm}(T)$ the set of all weak mixing probabilities on $(\O,\mathcal{F})$.

For convenience to use, we list the following well-known results.
\begin{theorem}[Birkhoff's ergodic theorem\cite{Birkhoff1}]\label{thm:Birkhoff for measures}
	Let $(\Omega,\mathcal{F},P,T)$ be a measure preserving system. 
	For any $f\in L^1(\O,\mathcal{F},P)$, 
	\[\lim_{n\to\infty}\frac{1}{n}\sum_{i=0}^{n-1}f(T^i\o)=\mathbb{E}_P(f\mid \mathcal{I})(\o)\] 
	for $P$-a.s. $\o\in \O$, and $\mathbb{E}_P(f\mid \mathcal{I})$ is the conditional expectation of $f$ with respect to  the $\sigma$-algebra $\mathcal{I}$. If in addition, $P$ is ergodic  then $\mathbb{E}_P(f\mid\mathcal{I})=\mathbb{E}_P(f)$, $P$-a.s.
\end{theorem}
\begin{theorem}[Subadditive ergodic theorem\cite{Kingman1968,Kingman1973}]\label{thm:standard subadditive}
	Let $(\Omega,\mathcal{F},P,T)$ be a measure preserving system. Suppose that $\{f_n\}_{n\in\mathbb{N}}$ is a sequence of $\mathcal{F}$-measurable functions satisfying the following conditions:
	\begin{longlist}
		\item $\int |f_1|dP<\infty$;
		\medskip
		\item for each $k,n\in\mathbb{N}$, $f_{n+k}\le f_n+f_k\circ T^n$, $P$-a.s.
	\end{longlist} 
	Then there exists a $T$-invariant function $\phi:\O\to\mathbb{R}$ such that 
	\[\lim_{n\to\infty}\frac{1}{n}f_n=\phi\text{,   }P\text{-a.s. }\]
	Moreover, the function $\phi$ is given by
	\[\phi(\o)=\inf_{n\in\mathbb{N}}\frac{1}{n}\mathbb{E}_P(f_n\mid \mathcal{I})(\o)=\lim_{n\to\infty}\frac{1}{n}\mathbb{E}_P(f_n\mid \mathcal{I})(\o)\text{ for each }\o\in\O.\]
\end{theorem}
The following result is obtained by Bourgain \cite[Theorem 1]{Bourgain}.
\begin{theorem}\label{thm:Birk along poly}
	Let $(\Omega,\mathcal{F},P,T)$ be a measure preserving system, and let $p(x)$ be a polynomial with integer coefficients. If $f \in L^r(\Omega,\mathcal{F},P)$, $r>1$, then
	$$
	\lim _{n \rightarrow \infty} \frac{1}{n} \sum_{i=1}^{n} f(T^{p(i)} \o)
	$$
	exists for $P$-a.s. $\o \in \O$. Furthermore, if $T$ is weakly mixing, then
	the limit is equal to $\int fdP$, for $P$-a.s. $\o\in\O$.
\end{theorem}

The following result should be standard but we were not able to locate a clear
reference to it. Thus, we provide a proof for it.
\begin{lemma}\label{lem:invariant meausre }
	Let $(\O,\mathcal{F})$ be a measurable space, and $T:\O\to \O$ be a measurable transformation. Given $P,Q\in\mathcal{ M}(T)$  if $P(A)=Q(A)$ for any $A\in\mathcal{I}$, then $P=Q$. 
\end{lemma}
\begin{proof}
	For any bounded measurable function $f$, by Theorem \ref{thm:Birkhoff for measures}, one has that 
	\[P(\{\o\in\O:\lim_{n\to\infty}\frac{1}{n}\sum_{i=0}^{n-1}f( T^i\o)=\mathbb{E}_P(f\mid \mathcal{I})(\o)\})=1\]
	and 
	\[Q(\{\o\in\O:\lim_{n\to\infty}\frac{1}{n}\sum_{i=0}^{n-1}f( T^i\o)=\mathbb{E}_Q(f\mid \mathcal{I})(\o)\})=1.\]
	Since $\{\o\in\O:\lim_{n\to\infty}\frac{1}{n}\sum_{i=0}^{n-1}f( T^i\o)=\mathbb{E}_P(f\mid \mathcal{I})(\o)\}\in \mathcal{I}$ and $P|_{\mathcal{I}}=Q|_{\mathcal{I}}$, it follows that 
	\[Q(\{\o\in\O:\lim_{n\to\infty}\frac{1}{n}\sum_{i=0}^{n-1}f( T^i\o)=\mathbb{E}_P(f\mid \mathcal{I})(\o)\})=1.\]
	Thus, $\mathbb{E}_P(f\mid\mathcal{I})=\mathbb{E}_Q(f\mid\mathcal{I})$, $Q$-a.s.
	Since $P|_{\mathcal{I}}=Q|_{\mathcal{I}}$ and $\mathbb{E}_P(f\mid\mathcal{I})$ is $\mathcal{I}$-measurable, it follows that 
	\[\int fdP=\int\mathbb{E}_P(f\mid\mathcal{I})dP=\int\mathbb{E}_P(f\mid\mathcal{I})dQ=\int\mathbb{E}_Q(f\mid\mathcal{I})dQ=\int fdQ.\]
	The proof is completed as $f$ is arbitrary.
\end{proof}
\begin{rem}\label{rem:sing}
	Given two probabilities $P$ and $Q$, they are said to  be singular, if there exist two disjoint subsets $B,C\in\mathcal{F}$ with $\O=B\cup C$ such that $P(C)=0$ and $Q(B)=0$. 		
	By Lemma \ref{lem:invariant meausre }, it is easy to check that for any $P,Q\in\mathcal{ M}^e(T)$ with $P\neq Q$, $P$ is singular with $Q$.
\end{rem}
Next, we recall and prove some results for invariant capacities. Let $(\O,\mathcal{F},\mu)$ be a capacity space and $T$ be a measurable transformation from $\O$ to itself. Then
$\mu$ is invariant if for each $A\in\mathcal{F}$,
$\mu(A)=\mu(T^{-1}A).$ Recall that a $T$-invariant capacity is ergodic if for any $A\in\mathcal{I}$, $\mu(A)\in\{0,1\}$ and $\mu(A)=0$ or $\mu(A^c)=0$. If in addition, we suppose that $\mu$ is subadditive, then it is ergodic if and only if  for any $A\in\mathcal{I}$, $\mu(A)=0$ or $\mu(A^c)=0$.

The following result is a characterization of ergodicity for subadditive capacities via measurable functions.
\begin{lemma}[Theorem 4.4 in \cite{FWZ2020}]\label{lem:char. for ergodic}
	Let $(\O,\mathcal{F},\mu,T)$ be a capacity preserving system. If $\mu$  is subadditive and continuous then the following three statements are equivalent:
	\begin{longlist}
		\item $\mu$ is ergodic;
		\item if $f\in B(\O,\mathcal{F})$  is $T$-invariant then $f$ is  constant $\mu$-a.s.;
		\item if $f:\O\to \mathbb{R}$ is  $\mathcal{F}$-measurable and  $T$-invariant $\mu$-a.s. then $f$ is  constant $\mu$-a.s.
	\end{longlist}
	Note that (iii) $\Rightarrow$ (ii) $\Rightarrow$ (i) does not need the continuity of $\mu$.
\end{lemma}
We recall a version of Birkhoff's ergodic theorem for ergodic upper probabilities \cite[Theorem 4.5]{FWZ2020}.
\begin{theorem}\label{thm:original version}
	Let $(\O,\mathcal{F},V,T)$ be a capacity preserving system, where $V$ is an upper probability. Then $V$ is ergodic with respect to $T$ if and only if for any $f\in B(\O,\mathcal{F})$, there exists a unique $c_f\in\mathbb{R}$ such that
	$$
	\lim _{n \rightarrow \infty} \frac{1}{n} \sum_{i=0}^{n-1} f(T^{i} \o)=c_f \text{ for }V\text{-a.s. }\o\in\O.  
	$$
\end{theorem}
\begin{lemma}\label{lem:core of invariant lower probabilities}
	Let $(\O,\mathcal{F},V,T)$ be a capacity preserving system, where $V$ is an upper probability. Then
	\begin{longlist}
		\item	for any $P\in \operatorname{core}(V)$, there exists a unique $\hat{P}\in \mathcal{M}(T)\cap\operatorname{core}(V)$ such that $P(A)=\hat{P}(A)$ for any $A\in\mathcal{I}$;
		\medskip
		\item given $A\in\mathcal{F}$, if for any $P\in\operatorname{core}(V)$, $\hat{P}(A)=0$, then $V(A)=0$.
	\end{longlist}
\end{lemma} 
\begin{proof}
	The existence of $\hat{P}$ in the first statement was proved in  Corollary 1 of \cite{CMM2016}, and the uniqueness is obtained by Lemma \ref{lem:invariant meausre }.

	Fix any $A\in\mathcal{F}$ satisfying that for any $P\in\operatorname{core}(V)$, $\hat{P}(A)=0$. Let $\bar{A}=\cap_{n=1}^\infty\cup_{k=n}^\infty T^{-k}A$. Then by the invariance of $\hat{P}$, one deduces that \[\hat{P}(\bar{A})\le \sum_{i=1}^{\infty}\hat{P}(T^{-i}A)=0\] for any $P\in\operatorname{core}(V)$. Since $\cup_{k=n}^\infty T^{-k}A$ decreases to $\bar{A}$ as $n\to\infty$, it follows  that $T^{-1}\bar{A}=\bar{A}$, i.e., $\bar{A}\in\mathcal{I}$.
	This together with (i), implies that $P(\bar{A})=0$ for any $P\in\operatorname{core}(V)$. Thus, $V(\bar{A})=0$.  By the invariance of $V$, 
	\[V(A)=\lim_{n\to\infty}V(T^{-n}A)=\limsup_{n\to\infty}\int 1_{T^{-n}A}dV\le \int \limsup_{n\to\infty}1_{T^{-n}A}dV=V(\bar{A})=0,\]
	where the inequality is obtained from \cite[Corollary 9.5]{wang2009generalized}.
	The proof is completed.
\end{proof}
\begin{rem}
	In the following, the invariant  probability $\hat{P}$ given in the above lemma is called the invariant skeleton of $P\in\operatorname{core}(V)$.
\end{rem}

\subsection{Common conditional expectations}
In this subsection, we introduce the notion of common conditional expectation as follows.
\begin{definition}\label{defn:common conditional expectation}
	Let $(\O,\mathcal{F})$ be a measurable space, and $\mathcal{H}$ be a sub-$\sigma$-algebra of $\mathcal{F}$. Given a subset $\Lambda\subset \D^{\sigma}(\O,\mathcal{F})$ and $f\in B(\O,\mathcal{F})$, an $\mathcal{H}$-measurable function $g_f$ is said to be the common conditional expectation of $f$ with respect to $\mathcal{H}$ for $\Lambda$ if 
	\[\mathbb{E}_P(f\mid \mathcal{H})=g_f \text{, }P\text{-a.s.}\]
	for any $P\in\Lambda$, where $\mathbb{E}_P(f\mid \mathcal{H})$ is the conditional expectation of $f$ with respect to $\mathcal{H}$ for $P$.
\end{definition}

The following result shows that there exists a common conditional expectation with respect to $\mathcal{I}$ for all invariant probabilities on a standard measurable space. Namely,
\begin{lemma}\label{lem:common conditional expectation}
	Let $(\Omega,\mathcal{F})$ be a standard  measurable space, and $T:\Omega\to\Omega$ be  measurable. Then for any $f\in B(\Omega,\mathcal{F})$, there exists an $\mathcal{I}$-measurable function $g_f\in B(\Omega,\mathcal{F})$ such that for any $P\in\mathcal{M}(T)$, $g_f\in L^1(\Omega,\mathcal{I},P)$ and 
	\[\mathbb{E}_P(f\mid \mathcal{I})=g_f\text{, }P\text{-a.s.}\]
\end{lemma}
\begin{proof}
	Since $(\O,\mathcal{F})$ is a standard measurable space, let $\hat{\mathcal{F}}$ be a countable generating algebra of $\mathcal{F}$. For any $F\in \hat{\mathcal{F}}$, let 
	\[G(F)=\{\o\in\Omega:\lim_{n\to\infty}\frac{1}{n}\sum_{i=0}^{n-1}1_F(T^i\omega)\text{ exists}\}\]
	and  let 
	\[G(\hat{\mathcal{F}})=\cap_{F\in \hat{\mathcal{F}}}G(F).\]
	For $\o\in G(\hat{\mathcal{F}})$ and $F\in\hat{\mathcal{F}}$, define
	\[\mathfrak{p}_\o(F)=\lim_{n\to\infty}\frac{1}{n}\sum_{i=0}^{n-1}1_F(T^i\omega).\]
	It is easy to check that $\mathfrak{p}_\o$ is  nonnegative, normalized (i.e., $\mathfrak{p}_\o(\O)=1$), and finitely additive on $\hat{\mathcal{F}}$. Since the space is standard, by Carath\'eodory's extension theorem,
	for any $\o\in G(\hat{\mathcal{F}})$, $\mathfrak{p}_\o$ can be extended to a probability on $\mathcal{F}$, uniquely. We still denote them by $\mathfrak{p}_\o$.
	
	Note that $\mathfrak{p}_\o(F)$ is independent of $P\in\mathcal{M}(T)$, since it is a pointwise limit. Fix any $P\in\mathcal{M}(T)$.
	By Theorem \ref{thm:Birkhoff for measures}, one has that  $P(G(\hat{\mathcal{F}}))=1,$
	and for any  $F\in\hat{\mathcal{F}}$,
	\[\mathfrak{p}_\o(F)=\mathbb{E}_P(1_F\mid \mathcal{I})(\o)\text{ for }P\text{-a.s. }\o\in G(\hat{\mathcal{F}}).\]
	Thus, for any  $F\in\hat{\mathcal{F}}$,
	\[\mathfrak{p}_\o(F)=\mathbb{E}_P(1_F\mid \mathcal{I})(\o)\text{ for }P\text{-a.s. }\o\in \O.\]
	Note that for $P$-a.s. $\o\in\O$, $\mathbb{E}_P(1_F\mid \mathcal{I})(\o)$ is also nonnegative, normalized, and finitely additive on $\hat{\mathcal{F}}$.
	By the uniqueness of the extension, one has for $P$-a.s. $\o\in \O$,
	\begin{equation*}
		\mathfrak{p}_\o(F)=\mathbb{E}_P(1_F\mid \mathcal{I})(\o) \text{ for any }F\in\mathcal{F}.
	\end{equation*}
	Furthermore, by the  standard argument for approximation of a measurable function by simple functions, one has that for $P$-a.s. $\o\in \O$,
	\begin{equation}\label{eq2}
		\int_\O fd\mathfrak{p}_\o=\mathbb{E}_P(f\mid \mathcal{I})(\o) \text{ for any }f\in B(\O,\mathcal{F}).
	\end{equation}
	Let $g_f(\o)=\int_\O fd\mathfrak{p}_\o$ if $\o\in G(\hat{\mathcal{F}})$, and otherwise $g_f(\o)=0$. By \eqref{eq2} and the arbitrariness of $P\in\mathcal{ M}(T)$, $g_f$ is desired.
\end{proof}

\section{Birkhoff's ergodic theorem for capacities}\label{sec:Birk}
In the study of measure preserving systems, Birkhoff's ergodic theorem plays a crucial role in understanding the behavior of complex systems. It characterizes the ergodicity of a system and connects it with the time average behavior of certain functions. Upper probabilities are a natural generalization of probabilities and have important applications in areas such as decision theory and risk analysis.  In this section, we continue the work of Feng, Wu and Zhao \cite{FWZ2020} in extending Birkhoff's ergodic theorem to upper probabilities. By studying it, we gain a deeper understanding of the behavior of dynamics of these more general  systems. As an application, we provide a strong law of large numbers for stationary and ergodic sequences on upper probability spaces.
\subsection{Birkhoff's ergodic theorem for invariant upper probabilities}
In this section, we use common conditional expectations to represent the limit of ergodic mean for invariant upper probabilities. 
\begin{theorem}\label{thm:common}
	Let $(\Omega,\mathcal{F})$ be a standard measurable space,  $T:\O\to\O$ be a measurable transformation, and $V$ be an invariant upper probability. Then for any $f\in B(\O,\mathcal{F})$,
	$$
	V(\{\omega\in\O: \lim _{n \rightarrow \infty} \frac{1}{n} \sum_{i=0}^{n-1} f(T^i\omega)=g_f(\o)\}^c)=0,
	$$
	where $g_f$ is the common conditional expectation of $f$ given as in Lemma \ref{lem:common conditional expectation}.
\end{theorem}
\begin{proof}
	For any $P\in\operatorname{core}(V)$, let $\hat{P}$ be the invariant skeleton of $P$. Since $\hat{P}\in\mathcal{ M}(T)$, it follows from classical Birkhoff's ergodic theorem that 
	\[\lim_{n\to\infty}\frac{1}{n}\sum_{i=0}^{n-1}f(T^i\o)=\mathbb{E}_{\hat{P}}(f\mid\mathcal{I})(\o)\text{ for }\hat{P}\text{-a.s. }\o\in\O.\]
	By Lemma \ref{lem:common conditional expectation}, we have that   $\mathbb{E}_{\hat{P}}(f\mid\mathcal{I})(\o)=g_f(\o)$ for $\hat{P}$-a.s. $\o\in\O$.
	Since the set $\{\o
	\in \O:\lim_{n\to\infty}\frac{1}{n}\sum_{i=0}^{n-1}f(T^i\o)=g_f(\o)\}\in \mathcal{I}$,  for any $P\in\operatorname{core}(V)$, 
\begin{align*}
P(\{\o
\in \O:&\lim_{n\to\infty}\frac{1}{n}\sum_{i=0}^{n-1}f(T^i\o)=g_f(\o)\}^c)\\
&=\hat{P}(\{\o
\in \O:\lim_{n\to\infty}\frac{1}{n}\sum_{i=0}^{n-1}f(T^i\o)=g_f(\o)\}^c)=0.
\end{align*}
	Thus, 
\begin{align*}
V(\{\o
\in \O:&\lim_{n\to\infty}\frac{1}{n}\sum_{i=0}^{n-1}f(T^i\o)=g_f(\o)\}^c)\\
&=\max_{P\in\operatorname{core}(V)}P(\{\o
\in \O:\lim_{n\to\infty}\frac{1}{n}\sum_{i=0}^{n-1}f(T^i\o)=g_f(\o)\}^c)=0,
\end{align*}
	proving this result.
\end{proof}

\subsection{Birkhoff's ergodic theorem for ergodic upper probabilities}
In this section, we finish the proof of our first main result, whose proof is divided into Theorems \ref{lem:ergodic of core} and \ref{thm:main Birk}.

For the sake of presentation in the following proofs of this section, we denote 
\[\O_{f,\a}=\{\o\in\O:\lim_{n\to\infty}\frac{1}{n}\sum_{i=0}^{n-1}f(T^i\o)=\a\}\text{ for }f\in B(\O,\mathcal{F}),\a\in\mathbb{R}.\]
Now we prove the first statement in our first main result as follows.
\begin{theorem}\label{lem:ergodic of core}
	Let $(\O,\mathcal{F},V,T)$ be a capacity preserving system.
	If $V$ is an upper probability, then $V$ is ergodic if and only if there exists a (unique) $Q\in \mathcal{M}^e(T)\cap \operatorname{core}(V)$ such that for any $P\in\operatorname{core}(V)$,
	\[P(A)=Q(A) \text{ for any }A\in\mathcal{I}.\]
\end{theorem}
\begin{proof}
	($\Rightarrow$)
	Let $\hat{P}\in\operatorname{core}(V)\cap \mathcal{ M}(T)$ be the invariant skeleton of $P\in \operatorname{core}(V)$.
	
	\textbf{Step 1.} Prove that $\hat{P}\in \mathcal{M}^e(T)$ for any $P\in\operatorname{core}(V)$. 
	
	Fix any $A\in\mathcal{I}$. It suffices to prove $\hat{P}(A)=0$ or $1$.
	As $\hat{P}\in\operatorname{core}(V)$, one has 
	$\hat{P}(A)\le V(A)$.
	Since $V$ is ergodic and subadditive, it follows that  either
	\[V(A)=0\text{ or }V(A^c)=0.\]
	If $V(A)=0$, then $\hat{P}(A)=0$; if $V(A^c)=0$ then $\hat{P}(A^c)\le V(A^c)=0$ and hence $\hat{P}(A)=1$. This shows that $\hat{P}(A)=0$ or $1$, proving Step 1.

	\textbf{Step 2.} Prove the existence of $Q$. 
	
	Given any $A\in\mathcal{F}$, by Theorem \ref{thm:original version}, there exists a unique $c_A\in\mathbb{R}$ such that
	$V((\O_{1_A,c_A})^c)=0$.
	Thus, 
	\begin{equation}\label{eqq1}
		\hat{P}(\O_{1_A,c_A})=1\text{ for any }P\in\operatorname{core}(V).
	\end{equation}
	
	On the other hand, by Step 1, $\hat{P}\in \mathcal{M}^e(T)$, which together with the classical Birkhoff's ergodic theorem implies that
	\begin{equation}\label{eqq2}
		\hat{P}(\O_{1_A,\hat{P}(A)})=1\text{ for any }P\in\operatorname{core}(V).
	\end{equation}
	Comparing \eqref{eqq1} and \eqref{eqq2}, one has 
	\[\hat{P}(A)=c_A\text{ for any }P\in\operatorname{core}(V).\]
	Define
	$Q(A)=c_A\text{ for any }A\in\mathcal{F}.$
	Thus, for any $P\in\operatorname{core}(V)$,
	\[Q(A)=c_A=\hat{P}(A)\text{ for any }A\in\mathcal{F},\]
	that is, $Q=\hat{P}$ for any $P\in\operatorname{core}(V)$.
	This together with Step 1, implies that $Q\in\mathcal{ M}^e(T)\cap\operatorname{core}(V)$. The existence of $Q$ has been proved.
	
	\textbf{Step 3.}  Prove that $Q$ is the unique invariant probability in $\operatorname{core}(V)$. 
	
	We assume that  $Q'\in\mathcal{ M}(T)\cap\operatorname{core}(V)$.  By Step 2, we have that for any $P\in\operatorname{core}(V)$, $P|_{\mathcal{I}}=Q|_{\mathcal{I}}$. 
	In particular, $Q(A)=Q'(A)$ for any $A\in\mathcal{I}$, which together with Lemma \ref{lem:invariant meausre } implies that $Q=Q'$. 
	
	($\Leftarrow$) For any $A\in\mathcal{I}$, $Q(A)=0$ or $1$, as $Q$ is ergodic. This implies that either $P(A)=0$ for any $P\in\operatorname{core}(V)$  or $P(A)=1$ for any $P\in\operatorname{core}(V)$.  In the first case, $V(A)=0$. In the second case, $P(A^c)=0$ for any $P\in\operatorname{core}(V)$, and hence $V(A)=1$ and $V(A^c)=0$. As $A\in\mathcal{I}$ is arbitrary, we finish the proof.
\end{proof}

Now we prove the second statement in our first main result, which provides more information about the constant $c_f$ in Theorem \ref{thm:original version} to overcome the uncertainty caused by upper probabilities.
\begin{theorem}\label{thm:main Birk}
	Let $(\O,\mathcal{F},V,T)$ be a capacity preserving system, where $V$ is an  upper probability.  Then $V$ is ergodic with respect to $T$ if and only if there exists a unique $Q\in \mathcal{M}^e(T)\cap \operatorname{core}(V)$ such that for any  $f\in L^1(\O,\mathcal{F},Q)$,
	\begin{equation}\label{eq:1111}
		\lim _{n \rightarrow \infty} \frac{1}{n} \sum_{i=0}^{n-1} f(T^{i} \o)=\int f dQ \text{ for }V\text{-a.s. }\o\in\O.
	\end{equation}
\end{theorem}
\begin{proof}
	($\Leftarrow$) Consider $A\in\mathcal{I}$, then $Q(A)=0$ or $Q(A)=1$, as $Q\in \mathcal{ M}^e(T)$. Note that if $Q(A)=0$ then
	\[\O_{1_A,Q(A)}=\{\o\in\O:\lim_{n\to\infty}\frac{1}{n}\sum_{i=0}^{n-1}1_A(T^i\o)=0\}=\{\o\in\O:1_A(\o)=0\}=A^c.\]
	By \eqref{eq:1111}, one has that $V(A)=V((\O_{1_A,Q(A)})^c)=0$. Meanwhile,
	if $Q(A)=1$ then
	\[\O_{1_A,Q(A)}=\{\o\in\O:1_A(\o)=1\}=A.\]
	Similarly, by \eqref{eq:1111}, we have that $V(A^c)=V((\O_{1_A,Q(A)})^c)=0$. Thus, we obtain that either $V(A)=0$ or $V(A^c)=0$  for any $A\in\mathcal{I}$, and hence $V$ is ergodic.
	
	($\Rightarrow$) Let $Q\in\mathcal{ M}^e(T)\cap\operatorname{core}(V)$ be the ergodic probability  obtained by Theorem \ref{lem:ergodic of core}.
	Since $Q$ is ergodic, we have that for  any $f\in L^1(\O,\mathcal{F},Q)$,
	$Q((\O_{f,\int fdQ})^c)=0.$
	Since the set $\O_{f,\int fdQ}\in\mathcal{I}$, it follows that 
	$P((\O_{f,\int fdQ})^c)=0\text{ for any }P\in\operatorname{core}(V).$
	Thus, $V((\O_{f,\int fdQ})^c)=0.$
	The proof is completed, as the uniqueness can be obtained similar to Theorem \ref{lem:ergodic of core}.
\end{proof}
In the classical ergodic theory, the probability invariance seems to play an essential role. To break this restriction,   Hurewicz \cite{Hurewicz1944}  provided the condition of no wandering sets such that Birkhoff's ergodic theorem holds for a non-invariant probability.  In this seminal paper, this interesting problem was initiated. But   
progress along this line did not go very far as only limited number of results have been achieved since then.

As a consequence of Theorem \ref{thm:main Birk}, in the next result, we provide a condition in terms of an upper probability for Birkhoff's ergodic theorem to hold for a non-invariant probability. 
Beginning with this preliminary result, with the help of further results of ergodic theory of upper probabilities that we obtain in this paper later, we are able to push the ergodic theory of non-invariant probability a very big step by considering  kinds of ergodic theorems. Moreover, we will be able to construct an invariant upper probability and invariant skeleton from the non-invariant probability and measurable transformation.  Consequently, we obtain a number of results,
including averaging asymptotic independence, long time independence, subadditive ergodic theorem and multiplicative ergodic theorem, with no need of referring to the upper probability. See Remark \ref{rem4.5}, Theorem \ref{thm:app1}, Theorem \ref{thm:app2}, Theorem \ref{thm:sub for noninvariant} and Theorem \ref{thm:met for non} for details.
\begin{cor}\label{cor;non}
	Let $(\O,\mathcal{F},P)$ be a probability space, and $T:\O\to\O$ be a measurable transformation. If there exists an upper probability $V$ on $(\O,\mathcal{F})$ such that $V$ is ergodic with respect to $T$, and $P\in\operatorname{core}(V)$, then there exists a unique ergodic probability $Q\in\mathcal{ M}^e(T)\cap \operatorname{core}(V)$ on $(\O,\mathcal{F})$ such that  for any $f\in L^1(\O,\mathcal{F},Q)$,
	$$
	\lim _{n \rightarrow \infty} \frac{1}{n} \sum_{i=0}^{n-1} f(T^{i} \o)=\int f dQ,\text{ for }P\text{-a.s. }\o\in\O.
	$$
\end{cor} 
\begin{rem}
	\begin{longlist}
		\item	The limit $	\lim _{n \rightarrow \infty} \frac{1}{n} \sum_{i=0}^{n-1} f(T^{i} \o)$ is a pathwise limit for $f\in B(\O,\mathcal{F})$, so it does not depend on the choice of the upper probabilities. In fact, if $V'$ is another ergodic upper probability such that $P\in\operatorname{core}(V')$, and $Q'$ is the corresponding ergodic probability in $\mathcal{ M}^e(T)\cap\operatorname{core}(V')$, then $Q'=Q$, as by Corollary \ref{cor;non}, we have $\int fdQ=\int fdQ'$ for any $f\in B(\O,\mathcal{F})$.
		\medskip
		\item We will provide further results in constructing an appropriate upper probability $V$  and the ergodic probability $Q$ for a class of probabilities on $(\O,\mathcal{F})$.
	\end{longlist}
\end{rem}
Let us see a concrete example of Birkhoff's ergodic theorem for non-invariant probability.  More examples  can be found in Section \ref{sec:ex and open}.
\begin{example}\label{ex:upper}
	Let  $\O=[0,2)$ and $\mathcal{B}(\O)$ be the Borel $\sigma$-algebra on $\O$. Given an irrational number $\a$, define $T_\alpha:\O \rightarrow \O$ by
	$$
	T_\alpha(x)= \begin{cases}((x+\a) \bmod 1)+1, & x \in [0,1), \\ x-1, & x \in [1,2) .\end{cases}
	$$
	For each $i=1,2$, define
	$$
	\bar{P}_i(A):=P_i\left(A \cap [i-1,i)\right),\text { for any } A \in \mathcal{B}({\Omega}),
	$$
	where $P_i$ is the Lebesgue measure on $[i-1,i)$,
	and the upper probability is defined by
	\begin{equation}\label{eq:def of upper}
		V=\max\{\bar{P}_1,\bar{P}_2\}.
	\end{equation}
	First, it is not difficult to check that $\bar{P}_1,\bar{P}_2\in\operatorname{core}(V)$ and $V$ is $T$-invariant. 
	Note that the probability $Q:=\frac{1}{2}(\bar{P}_1+\bar{P}_2)$ is ergodic with respect to $T_\a$, by the observation that for any $A\in\mathcal{I}$,  $T_\alpha^{-2}(A\cap [i-1,i))=A\cap [i-1,i)$ for $i=1,2$, and the ergodicity  of the irrational rotation on the torus with respect to the Lebesgue measure. Note that for any $A\in\mathcal{F}$, if $Q(A)=0$ then $\bar{P}_i(A)=0$, $i=1,2$, and hence $V(A)=0$. In particular, for any $P\in\operatorname{core}(V)$,  $P|_{\mathcal{I}}=Q|_{\mathcal{I}}$, as $Q(\mathcal{I})=\{0,1\}$. By Theorem \ref{lem:ergodic of core}, we have that $V$ is ergodic. The ergodicity of $V$ was obtained in  Example 4.6 of \cite{FQZ2020} by a  different argument. By Corollary \ref{cor;non}, we have Birkhoff's ergodic theorem for probabilities $\bar{P}_1$ and $\bar{P}_2$, but both $\bar{P}_1$ and $\bar{P}_2$ are not invariant with respect to $T_\a$.
\end{example}
The following corollary shows the core structure  of any ergodic upper probability. The first statement strengths the result in Theorem \ref{lem:ergodic of core}.
\begin{cor}\label{cor:Q<=>V}
	If $V$ is an ergodic upper probability on a measurable space $(\O,\mathcal{F})$ with respect to a measurable transformation $T:\O\to \O$, then 
	$\operatorname{core}(V)\cap \mathcal{ M}(T)=\operatorname{core}(V)\cap \mathcal{ M}^e(T)$
	has only one element, denoted by $Q$. Moreover, for any $A\in\mathcal{F}$, 
	\[Q(A)=0\text{ if and only if }V(A)=0.\]
	In particular, $V(A)=Q(A)$ for any $A\in\mathcal{I}$.
\end{cor}
\begin{proof}
	From Theorem \ref{lem:ergodic of core}, there exists a unique probability $Q\in\operatorname{core}(V)\cap \mathcal{ M}^e(T)$ satisfying that $Q|_{\mathcal{I}}=P|_{\mathcal{I}}$ for any $P\in\operatorname{core}(V)$. If there exists $Q'\in \operatorname{core}(V)\cap \mathcal{ M}(T)$, then $Q$ is the invariant skeleton of $Q'$, which together with Lemma \ref{lem:invariant meausre } implies that $Q=Q'$, proving the first statement. 
	
	Now we prove the second statement. Given $A\in\mathcal{F}$, if $V(A)=0$ then $Q(A)=0$, as $Q\in\operatorname{core}(V)$. Conversely, suppose that $Q(A)=0$. By Theorem \ref{lem:ergodic of core}, we have that for any $P\in\operatorname{core}(V)$, $P|_{\mathcal{I}}=Q|_{\mathcal{I}}$, which together with (i) of Lemma \ref{lem:core of invariant lower probabilities}, 
	\[\hat{P}=Q\text{ for any }P\in\operatorname{core}(V).\]
	In particular, $\hat{P}(A)=Q(A)=0$ for any $P\in\operatorname{core}(V)$.
	Thus, by (ii) of Lemma \ref{lem:core of invariant lower probabilities}, we have that $V(A)=0$.
\end{proof}
Finally, we consider the invariant upper probabilities on a class of special measurable spaces.
\begin{cor}
	Let $(\O,\mathcal{F})$ be a measurable space and $T:\O\to \O$ be a measurable transformation. If $(\O,\mathcal{F})$ is uniquely ergodic with respect to $T$, that is, $\mathcal{ M}(T)=\{Q\}$, then each invariant upper probability is ergodic.
\end{cor}
\begin{proof}
	Fix an invariant upper probability $V$ on $(\O,\mathcal{F})$.Let $\hat{P}$ be the invariant skeleton of $P\in\operatorname{core}(V)$. 
	Since $(\O,\mathcal{F})$ is uniquely ergodic with respect to $T$, it follows that $\hat{P}=Q$ for any $P\in\operatorname{core}(V)$. It is well known that if $\mathcal{ M}(T)$ has only one element then $\mathcal{ M}(T)=\mathcal{ M}^e(T)$, and hence $Q\in\mathcal{ M}^e(T)$. By Theorem \ref{lem:ergodic of core}, it is easy to see that $V$ is ergodic.
\end{proof}
\subsection{Ergodicity of stationary processes on capacity spaces}
In classical probability theory, the notion of stationary stochastic process is one possible generalization of independent identically distribution random variables. Recently, it was defined on capacity spaces by Cerreia-Vioglio, Maccheroni and Marinacci \cite{CMM2016} as follows.
\begin{definition}
	Given a capacity space $(\Omega, \mathcal{F}, \mu)$, a stochastic process $\left\{Y_{n}\right\}_{n \in \mathbb{N}}$ is called to be stationary if  for each $n \in \mathbb{N}, k \in \mathbb{Z}_{+}$ and Borel subset $A$ of $\mathbb{R}^{k+1}$,
	$$
	\mu\left(\left\{\omega\in\O:\left(Y_{n}(\omega), \ldots, Y_{n+k}(\omega)\right) \in A\right\}\right)=\mu\left(\left\{\omega\in\O:\left(Y_{n+1}(\omega), \ldots, Y_{n+1+k}(\omega)\right) \in A\right\}\right) .
	$$
\end{definition}
In classical probability theory, independent identically distributed random variables form a stationary sequence. However, in the context of capacity theory, this result does not hold in general. A counterexample can be found in Example 2.1 of \cite{FWZ2020}.

A stationary ergodic process is a type of stochastic process that has been extensively studied in the fields of probability theory, statistics, and information theory. As it conforms to the ergodic theorem, there exists a strong law of large numbers for stationary ergodic stochastic sequences on probability spaces (see \cite[Theorem 2.1 of Chapter X]{Doob1953} for example). More recently, Feng, Wu and Zhao extended this result to upper probability  spaces \cite[Theorem 5.1]{FWZ2020}. 

In this section, we use Theorem \ref{thm:main Birk} to obtain a stronger result. 
Let $\left(\mathbb{R}^{\mathbb{N}}, \sigma(\mathcal{C})\right)$ denote the space of sequences endowed with the $\sigma$-algebra generated by the set of all cylinders $\mathcal{C}$.  Any set $C\in\mathcal{C}$ is called a cylinder, which has the following form
$$
C=\left\{\mathbf{x}=\left(x_{1}, x_{2}, x_{3}, \ldots\right):\left(x_{1}, \ldots, x_{n}\right) \in H\right\},
$$
where $n \in \mathbb{N}$ and $H \in \mathcal{B}\left(\mathbb{R}^{n}\right)$. It is well known that $\mathcal{C}$ is an algebra. We consider the left shift transformation $T: \mathbb{R}^{\mathbb{N}} \rightarrow \mathbb{R}^{\mathbb{N}}$ defined by
$$
T(\mathbf{x})=T\left(x_{1}, x_{2}, x_{3}, \ldots\right)=\left(x_{2}, x_{3}, x_{4}, \ldots\right)\text { for any } \mathbf{x}=\left(x_{1}, x_{2}, x_{3}, \ldots\right) \in \mathbb{R}^{\mathbb{N}}.
$$
The stochastic process $\left\{Y_{n}\right\}_{n \in \mathbb{N}}$ induces a measurable map from $(\Omega, \mathcal{F})$ to $\left(\mathbb{R}^{\mathbb{N}}, \sigma(\mathcal{C})\right)$ by
$$
\omega \mapsto \mathbf{Y}(\omega)=\left(Y_{1}(\omega), Y_{2}(\omega), Y_{3}(\omega), \ldots\right)\text { for any } \omega \in \Omega.
$$
Define $\mu_{\mathbf{Y}}: \sigma(\mathcal{C}) \rightarrow[0,1]$ by
\begin{equation}\label{eq:defn of process}
	\mu_{\mathbf{Y}}(C)=\mu\left(\mathbf{Y}^{-1}(C)\right)\text { for any } C \in \sigma(\mathcal{C}).
\end{equation}

It is easy to check that $\mu_{\mathbf{Y}}$ is a capacity on $\sigma(\mathcal{C})$ and $\mu_{\mathbf{Y}}$ is continuous/convex/concave if $\mu$ is continuous/convex/concave respectively, as $\mathbf{Y}^{-1}\left(\bigcup_{n=1}^{\infty} C_{n}\right)=\bigcup_{n=1}^{\infty} \mathbf{Y}^{-1}\left(C_{n}\right)$ and $\mathbf{Y}^{-1}\left(\bigcap_{n=1}^{\infty} C_{n}\right)=\bigcap_{n=1}^{\infty} \mathbf{Y}^{-1}\left(C_{n}\right)$, for any $\left\{C_{n}\right\}_{n \in \mathbb{N}} \subseteq \sigma(\mathcal{C})$. 

The following result was obtained in Proposition 5.1 of \cite{FWZ2020} to establish the relation between dynamical systems and  stationary stochastic processes  on a capacity space.
\begin{prop}\label{prop:stationary and invariant}
	Let $\mathbf{Y}=\left\{Y_{n}\right\}_{n \in \mathbb{N}}$ be a stochastic process on the capacity space $(\Omega, \mathcal{F}, \mu)$, where $\mu$ is continuous. Then $\mathbf{Y}$ is stationary if and only if $\mu_{\mathbf{Y}}$ is the shift transformation $T$-invariant.
\end{prop} 

Similar to the ergodicity of stochastic processes on probability spaces, Feng, Wu and Zhao \cite{FWZ2020} introduced the ergodicity of stochastic processes on  capacity spaces as follows.
\begin{definition}
	The stationary stochastic process $\textbf{Y}$ on a capacity space $(\Omega, \mathcal{F}, \mu)$ is called ergodic if the left shift transformation $T$ is ergodic with respect to $\mu_{\mathbf{Y}}$.
\end{definition} 

Now we give the strong law of large numbers for ergodic stationary  stochastic sequences on an upper probability space.
\begin{theorem}\label{thm:application for process}
	Let $V$ be an  upper probability on a measurable space $(\O,\mathcal{F})$. Given a bounded stationary process $\textbf{Y}=\{Y_n\}_{n\in\mathbb{N}}$ on the capacity space $(\O,\mathcal{F},V)$, if $\textbf{Y}$ is ergodic, then there exists $Q\in\operatorname{core}(V)$ such that 
	\[\lim_{n\to\infty}\frac{1}{n}\sum_{i=1}^nY_i=\int Y_1dQ\text{, }V\text{-a.s.}\]
\end{theorem}
\begin{rem}
	There are many references on strong law of large numbers for capacities under different definitions, see for example \cite{MR3484493,MR3021837,MR2414259,MR2135316,Peng2009,CMM2016}. In particular, Feng, Wu and Zhao \cite{FWZ2020} replaced the independent identically distributed hypothesis by the stationarity and ergodicity. It was obtained in \cite{FWZ2020} that there exists a constant $c$ such that 
	\[\lim_{n\to\infty}\frac{1}{n}\sum_{i=1}^nY_i=c\text{, }V\text{-a.s.}\]
	By strengthening Birkhoff's ergodic theorem for capacity preserving systems, in the following, we obtain a probability $Q\in\operatorname{core}(V)$ such that $c=\int Y_1dQ$. Moreover, if $V$ is a probability, then  $\operatorname{core}(V)=\{V\}$, and so this result is the classical strong law of large numbers in ergodic theory for probabilities (see \cite[Page 24]{Krengel+1985} for example).
\end{rem}
Before proving  Theorem \ref{thm:application for process}, we prove a lemma.
\begin{lemma}\label{lem:core of process}
	Let $V$ be an  upper probability on a measurable space $(\O,\mathcal{F})$, and  $\textbf{Y}=\{Y_n\}_{n\in\mathbb{N}}$ be a bounded stationary process on a capacity space $(\O,\mathcal{F},V)$. Then for any $\widetilde{P}\in \operatorname{core}(V_{\textbf{Y}})$, there exists  $P'\in\operatorname{core}(V)$ such that 
	\begin{equation}\label{eq:33}
		\widetilde{P}(C)=P'(\textbf{Y}^{-1}(C)) \text{ for any }C\in\sigma(\mathcal{C}).
	\end{equation}
\end{lemma}
\begin{proof}
	Let $\mathcal{Y}^\infty=\textbf{Y}^{-1}(\sigma(\mathcal{C}))$. We check that $\mathcal{Y}^\infty$ is a $\sigma$-algebra. For this, let $A_1,A_2,\ldots\in\mathcal{Y}^\infty$. Then there exist $C_{A_1},C_{A_2},\ldots\in\sigma(\mathcal{C})$ such that $A_i=\textbf{Y}^{-1}(C_{A_i})$ for each $i\in\mathbb{N}$. Note that $\cup_{i=1}^\infty A_i=\textbf{Y}^{-1}(\cup_{i=1}^\infty C_{A_i})\in \textbf{Y}^{-1}(\sigma(\mathcal{C}))=\mathcal{Y}^\infty$. Meanwhile, it is easy to check that $\O\in\mathcal{Y}^\infty$ and $\mathcal{Y}^\infty$ is closed under complementation, proving that $\mathcal{Y}^\infty$ is a $\sigma$-algebra. Now let $V_{\textbf{Y}}$ be defined as in \eqref{eq:defn of process}. Then $V_{\textbf{Y}}$ is an upper probabilities as well. Indeed, as $\mathbf{Y}^{-1}\left(\bigcup_{n=1}^{\infty} C_{n}\right)=\bigcup_{n=1}^{\infty} \mathbf{Y}^{-1}\left(C_{n}\right)$ and $\mathbf{Y}^{-1}\left(\bigcap_{n=1}^{\infty} C_{n}\right)=\bigcap_{n=1}^{\infty} \mathbf{Y}^{-1}\left(C_{n}\right)$, for any $\left\{C_{n}\right\}_{n \in \mathbb{N}} \subseteq \sigma(\mathcal{C})$, we have that $V_{\textbf{Y}}$ is continuous, which together with the fact that $V_{\textbf{Y}}(C)=\max_{P\in\operatorname{core}(V)}P(\textbf{Y}^{-1}C)$ for any $C\in\mathcal{F}$, implies that $V_{\textbf{Y}}$ is an upper probability.
	
	Fix $\widetilde{P}\in \operatorname{core}(V_{\textbf{Y}})$. As $V_{\textbf{Y}}$ is an upper probability, it follows that  $\widetilde{P}$ is a probability on $\sigma(\mathcal{C})$. For any $A\in\mathcal{Y}^\infty$, let $C_A\in\sigma(\mathcal{C})$ such that $A=\textbf{Y}^{-1}(C_A)$.  Define a set function $P^*:\mathcal{Y}^\infty\to [0,1]$ via 
	\[P^*(A)=\widetilde{P}(C_A)\text{ for any }A\in\mathcal{Y}^\infty.\]
	First, we prove that the definition is well-defined, i.e.,  for any $A\in\mathcal{Y}^\infty$, if there exist $C_A,C_A'\in \sigma(\mathcal{C})$ such that $A=\textbf{Y}^{-1}(C_A)=\textbf{Y}^{-1}(C_A')$ then $\widetilde{P}(C_A\Delta C_A')=0$, where $C_A\Delta C_A'=(C_A\setminus C_A')\cup (C_A'\setminus C_A)$. Indeed, since $\widetilde{P}\in \operatorname{core}(V_{\textbf{Y}}),$ one has that
	\begin{align*}
		\widetilde{P}(C_A\Delta C_A')\le&V_{\textbf{Y}}(C_A\Delta C_A')\le V_{\textbf{Y}}(C_A\cap (C_A')^c)+V_{\textbf{Y}}(C_A'\cap (C_A)^c)\\
		=&V((\textbf{Y}^{-1}C_A)\cap(\textbf{Y}^{-1}(C_A')^c))+V((\textbf{Y}^{-1}C_A')\cap(\textbf{Y}^{-1}(C_A)^c))\\
		=&V(A\cap A^c)+V(A\cap A^c)=0,
	\end{align*}
	proving this definition is well-defined.
	
	Now we prove that $P^*$ is a probability on $\mathcal{Y}^\infty$:
	First, it is easy to check that $P^*(\emptyset)=0$ and $P^*(\O)=1$. Next we only need to prove the $\sigma$-additivity of $P^*$. Fix any $\{A_n\}_{n\in\mathbb{N}}\subset \mathcal{Y}^\infty$ with $A_i\cap A_j=\emptyset$ for any $i\neq j$.  Then 
	\[\widetilde{P}(C_{A_i}\cap C_{A_j})\le V_{\textbf{Y}}(C_{A_i}\cap C_{A_j})=V(\textbf{Y}^{-1}C_{A_i}\cap \textbf{Y}^{-1}C_{A_j})=V(A_j\cap A_j)=0,\text{ for any }i\neq j,\]
	which together with the fact that $\cup_{i=1}^\infty A_i=\textbf{Y}^{-1}(\cup_{i=1}^\infty C_{A_i}) $ implies that
	\[P^*(\cup_{i=1}^\infty A_i)=\widetilde{P}(\cup_{i=1}^\infty C_{A_i})=\sum_{i=1}^\infty \widetilde{P}( C_{A_i})=\sum_{i=1}^\infty P^*(A_i).\]
	Moreover, by the construction, it is easy to check that 	$\widetilde{P}(C)=P(\textbf{Y}^{-1}(C)) \text{ for any }C\in\sigma(\mathcal{C}).$
	
	To finish the proof, now we only need to prove $P^*$ can be extended to a probability $P'$ on $\mathcal{F}$  such that $P'|_{\mathcal{Y}^\infty}=P^*$ and $P'\in\operatorname{core}(V)$. Define a functional $I$ on $B(\O,\mathcal{F})$ by 
	\[I(f)=\sup_{P\in \operatorname{core}(V)}\int f dP,\text{ for any }f\in B(\O,\mathcal{F}).\]
	By the linearity of integrals of probabilities $P\in\operatorname{core}(V)$, it is easy to check  that  $I(f+g)\le I(f)+I(g)$ for any $f,g\in B(\O,\mathcal{F})$ and $I(\l f)=\l I(f)$ for any $\l\ge0$ and $f\in B(\O,\mathcal{F}) $, i.e.,  $I$ is a sublinear functional on $B(\O,\mathcal{F})$. As $J(f):=\int f dP^*$ is a linear functional on $B(\O,\mathcal{Y}^\infty)$ with $J(f)\le I(f)$ for any $f\in B(\O,\mathcal{Y}^\infty)$ and $B(\O,\mathcal{Y}^\infty)$ is a linear subspace of $B(\O,\mathcal{F})$, it follows from Hahn-Banach dominated extension theorem that there exists a bounded linear functional $J'$ on $B(\O,\mathcal{F})$  such that 
	\begin{equation}\label{eq:5}
		J'(f)=J(f), \text{ for any }B(\O,\mathcal{Y}^\infty),
	\end{equation}and 
	\begin{equation}\label{eq:4}
		J'(f)\le I(f),\text{ for any }f\in B(\O,\mathcal{F}).
	\end{equation}
	By Riesz representation theorem (see \cite[Lemma 13.3]{Aliprantis1999} for example), there exists $P'\in\D(\O,\mathcal{F})$ such that $J'(f)=\int fdP'$ for any $f\in B(\O,\mathcal{F})$.
	
	Now we prove that $P'$ is desired.
	By \eqref{eq:4}, we have that for any $A\in\mathcal{F}$, $P'(A)\le I(1_A)= V(A)$, and hence $P'\in\operatorname{core}(V)$. In particular, $P'$ is a probability on $(\O,\mathcal{F})$, as $V$ is an upper probability. From \eqref{eq:5}, for any  $A\in\mathcal{Y}^\infty$, one deduces that $P'(A)=P^*(A)$, and hence
	$\widetilde{P}(C)=P^*(\textbf{Y}^{-1}(C))= P'(\textbf{Y}^{-1}(C))\text{ for any }C\in\sigma(\mathcal{C}).$ The proof is completed.
\end{proof}
With the help of the above lemma, we are able to prove Theorem \ref{thm:application for process}.
\begin{proof}[Proof of Theorem \ref{thm:application for process}]
	By Proposition \ref{prop:stationary and invariant}, $V_{\textbf{Y}}$ is the left shift transformation $T$-invariant. Define $f: \mathbb{R}^{\mathbb{N}} \rightarrow \mathbb{R}$ by $f(\mathbf{x})=f\left(x_1, x_2, x_3, \ldots\right)=x_1$, for any $\mathbf{x}=\left(x_1, x_2, x_3, \ldots\right) \in \mathbb{R}^{\mathbb{N}}$. Since $V_{\textbf{Y}}$ is ergodic with respect to $T$, then by Theorem \ref{thm:main Birk} we deduce  that there exists  $\widetilde{Q}\in\operatorname{core}(V_{\textbf{Y}})\cap \mathcal{M}^e(T)$ such that
	$$
	\lim _{n \rightarrow \infty} \frac{1}{n} \sum_{i=0}^{n-1} f(T^{i} \mathbf{x})=\int fd\widetilde{Q}\text{ for }V_{\textbf{Y}}\text{-a.s. }\mathbf{x}\in \mathbb{R}^\mathbb{N}.
	$$
	Notice that $\frac{1}{n} \sum_{i=1}^n x_i=\frac{1}{n} \sum_{i=0}^{n-1} f(T^{i} \mathbf{x})$ for each $n\in\mathbb{N}$, we have
	$$
	\mathbf{Y}^{-1}\left(\left\{\mathbf{x}\in\mathbb{R}^{\mathbb{N}}: \lim _{n \rightarrow \infty} \frac{1}{n} \sum_{i=1}^n x_i=\int fd\widetilde{Q}\right\}\right)=\left\{\omega\in\O: \lim _{n \rightarrow \infty} \frac{1}{n} \sum_{i=1}^n Y_i(\omega)=\int fd\widetilde{Q}\right\}.
	$$
	Thus, 
	\[\lim_{n \rightarrow \infty}\frac{1}{n}\sum_{i=1}^n Y_i(\o)=\int fd\widetilde{Q}\text{ for }V\text{-a.s. }\o\in \O.\]
	By Lemma \ref{lem:core of process}, there exists $Q\in\operatorname{core}(V)$ such that $\widetilde{Q}(A)=Q(\textbf{Y}^{-1}A)$ for any $A\in\mathcal{B}(\mathbb{R}^\mathbb{N})$.Thus,
	$$
	\int_{\mathbb{R}^N} fd \widetilde{Q}=\int_{\mathbb{R}^{\mathbb{N}}}f d Q\left(\mathbf{Y}^{-1}\right)=\int_{\Omega} f(\mathbf{Y}) d Q=\int_{\Omega} Y_1 dQ,
	$$
	proving the theorem.
\end{proof}

\section{Ergodicity in terms of independence}\label{sec:concave}
In this section, we characterize the ergodicity of upper probabilities in terms of independence.
As a reminder, the ergodicity of an invariant probability measure can be characterized as follows (see for example (2.31) in \cite{Ward2011}).
\begin{prop}\label{prop:inde}
	Let $(\O,\mathcal{F},P,T)$ be a measure preserving system. Then $P$ is ergodic if and only if 
	\[\lim_{n\to\infty}\int\frac{1}{n}\sum_{i=0}^{n-1}1_B\cdot( 1_C\circ T^i)dP=P(B)P(C)\text{ for any }B,C\in\mathcal{F}.\]
\end{prop}
\subsection{Upper probabilities versus probabilities in terms of independence}
The following proposition shows that for an ergodic upper probability, it satisfies a corresponding result in Proposition \ref{prop:inde} if and only if it is a probability. 
\begin{prop}\label{prop:eq}
	Let $(\O,\mathcal{F},V,T)$ be a capacity preserving system, where $V$ is an ergodic upper probability. Then the following two statements are equivalent:
	\begin{longlist}
		\item $\lim_{n\to\infty}\int\frac{1}{n}\sum_{i=0}^{n-1}1_B\cdot( 1_C\circ T^i)dV=V(B)V(C)$ for any $B,C\in\mathcal{F}$;
		\medskip
		\item $V$ is a probability.
	\end{longlist}
\end{prop}
\begin{proof}
	(ii) $\Rightarrow$ (i). This is obtained by Proposition \ref{prop:inde}.
	
	(i) $\Rightarrow$ (ii). Let $Q\in\mathcal{ M}^e(T)\cap \operatorname{core}(V)$ be obtained in Theorem \ref{lem:ergodic of core}. Fix any $B,C\in\mathcal{F}$. By Theorem \ref{thm:main Birk}, one has that 
	\[\lim_{n\to\infty}\frac{1}{n}\sum_{i=0}^{n-1}1_B\cdot( 1_C\circ T^i)=Q(C)\cdot 1_B\text{, }V\text{-a.s.}\]
	This combined with Lemma \ref{lem:dominate convergence} implies that 
	\[\lim_{n\to\infty}\int\frac{1}{n}\sum_{i=0}^{n-1}1_B\cdot( 1_C\circ T^i)dV=V(B)Q(C).\]
	Taking $B=\O$, one has that $V(B)=1$, and hence by (i),  $Q(C)=V(C)$. The proof is completed, as $C\in\mathcal{F}$ is arbitrary. 
\end{proof}
\subsection{Characterizations of ergodicity of upper probabilities by weak independence}
In this subsection, we provide the following characterization of ergodicity in terms of ``weak'' independence.
\begin{theorem}\label{thm:ergodicity of upper}
	Let $(\O,\mathcal{F},V,T)$ be a capacity preserving system, where $V$ is an upper probability. Then the following four statements are equivalent:
	\begin{longlist}
		\item $V$ is ergodic;
		\medskip
		\item there exists $Q\in\operatorname{core}(V)\cap \mathcal{ M}^e(T)$ such that $V(A)=Q(A)$ for any $A\in\mathcal{I}$, and 
		\[\lim_{n\to\infty}\int\frac{1}{n}\sum_{i=0}^{n-1} f\cdot(g\circ T^i)dV=\left(\int fdV\right)\left(\int gdQ\right)\]
		for any $f,g\in B(\O,\mathcal{F})$ such that $g\ge 0$;
		\medskip
		\item there exists $Q\in\operatorname{core}(V)\cap \mathcal{ M}(T)$ such that 
		\[\lim_{n\to\infty}\int\frac{1}{n}\sum_{i=0}^{n-1} 1_B\cdot(1_C\circ T^i)dV=V(B)Q(C)\]
		for any $B,C\in \mathcal{F}$;
		\medskip
		
		\item there exists $Q\in\D^{\sigma}(\O,\mathcal{F})$ such that  $V|_{\mathcal{I}}\ll Q|_{\mathcal{I}}$ (i.e., for any $A\in \mathcal{I}$ if $Q(A)=0$ then $V(A)=0$) and
		\begin{equation}\label{eq:1234}
			\liminf_{n\to\infty}\int\frac{1}{n}\sum_{i=0}^{n-1} 1_B\cdot(1_C\circ T^i)dV\ge V(B)Q(C)\text{ for any }B,C\in\mathcal{F}.
		\end{equation}
	\end{longlist}
\end{theorem}
\begin{proof}
	(i) $\Rightarrow$ (ii). Suppose that $V$ is ergodic. Given $f,g\in B(\O,\mathcal{F})$ with $g\ge0$,
	as $V$ is an upper probability, and there exists $M>0$ such that
	\[-M\le \frac{1}{n}\sum_{i=0}^{n-1}f(\o)g(T^i\o)\le M \text{ for each }\o\in\O\text{ and }n\in\mathbb{N},\]
	it follows from Lemma \ref{lem:dominate convergence} and Theorem \ref{thm:main Birk} that there exists $Q\in\mathcal{ M}^e(T)\cap\operatorname{core}(V)$ such that
	\begin{equation*}
		\lim_{n\to\infty}\int \frac{1}{n}\sum_{i=0}^{n-1}f(\o)g(T^i\o)dV(\o)=\int \left(f(\o)\int gdQ\right)dV(\o)=\left(\int fdV\right)\left(\int gdQ\right),
	\end{equation*}
	which finishes the proof of (ii).
	
	(ii) $\Rightarrow$ (iii). This is obtained by taking $f=1_B$ and $g=1_C$.
	
	(iii) $\Rightarrow$ (iv). It suffices to prove that $V|_{\mathcal{I}}\ll Q|_{\mathcal{I}}$. Indeed, for any $A\in\mathcal{I}$ with $Q(A)=0$, we have
	\[0=V(\O)Q(A)=\lim_{n \rightarrow \infty}\int\frac{1}{n}\sum_{i=0}^{n-1} 1_\O\cdot(1_A\circ T^i)dV=V(A),\]
	proving (iv).
	
	(iv) $\Rightarrow$ (i). Given $A\in\mathcal{I}$,  it suffices to prove that either $V(A)=0$ or  $V(A^c)=0$. Indeed, let $B=A$ and $C=A^c$ in \eqref{eq:1234}. Note that $A^c\in\mathcal{I}$, so $1_{A^c}\circ T^i=1_{A^c}$, then we obtain that 
	\[Q(A^c)V(A)\le0.\]
	If $V(A)\neq 0$, then $Q(A^c)=0$, which, with the assumption that $V|_{\mathcal{I}}\ll Q|_{\mathcal{I}}$, implies that $V(A^c)=0$.  Thus $V$ is ergodic.
\end{proof}
Moreover, we provide a characterization of ergodicity of  an upper probability via the asymptotic independence of probabilities in its core.
\begin{theorem}\label{thm:ergodic via independece of core}
	Let $(\O,\mathcal{F},V,T)$ be a capacity preserving system, where $V$ is an upper probability. Then the following statements are equivalent:
	\begin{longlist}
		\item $V$ is ergodic;
		\medskip
		\item there exists $Q\in\operatorname{core}(V)\cap \mathcal{ M}^e(T)$ such that for any  $P\in\operatorname{core}(V)$, 
		\[\lim_{n\to\infty}\frac{1}{n}\sum_{i=0}^{n-1}P(B\cap T^{-i}C)=P(B)Q(C)\text{ for any }B,C\in\mathcal{F};\]
		\medskip
		\item there exists $Q\in\operatorname{core}(V)\cap \mathcal{ M}(T)$ such that for any  $P\in\operatorname{core}(V)$, 
		\begin{equation}\label{eq:111111}
			\lim_{n\to\infty}\frac{1}{n}\sum_{i=0}^{n-1}P(B\cap T^{-i}C)=P(B)Q(C)\text{ for any }B,C\in\mathcal{F}.
		\end{equation}
	\end{longlist}
\end{theorem}
\begin{proof}
	(i) $\Rightarrow$ (ii). This is a direct consequence of Theorem \ref{thm:main Birk} and dominated convergence theorem.
	
	(ii) $\Rightarrow$ (iii). It is trivial.
	
	(iii) $\Rightarrow$ (i). By \eqref{eq:111111}, one has that for any $P\in\operatorname{core}(V)$,
	\begin{equation}\label{eq:q}
		P(A)=P(A)Q(A),\text{ for any }A\in\mathcal{I}.
	\end{equation}
	Firstly, we prove that $Q$ is ergodic. Indeed, for any $A\in\mathcal{I}$, if $V(A)=0$ then $Q(A)=0$. If $V(A)>0$ then there exists $P\in\operatorname{core}(V)$ such that $P(A)>0$. It follows from \eqref{eq:q} that $Q(A)=1$. Therefore, $Q$ is ergodic.
	
	Next we prove that $Q|_{\mathcal{I}}=P|_{\mathcal{I}}$ for any $P\in\operatorname{core}(V)$. Fix any $P\in\operatorname{core}(V)$ and $A\in\mathcal{I}$. As $Q$ is ergodic, $Q(A)=0$ or $1$. By \eqref{eq:q}, if $Q(A)=0$ then $P(A)=0$. If $Q(A)=1$ then by applying \eqref{eq:q} on $A^c$, we have $P(A^c)=0$, and hence $P(A)=1=Q(A)$. Thus, $Q|_{\mathcal{I}}=P|_{\mathcal{I}}$ for any $P\in\operatorname{core}(V)$, which together with Theorem \ref{lem:ergodic of core}, implies that $V$ is ergodic.
\end{proof}

\subsection{Applications: weak independence for non-invariant probabilities in terms of the invariant skeleton}
In Corollary \ref{cor;non},  Birkhoff's ergodic theorem for non-invariant probabilities was discussed. The following corollary as a direct consequence of the proof of Theorem \ref{thm:ergodic via independece of core} provides a further result in this direction as the asymptotic independence on large time average. 
\begin{cor}\label{cor:non}
	Suppose exactly the same conditions as that of Corollary \ref{cor;non}, then 
	\begin{equation}\label{eq:5.1a}
		\lim_{n\to\infty}\frac{1}{n}\sum_{i=0}^{n-1}P(B\cap T^{-i}C)=P(B)Q(C)\text{ for any }B,C\in\mathcal{F}.
	\end{equation}
	In particular, 
	\begin{equation}\label{eq:5.1b}
		\lim_{n\to\infty}\frac{1}{n}\sum_{i=0}^{n-1}P\circ T^{-i}=Q.
	\end{equation}
\end{cor}
\begin{rem}\label{rem4.5}
	\begin{longlist}
		\item Formula \eqref{eq:5.1b} gives the construction of the ergodic probability $Q\in\mathcal{ M}^e(T)\cap \operatorname{core}(V)$ as the average of $\{P\circ T^{-i}\}_{i=0}^\infty$ under the assumption that there exists an ergodic upper probability $V$ with $P\in\operatorname{core}(V)$.

		\medskip
		\item From the proof of Proposition \ref{prop:eq}, under the assumption that an upper probability mentioned in (i) exists, then the large time average $\lim_{n \rightarrow \infty}\frac{1}{n}\sum_{i=0}^{n-1}P'\circ T^{-i}=Q$ for all $P'\in\operatorname{core}(V)$ independent of what $P'$ is used as long as $P'\in\operatorname{core}(V)$.
		
		\medskip
		
		\item 	
		If $T$ is invertible, then the $\sigma$-algebra of invariant subsets with respect to $T$  is equal to that with respect to $T^{-1}$. If $Q$ defined by \eqref{eq:5.1b} is ergodic with respect to $T$, then $Q$ is also ergodic with respect to $T^{-1}$. Thus, $\lim_{n \rightarrow \infty}\frac{1}{n}\sum_{i=0}^{n-1}P\circ T^i=Q$ on $\mathcal{F}$, which together with \eqref{eq:5.1b} implies that 
		\begin{equation}\label{eq:6}
			\lim_{\substack{n+m\to\infty\\m,n\ge0}}\frac{1}{m+n+1}\sum_{i=-m}^nP\circ T^{i}=Q.
		\end{equation}
		
		\item
		Instead of starting with  a given ergodic upper probability, let us  consider a probability space $(\O,\mathcal{F},P)$ and an invertible measurable transformation $T:\O\to\O$. Suppose that the limit \eqref{eq:5.1b} exists, denoted by $Q$.  Then by 
		Vitali-Hahn-Saks's theorem (Lemma \ref{lem:VHStheorem}), we have $Q$ is a probability.  It is obvious that $Q\in\mathcal{ M}(T)$ and $P|_{\mathcal{I}}=Q|_{\mathcal{I}}$. If  further assume that $Q$ is ergodic, by the classical Birkhoff's ergodic theorem for invertible transformations, we have that 
		\[\lim_{\substack{n+m\to\infty\\m,n\ge0}}\frac{1}{n+m+1}\sum_{i=-m}^n1_A(T^{i}\o)=Q(A),\text{ for }Q\text{-a.s. }\o\in\O.\]
		Denote $$\widetilde{\O}=\{\o\in\O:\lim_{\substack{n+m\to\infty\\m,n\ge0}}\frac{1}{n+m+1}\sum_{i=-m}^n1_A(T^{i}\o)=Q(A)\}.$$ Then $\widetilde{\O}\in\mathcal{I}$, and so $P(\widetilde{\O})=Q(\widetilde{\O})=1$. It follows that \eqref{eq:6} holds true without referring to an upper probability beforehand.
		
		However, we can construct an upper probability under above assumptions. Let us begin with a claim that $P\circ T^{i}$ is absolutely continuous with respect to $Q$ for each $i\in\mathbb{Z}$. Indeed, for any $A\in\mathcal{F}$ with $Q(A)=0$, then $Q(A^\infty)=0$, where $A^\infty=\cup_{i=-\infty}^\infty T^iA$. As $A^\infty\in\mathcal{I}$, it follows that for any $i\in\mathbb{Z}$,
		\[P(T^iA)\le P(T^iA^\infty)=P(A^\infty)=Q(A^\infty)=0,\]
		proving this claim. Define 
		\[\l_{m,n}=\frac{1}{m+n+1}\sum_{i=-m}^nP\circ T^{i},\text{ }m,n\ge0\]and 
		
		\begin{equation}\label{eq:8=}
			V(A)=\sup_{m,n\in\mathbb{Z}_+}\{\l_{m,n}(A)\}\text{ for each }A\in\mathcal{F}.
		\end{equation}
		Now by Vitali-Hahn-Saks's theorem again, for any $\epsilon>0$, there exists $\delta>0$ such that  for any $A\in\mathcal{F}$ if $Q(A)<\d$ then $\l_{m,n}(A)<\e$ for all $m,n\in\mathbb{Z}$. Thus, it is easy to check that $V$ is continuous, and hence $V$ is an upper probability. From the definition of $V$, we have that $V$ is $T$-invariant. Then $(\O,\mathcal{F},V,T)$ is an upper probability preserving system. Moreover, if $Q$ is ergodic then $V$ is also ergodic.
		\medskip
		\item  Note that if we have no additional condition on $Q$ then the supremum of a family of probabilities defined as  \eqref{eq:8=} may not be an upper probability. For example, we consider $P=\delta_x$ for some non-periodic point $x\in \O$, and let 
		\[A_k=\{T^lx:|l|\ge k\}\text{ for each }k\in\mathbb{N}.\]
		Then $A_k$ decreases to $\emptyset$, as $k\to\infty$, and $V(A_k)=\max_{i\in\mathbb{Z}}\delta_{x}(T^{i}A_k)=1$ for each $k\in\mathbb{N}$. Thus, $V$ is not continuous, and hence it is not an upper probability.
	\end{longlist}
\end{rem}
\begin{theorem}\label{thm:app1}
	Let $(\O,\mathcal{F},P)$ be a probability space, and $T:\O\to \O$ be an invertible measurable transformation. Suppose that the limit $\lim_{n\to\infty}\frac{1}{n}\sum_{i=0}^{n-1}P\circ T^{i}$ exists, denoted by $Q$. Then $Q\in\mathcal{ M}(T)$. Moreover, if $Q$ is ergodic, then
	\begin{equation}\label{eq:4.5a}
		\lim_{\substack{n+m\to\infty\\m,n\ge0}}\frac{1}{m+n+1}\sum_{i=-m}^nP(B\cap T^i C)=P(B)Q(C)\text{ for any }B,C\in\mathcal{F}
	\end{equation}
	and  for any $f\in L^1(\O,\mathcal{F},Q)$,
	\begin{equation}\label{eq:4.5b}
		\lim_{\substack{n+m\to\infty\\m,n\ge0}}\frac{1}{m+n+1}\sum_{i=-m}^nf(T^i\o)=\int fdQ\text{ for }P\text{-a.s. } \o\in\O.
	\end{equation}
	Conversely, if \eqref{eq:4.5a} or \eqref{eq:4.5b} holds, then $Q$ is ergodic. In particular, \eqref{eq:4.5a} and \eqref{eq:4.5b} are equivalent.
\end{theorem}
\begin{proof} 
	Let $V$ be defined by \eqref{eq:8=}. Then $V$ is an upper probability by the argument in Remark \ref{rem4.5} (iv).
	By the definition of $Q$ and the assumption that $Q$ is ergodic, we can see that $Q\in\operatorname{core}(V)\cap\mathcal{ M}^e(T)$  and  $P|_\mathcal{I}=Q|_\mathcal{I}$. For any $A\in\mathcal{I}$, as $Q$ is ergodic, $Q(A)\in\{0,1\}$. So $P(T^iA)=Q(A)\in \{0,1\}$ for any $i\in\mathbb{Z}$. Thus, either 
	\[P(T^iA)=0\text{ for all }i\in\mathbb{Z}\text{ or }P(T^iA^c)=0\text{ for all }i\in\mathbb{Z}.\]
	Therefore, $V(A)=0$ or $V(A^c)=0$. This implies that $V$ is ergodic. Then applying  Corollary \ref{cor:non} and Theorem \ref{thm:main Birk} on $T$ and $T^{-1}$, we finish the proof of the first part of  this theorem, as $P\in\operatorname{core}(V)$.
	
	Now we prove the converse part. If \eqref{eq:4.5a} holds then for any $A\in\mathcal{I}$, when $P(A)>0$, we consider $0=P(A\cap A^c)=P(A)Q(A^c)$, and hence $Q(A^c)=0$; when $P(A)=0$, we consider  $0=P(A^c\cap A)=P(A^c)Q(A)=Q(A)$. Thus, $Q$ is ergodic. 
	
	If \eqref{eq:4.5b} holds, then for any $A\in\mathcal{I}$,  $1_A(\o)=Q(A)$ for $P$-a.s. Thus, $Q(A)=\{0,1\}$, and hence $Q$ is ergodic.
\end{proof}
\begin{rem}\label{rem:app1}
	In the special case that $P$ is $T$-invariant probabilities, we have $V=Q=P$ and $P\circ T^{i}=P$, $i\in\mathbb{Z}$, then the results in Theorem \ref{thm:app1} are results in classical ergodic theory with no any extra condition imposed (e.g. see (1.1.4) in Da Prato and Zabczyk \cite{DaPrato} corresponding to \eqref{eq:4.5a} and Birkhoff's law of large numbers corresponding to \eqref{eq:4.5b}). Our results are sharp in the classical ergodic theory and hold true for possibly non-invariant probabilities.
\end{rem}
\subsection{Characterizations of ergodicity of continuous concave capacities}
In this subsection, we 	provide more characterizations of ergodicity of a special type of upper probabilities, namely, continuous concave capacities.
Let us recall an important properties of concave capacities.
\begin{lemma}[Proposition 3 in \cite{Schmeidler1986}]\label{lem:intgral for concave}
	If $\mu$ is a concave capacity on a measurable space $(\O,\mathcal{F})$, then 
	\[\int fd\mu=\max_{P\in\operatorname{core}(\mu)}\int fdP\text{ for any }f\in B(\O,\mathcal{F}).\]
\end{lemma}

\begin{theorem}\label{thm:}
	Let $(\O,\mathcal{F},\mu,T)$ be a capacity preserving system. If $\mu$ is continuous from above and concave then the following four statements are equivalent:
	\begin{longlist}
		\item $\mu$ is ergodic;
		\medskip
		\item there exists $Q\in\operatorname{core}(\mu)\cap \mathcal{ M}^e(T)$ such that $\mu(A)=Q(A)$ for any $A\in\mathcal{I}$, and 
	\begin{align*}
	\lim_{n\to\infty}&\max_{P\in\operatorname{core}(\mu)}\frac{1}{n}\sum_{i=0}^{n-1}\int f\cdot(g\circ T^i)dP\\
	&=\max_{P\in\operatorname{core}(\mu)}\lim_{n\to\infty}\frac{1}{n}\sum_{i=0}^{n-1}\int f\cdot(g\circ T^i)dP=\int fd\mu\int gdQ
	\end{align*}
		for any $f,g\in B(\O,\mathcal{F})$ such that $g\ge 0$;
		\medskip
		\item there exists $Q\in\operatorname{core}(\mu)\cap \mathcal{ M}(T)$ such that 
		\[\lim_{n\to\infty}\max_{P\in\operatorname{core}(\mu)}\frac{1}{n}\sum_{i=0}^{n-1}P(B\cap T^{-i}C)=\max_{P\in\operatorname{core}(\mu)}\lim_{n\to\infty}\frac{1}{n}\sum_{i=0}^{n-1}P(B\cap T^{-i}C)=\mu(B)Q(C)\]
		for any $B,C\in \mathcal{F}$;
		\medskip
		
		\item there exists $Q\in\D^{\sigma}(\O,\mathcal{F})$ such that  $\mu|_{\mathcal{I}}\ll Q|_{\mathcal{I}}$  and
		\begin{equation*}
			\liminf_{n\to\infty}\frac{1}{n}\sum_{i=0}^{n-1}\mu(B\cap T^{-i}C)\ge\mu(B)Q(C)\text{ for any }B,C\in\mathcal{F}.
		\end{equation*}
	\end{longlist}
\end{theorem}
\begin{proof}
	(i) $\Rightarrow$ (ii). 
	From \cite[Page 3382 and 3383]{CMM2016}, since $\mu$ is concave and continuous from above, it follows that $\mu$ is an upper probability. 
	Thus, by Theorem \ref{lem:ergodic of core}, there exists  $Q\in\mathcal{ M}^e(T)\cap \operatorname{core}(\mu)$ such that for any $A\in\mathcal{I}$,  $Q(A)=P(A)$ for any $P\in\operatorname{core}(\mu)$.  Given any $f,g\in B(\O,\mathcal{F})$ with $g\ge0$, let
	\[\mathcal{A}_{g}=\{\o\in\O:\lim_{n\to\infty}\frac{1}{n}\sum_{i=0}^{n-1}g(T^i\o)=\int gdQ\}.\]
	Applying Birkhoff's ergodic theorem on the ergodic probability $Q$, one has 
	$Q(\mathcal{A}_{g})=1,$
	which together with the invariance of $\mathcal{A}_{g}$, implies that 
	\begin{equation}\label{eq11}
		P(\mathcal{A}_{g})=1\text{ for any }P\in\operatorname{core}(\mu).
	\end{equation}
	
	By \eqref{eq11} and the dominated convergence theorem, one deduces that
	\begin{equation}\label{eq:222}
		\lim_{n\to\infty}\frac{1}{n}\sum_{i=0}^{n-1}\int f\cdot (g\circ T^i)dP=\int fdP\int gdQ\text{ for any }P\in\operatorname{core}(\mu).
	\end{equation}
	Using \eqref{eq11} again, one has $\mu(\mathcal{A}_{g}^c)=0$, i.e., 	\[\lim_{n\to\infty}\frac{1}{n}\sum_{i=0}^{n-1}g(T^i\o)=\int gdQ\text{ for }\mu\text{-a.s. }\o\in\O.\]
	This implies that
	\[\lim_{n\to\infty}\frac{1}{n}\sum_{i=0}^{n-1}f(\o)g(T^i\o)=f(\o)\int gdQ\text{ for }\mu\text{-a.s. }\o\in\O.\]
	By Lemma \ref{lem:dominate convergence}, we deduce that 
	\begin{equation*}
		\lim_{n\to\infty}\int \frac{1}{n}\sum_{i=0}^{n-1}f(\o)g(T^i\o)d\mu(\o)=\int fd\mu\int gdQ.
	\end{equation*}
	Meanwhile, by Lemma \ref{lem:intgral for concave}, one has that 
	\begin{align*}
		\int \frac{1}{n}\sum_{i=0}^{n-1}f(\o)g(T^i\o)d\mu(\o)=&\max_{P\in\operatorname{core}(\mu)}\int \frac{1}{n}\sum_{i=0}^{n-1}f(\o)g(T^i\o)dP(\o)\\
		=&\max_{P\in\operatorname{core}(\mu)}\frac{1}{n}\sum_{i=0}^{n-1}\int f(\o)g(T^i\o)dP(\o).
	\end{align*}
	Thus, by \eqref{eq:222},
	\begin{align*}
		\int fd\mu\int gdQ&=\lim_{n\to\infty}\max_{P\in\operatorname{core}(\mu)}\frac{1}{n}\sum_{i=0}^{n-1}\int f(\o)g(T^i\o)dP(\o)\\
		&\ge \max_{P\in\operatorname{core}(\mu)}\lim_{n\to\infty}\frac{1}{n}\sum_{i=0}^{n-1}\int f(\o)g(T^i\o)dP(\o)\\
		&=\max_{P\in\operatorname{core}(\mu)}\int fdP\int gdQ=\int fd\mu\int gdQ,
	\end{align*}
	which finishes the proof of (ii).
	
	(ii) $\Rightarrow$ (iii). This is obtained by taking $f=1_B$ and $g=1_C$.
	
	(iii) $\Rightarrow$ (iv). Since
	\[\max_{P\in\operatorname{core}(\mu)}\frac{1}{n}\sum_{i=0}^{n-1}P(B\cap T^{-i}C)\le \frac{1}{n}\sum_{i=0}^{n-1}\mu(B\cap T^{-i}C)\text{ for any }B,C\in\mathcal{F},\]
	it follows from (iii) that
	\[	\liminf_{n\to\infty}\frac{1}{n}\sum_{i=0}^{n-1}\mu(B\cap T^{-i}C)\ge\liminf_{n\to\infty}\max_{P\in\operatorname{core}(\mu)}\frac{1}{n}\sum_{i=0}^{n-1}P(B\cap T^{-i}C)=\mu(B)Q(C)\text{ for any }B,C\in\mathcal{F}.\]
	Now we prove	$\mu|_{\mathcal{I}}\ll Q|_{\mathcal{I}}$. Indeed, for any $A\in\mathcal{I}$ with $Q(A)=0$, by (iii) for $B=\O$, $C=A$, one has that $V(A)=\max_{P\in\operatorname{core}(V)}P(A)=Q(A)=0$.
	This finishes (iv).
	
	(iv) $\Rightarrow$ (i). It can be proved by the same arguments of (iv) $\Rightarrow$ (i) in Theorem \ref{thm:ergodicity of upper}.
\end{proof}

\section{Weak mixing for capacities}\label{sec:weak mixing}
In this section, we  provide a formal definition   of weak mixing for capacity preserving systems. In classical ergodic theory, weak mixing has been extensively studied for measure preserving systems. Naturally, the concept of weak mixing should play a similar role in characterizing the level of randomness or disorder in  capacity preserving systems. 
\subsection{Definition of weak mixing of invariant capacities}
In classical ergodic theory, weak mixing can be characterized by measurable eigenfunctions (see  \cite{Ward2011} for example). Namely, for a measure preserving system $(\O,\mathcal{F},P,T)$, $P$ is weakly mixing if and only if each $\mathcal{F}$-measurable function $f:\O\to\mathbb{C}$ satisfying that $f\circ T=\lambda f$, $P$-a.s. for some $\l\in\mathbb{C}$, is  constant, $P$-a.s. It follows that $\l=1$ is the unique eigenvalue of the transformation operator $f\mapsto f\circ T$ and the eigenvalue is simple. Motivated by this characterization, we state a similar definition for capacities.
\begin{definition}
	Let $(\O,\mathcal{F},\mu,T)$ be a capacity preserving system. The capacity $\mu$ is called weak mixing (with respect to $T$) if each $\mathcal{F}$-measurable function $f:\O\to\mathbb{C}$ satisfying that $f\circ T=\lambda f$, $\mu$-a.s. for some $\l\in\mathbb{C}$, is constant, $\mu$-a.s.
\end{definition}

By Lemma \ref{lem:char. for ergodic}, it is easy to see that for a  subadditive capacity $\mu$,  weak mixing  implies ergodicity. Examples of weakly mixing capacity preserving systems can be found in Section \ref{sec:ex and open}.
It is easy to check that given any $\mathcal{F}$-measurable function $f$ that is not zero $\mu$-a.s., if there exists $\l\in\mathbb{C}$ such that $f\circ T=\l f$, $\mu$-a.s., then $|\l|=1$. This can be seen from that the Choquet integral satisfies $\int|f|^2d\mu=\int|f\circ T|^2d\mu=|\l|^2\int|f|^2d\mu$.

\subsection{Weak mixing and ergodicity on the product space of upper probabilities}
Recall that for a measure preserving system $(\Omega,\mathcal{F},P,T)$, $P$ is weakly mixing if and only if $P\times P$ is ergodic with respect to $T\times T$ (see \cite[Theorem 2.36]{Ward2011}). This raises a natural question:
is there a similar result for capacity preserving systems?  Before answering this question, we must define the product system of two capacity preserving systems. 

We note that from \cite[Chapter 12]{Denneberg2000}, the  Carath\'eodory's extension theorem from an algebra to a $\sigma$-algebra is not true for capacities. However, we have a natural method to define the product of two upper probabilities as follows: 
Let $V_i=\max_{P_i\in\operatorname{core}(V_i)}P_i$ be two upper probabilities defined on  measurable spaces $(\O_i,\mathcal{F}_i)$, $i=1,2$. Then we define
\[V_1\times V_2=\max_{(P_1,P_2)\in\operatorname{core}(V_1)\times \operatorname{core}(V_2)}P_1\times P_2,\]
where $P_1\times P_2$ is the  product measure of $P_1$ and $P_2$.  Since $\operatorname{core}(V_1)\times \operatorname{core}(V_2)$ is a compact subset of $\Delta^{\sigma}(\O_\times\O_2,\mathcal{F}_1\times\mathcal{F}_2)$, $V_1\times V_2$ is an upper probability on $(\O_1\times\O_2,\mathcal{F}_1\times\mathcal{F}_2).$
It is easy to check that for any $A\in\mathcal{F}_1$ and $B\in\mathcal{F}_2$, 
\[V_1\times V_2(A\times B)=V_1(A)\cdot V_2(B).\]
Thus, the definition is well-defined. Since $V_1\times V_2$ is also an upper probability, it follows that 
\[V_1\times V_2=\max_{\widetilde{P}\in\operatorname{core}(V_1\times V_2)}\widetilde{P}.\]
Note that generally, $\operatorname{core}(V_1)\times \operatorname{core}(V_2)\subsetneq \operatorname{core}(V_1\times V_2)$. For example, in Example \ref{ex:upper},  $\frac{1}{2}(\bar{P}_1\times \bar{P}_2+\bar{P}_2\times \bar{P}_1)\in\operatorname{core}(V_1\times V_2)$ but not in $\operatorname{core}(V_1)\times \operatorname{core}(V_2)$.

Consider two capacity preserving systems $(\O_1, \mathcal{F}_1, V_1, T_1)$ and $(\O_2, \mathcal{F}_2, V_2, T_2)$, where $V_1$ and $V_2$ are upper probabilities. Let $(\O_1 \times \O_2, \mathcal{F}_1 \times \mathcal{F}_2, V_1 \times V_2)$ be their product space defined as above, and $T_1\times T_2:\O_1\times\O_2\to \O_1\times\O_2$, by $ (\o_1,\o_2)\mapsto(T_1\o_1,T_2\o_2)$. 
\begin{prop} Let $(\O_1, \mathcal{F}_1, V_1, T_1)$ and $(\O_2, \mathcal{F}_2, V_2, T_2)$ be two capacity preserving systems, where $V_1$ and $V_2$ are upper probabilities.
	Then the upper probability $V_1\times V_2$ is $T_1\times T_2$-invariant. In particular,
	$(\O_1 \times \O_2, \mathcal{F}_1 \times \mathcal{F}_2, V_1 \times V_2,T_1\times T_2)$  is a capacity preserving system.
\end{prop}
\begin{proof}
	Set
	\[\mathcal{C}=\{\widetilde{A}\in\mathcal{F}_1\times \mathcal{F}_2:V_1\times V_2((T_1\times T_2)^{-1}\widetilde{A})=V_1\times V_2(\widetilde{A})\}.\]
	Since $V_1\times V_2(A\times B)=V_1(A)\cdot V_2(B)$ for any $A\in\mathcal{F}_1$ and $B\in\mathcal{F}_2$ and $V_i$ is $T_i$-invariant, $i=1,2$, it follows that 
	\[A\times B\in \mathcal{C}\text{ for any }A\in\mathcal{F}_1,B\in \mathcal{F}_2.\]
	Thus, by monotone class theorem, it suffices to prove that $\mathcal{C}$ is a monotone class. Indeed, for any $\{\widetilde{A}_n\}_{n\in\mathbb{N}}\subset
	\mathcal{C}$ with $\widetilde{A}_n\subset \widetilde{A}_{n+1}$ for each $n\in\mathbb{N}$, as $V_1\times V_2$ is continuous, one has that 
	\begin{align*}
		V_1\times V_2(\cup_{n=1}^\infty\widetilde{A}_n)=&\lim_{n\to\infty}V_1\times V_2(\widetilde{A}_n)\\
		=&\lim_{n\to\infty}V_1\times V_2((T_1\times T_2)^{-1}\widetilde{A}_n)=V_1\times V_2((T_1\times T_2)^{-1}\cup_{n=1}^\infty\widetilde{A}_n).
	\end{align*}
	This shows that $\cup_{n=1}^\infty\widetilde{A}_n\in\mathcal{C}$. Similarly, we can prove that for any $\{\widetilde{A}_n\}_{n\in\mathbb{N}}\subset
	\mathcal{C}$ with $\widetilde{A}_n\supset \widetilde{A}_{n+1}$ for each $n\in\mathbb{N}$, $\cap_{n=1}^\infty A_n\in\mathcal{C}$. Thus, $\mathcal{C}$ is a monotone class, proving this proposition.
\end{proof}
In what follows, we focus on the study of weak mixing of upper probabilities.
Similar to Theorem \ref{lem:ergodic of core}, the following lemma provides a characterization of the cores of weakly mixing upper probabilities.
\begin{lemma}\label{lem:core of w.m.}
	Let $(\O,\mathcal{F},V,T)$ be a capacity preserving system, where $V$ is an upper probability. Then $V$ is weakly mixing if and only if there exists a (unique) $Q\in \mathcal{M}^{wm}(T)\cap \operatorname{core}(V)$ such that for any $P\in\operatorname{core}(V)$,
	$P|_{\mathcal{I}}=Q|_{\mathcal{I}}$.
\end{lemma}
\begin{proof}
	($\Rightarrow$) By Theorem \ref{lem:ergodic of core}, one has proved that there exists a unique $Q\in\mathcal{ M}^e(T)\cap\operatorname{core}(V)$ such that for any $P\in\operatorname{core}(V)$, 	$P|_{\mathcal{I}}=Q|_{\mathcal{I}}$.
	Thus, we only need to prove that $Q$ is weakly mixing. To see  this, consider an $\mathcal{F}$-measurable function $f:\O\to\mathbb{C}$ with $f\circ T=\l f$, $Q$-a.s. for some $\l\in\mathbb{C}$. Since $Q(\{\o\in\O:f(T\o)=\l f(\o)\}^c)=0$  and $V$ is ergodic, it follows from Corollary \ref{cor:Q<=>V} that 
	\[V(\{\o\in\O:f( T\o)=\l f(\o)\}^c)=0.\]
	Since $V$ is weakly mixing, there exists a constant $c_f$ such that $V(\{\o\in\O:f(\o)=c_f\}^c)=0$. Since $Q\in\operatorname{core}(V)$, one has $Q(\{\o\in\O:f(\o)=c_f\}^c)=0$, i.e., $f$ is  constant, $Q$-a.s. This implies that $Q\in\mathcal{ M}^{wm}(T)$.

	($\Leftarrow$) Consider an $\mathcal{F}$-measurable function $f:\O\to\mathbb{C}$ with $f\circ T=\l f$, $V$-a.s. for some constant $\l\in\mathbb{C}$. Since $Q\in\operatorname{core}(V)$, it follows that $Q(\{\o\in\O:f(T\o)=\l f(\o)\}^c)=0$, which together with the assumption $Q\in\mathcal{ M}^{wm}(T)$, implies that there  exists a constant $c_f$ such that $Q(\{\o\in\O:f(\o)=c_f\}^c)=0$.  By Theorem  \ref{lem:ergodic of core}, we have that $V$ is ergodic. Thus, it follows from Corollary \ref{cor:Q<=>V} that 
	$V(\{\o\in\O:f(\o)=c_f\}^c)=0,$ i.e., $f$ is  constant $V$-a.s. Thus,
	$V$ is weakly mixing.
\end{proof}
As a corollary of Lemma \ref{lem:core of w.m.}, we provide a Birkhoff's ergodic theorem along polynomials on weakly mixing upper probability spaces.
\begin{cor}\label{thm:Birkhoff along sequence}
	Let $(\Omega,\mathcal{F},V,T)$ be a capacity preserving system, where $V$ is a weakly mixing upper probability, and let $p(x)$ be a polynomial with integer coefficients. Let $Q$ be the unique weakly mixing probability given in Lemma \ref{lem:core of w.m.}. Then $f \in L^r(\Omega,\mathcal{F},Q)$, $r>1$, one has that
	$$
	\lim _{n \rightarrow \infty} \frac{1}{n} \sum_{i=1}^{n} f(T^{p(i)} \o)=\int fdQ \text{ for }V\text{-a.s. }\o\in\O.
	$$
\end{cor}
\begin{proof}
	Since $Q$ is weakly mixing, it follows from Theorem \ref{thm:Birk along poly} that 
	$$
	Q(\{\lim _{n \rightarrow \infty} \frac{1}{n} \sum_{i=1}^{n} f(T^{p(i)} \o)=\int fdQ\}^c)=0.
	$$
	By Corollary \ref{cor:Q<=>V}, we have that 
	$$
	V(\{\lim _{n \rightarrow \infty} \frac{1}{n} \sum_{i=1}^{n} f(T^{p(i)} \o)=\int fdQ\}^c)=0,
	$$
	proving this theorem.
\end{proof}
\begin{rem}
	When $r=1$  and $p(x)=x^2$, even if $V$ is a weakly mixing probability, Corollary \ref{thm:Birkhoff along sequence} may not hold. In fact, Buczolich and Mauldin \cite{BuczolichMauldin2010} proved that it is not true that given a measure preserving  system $(\O,\mathcal{F},P,T)$ and $f\in L^1(\mu)$, the ergodicity mean 
	$\lim_{n\to\infty}\frac{1}{n}\sum_{i=1}^nf\circ T^{n^2}$ converges, $P$-a.s.
\end{rem}
We now proceed to establish a correspondence between the weak mixing of a capacity preserving system and the ergodicity of its product system. 
\begin{theorem}\label{thm:ergodic by independ.}
	Let $(\O,\mathcal{F},V,T)$ be a capacity preserving system, where $V$ is an upper probability. Then the following statements are equivalent:
	\begin{longlist}
		\item $V$ is weakly mixing with respect to $T$;
		\medskip
		\item For any capacity preserving system $(\O',\mathcal{F}',V',T')$ with $V'$ being an ergodic upper probability, $V\times V'$ is ergodic with respect to $T\times T'$;
		\medskip
		\item $V\times V$ is ergodic with respect to $T\times T$.
	\end{longlist}
\end{theorem}
\begin{proof}
	(i) $\Rightarrow$ (ii). Denote $\mathcal{I}'=\{A'\in\mathcal{F}':T'^{-1}A'=A'\}$ and $\widetilde{\mathcal{I}}=\{\widetilde{A}\in\mathcal{F}\times\mathcal{F}':(T\times T')^{-1}\widetilde{A}=\widetilde{A}\}.$  Applying Lemma \ref{lem:core of w.m.} and Theorem \ref{lem:ergodic of core} on $(\O,\mathcal{F},V,T)$ and $(\O',\mathcal{F}',V',T')$, respectively, we have that there exists a unique $Q\in\mathcal{ M}^{wm}(T)\cap\operatorname{core}(V)$ such that for any $P\in\operatorname{core}(V)$,
	$P|_{\mathcal{I}}=Q|_{\mathcal{I}}$,
	and   a unique $Q'\in\mathcal{ M}^{e}(T')\cap\operatorname{core}(V')$ such that for any $P'\in\operatorname{core}(V')$,
	$P'|_{\mathcal{I}'}=Q'|_{\mathcal{I}'}$.
	
	By contradiction, we assume that $V\times V'$ is not ergodic. Then there exists $F\in B(\O\times\O',\widetilde{\mathcal{I}})$ which is not constant, $V\times V'$-a.s. 
	On the other hand,  since $Q\in\mathcal{ M}^{wm}(T)$ and $Q'\in\mathcal{ M}^e(T)$, it follows that $Q\times Q'$ is ergodic (see Theorem 2.36 in \cite{Ward2011}). Thus, there exists a constant $c_F\in\mathbb{C}$ such that 
	\[F=c_F,\quad Q\times Q'\text{-a.s.}\]  
	Let \[\widetilde{A}=\{(\o,\o')\in\O\times\O':F(\o,\o')=c_F\}^c\in\mathcal{F}\times \mathcal{F}'.\] Then
	\begin{equation}\label{eq1}
		V\times V'(\widetilde{A})>0\text{ and }Q\times Q'(\widetilde{A})=0.
	\end{equation}
	
	By Fubini's theorem (a suitable version can  be found in \cite[Theorem 2.36]{Folland1999}), one has that 
	\[0=Q\times Q'(\widetilde{A})=\int 1_{\widetilde{A}}d(Q\times Q')=\int Q'(\widetilde{A}_\o)dQ(\o),\]
	where $\widetilde{A}_\o=\{\o'\in\O':(\o,\o')\in\O\times \O'\}$ for each $\o\in\O$. Thus, 
	\[Q'(\widetilde{A}_\o)=0\text{ for }Q\text{-a.s. }\o\in\O. \]
	By Corollary \ref{cor:Q<=>V},, we obtain that for $Q$-a.s. $\o\in\O$, $V'(\widetilde{A}_\o)=0$. Thus, for $Q$-a.s. $\o\in\O$, $P'(\widetilde{A}_\o)=0$ for any $P'\in\operatorname{core}(V')$. Note that $\{\o\in\O:P'(\widetilde{A}_\o)=0\}$ is a $\mathcal{F}$-measurable set, as $P'$ is a probability. Then we deduce that 
	\[Q(\{\o\in\O:P'(\widetilde{A}_\o)=0\}^c)=0,\text{ for any }P'\in\operatorname{core}(V').\]
	Using Corollary \ref{cor:Q<=>V} again, one has that 
	\[V(\{\o\in\O:P'(\widetilde{A}_\o)=0\}^c)=0,\text{ for any }P'\in\operatorname{core}(V').\]
	Therefore, for any $P\in\operatorname{core}(V)$,
	\[P(\{\o\in\O:P'(\widetilde{A}_\o)=0\}^c)=0.\]
	By Fubini's theorem, one deduces that 
	\[P\times P'(\widetilde{A})=\int P'(\widetilde{A}_\o)dP(\o)=0.\]
	As $(P,P')\in\operatorname{core}(V)\times \operatorname{core}(V')$ is arbitrary, it follows that
	\[V\times V'(\widetilde{A})=\max_{(P,P')\in\operatorname{core}(V)\times \operatorname{core}(V')}P\times P'(\widetilde{A})=0,\]
	which is a contradiction with \eqref{eq1}. Thus, $V\times V'$ is ergodic.
	
	(ii) $\Rightarrow$ (iii). This is trivial.
	
	(iii) $\Rightarrow$ (i). Let $f$ be an $\mathcal{F}$-measurable function with $f\circ T=\l f$ $V$-a.s. for some $\l\in \mathbb{C}$. If $f=0$, $V$-a.s. there is nothing to prove. So we suppose that 
	\begin{equation}\label{eq:7}
		V(\{\o\in\O:f(\o)=0\}^c)>0.
	\end{equation}
	Let 
	\[F(\o_1,\o_2)=f(\o_1)\overline{f(\o_2)} \text{ for any }(\o_1,\o_2)\in \O\times \O.\]
	Then $F(T\o_1,T\o_2)=|\lambda|^2F(\o_1,\o_2)$ for $V\times V$-a.s. $(\o_1,\o_2)\in \O\times \O$. Since $f\circ T=\l f$, $V$-a.s., it follows that $\int|f\circ T|dV=|\l|\int|f|dV$. By \eqref{eq:7}, $\int |f|dV>0$, and thus, by the $T$-invariance of $V$, we have that  $|\l|=1$.  So $F(T\o_1,T\o_2)=F(\o_1,\o_2)$ for $V\times V$-a.s. $(\o_1,\o_2)\in \O\times \O$. Applying  Lemma \ref{lem:char. for ergodic} on the ergodic upper probability $V\times V$, one deduces that there exists $a\in\mathbb{C}$ such that  $F=a$, $V\times V$-a.s. Denote 
	\[\widetilde{\O}=\{(\o_1,\o_2)\in\O_1\times\O_2:f(\o_1)\overline{f(\o_2)}=a\}.\]
	Then
	\begin{equation}\label{eq:8}
		V\times V(\widetilde{\O}^c)=0.
	\end{equation}
	By \eqref{eq:7}, there exists $P_0\in\operatorname{core}(V)$ such that \begin{equation}\label{eq:15}
		P_0(\{\o\in\O:f(\o)\neq0\})>0.
	\end{equation} 
	Now we prove that $f=\int fdP_0$, $V$-a.s. By contradiction, we assume that $V(\{\o\in\O:f(\o)=\int fdP_0\}^c)>0$. In particular, there exists $P_1\in\operatorname{core}(V)$ such that 
	\begin{equation}\label{eq:1001}
		P_1(\{\o\in\O:f(\o)\neq\int fdP_0\})>0.
	\end{equation}
	Let $P=\frac{1}{2}P_0+\frac{1}{2}P_1$. Since $\operatorname{core}(V)$ is convex, $P\in\operatorname{core}(V)$.
	According to \eqref{eq:8}, we have 
	$ P\times P(\widetilde{\O})=1.$
	By Fubini's theorem, for $P$-a.s. $\o_2\in\O$, $P(\widetilde{\O}_{\o_2})=1$, where $\widetilde{\O}_{\o_2}=\{\o_1\in\O:(\o_1,\o_2)\in\widetilde{\O}\}$. By \eqref{eq:15},  there exists $\o_2\in\O$ such that $f(\o_2)\neq 0$ and $P(\widetilde{\O}_{\o_2})=1$. Thus, for $P$-a.s. $\o_1\in\O$, 
	$f(\o_1)={a}/{\overline{f(\o_2)}}.$ As $P=\frac{1}{2}P_0+\frac{1}{2}P_1$, it follows that for $P_0$-a.s. $\o_1\in\O$, 
	$f(\o_1)={a}/{\overline{f(\o_2)}},$
	which implies that 
	\[\int f(\o_1)dP_0(\o_1)=\int ({a}/{\overline{f(\o_2)}})dP_0(\o_1)={a}/{\overline{f(\o_2)}}.\]
	Therefore, $f(\o_1)=\int fdP_0$ for $P$-a.s. $\o_1\in\O$, and hence it also holds $P_1$-a.s. This is a contradiction with \eqref{eq:1001}.
	So $f=\int fdP_0$, $V$-a.s.,  which shows that $V$ is weakly  mixing.
\end{proof}	
From Theorem \ref{thm:ergodic by independ.},  we have that in this situation the upper probability $V$ is weakly mixing, and $V\times V$ is ergodic. Applying Lemma \ref{lem:core of w.m.} on $V$ and Theorem \ref{lem:ergodic of core} on $V\times V$, we see that two ergodic probabilities $Q\times Q$ and $\widetilde{Q}$ on $(\O\times\O,\mathcal{F}\times\mathcal{F})$, respectively. 
The following result shows that these two probabilities are the same.
\begin{cor}\label{cor:weak Q}
	Let $(\O,\mathcal{F},V,T)$ be a weakly mixing capacity preserving system, where $V$ is an upper probability, and $Q\in\mathcal{ M}^{wm}(T)\cap \operatorname{core}(V)$ be as in Lemma \ref{lem:core of w.m.}. Then for any $\widetilde{P}\in \operatorname{core}(V\times V)$,
	\begin{equation*}
		\widetilde{P}(\widetilde{A})=Q\times Q(\widetilde{A})\text{ for any }\widetilde{A}\in\widetilde{\mathcal{I}}.
	\end{equation*}
	Furthermore,  $\mathcal{ M}(T\times T)\cap\operatorname{core}(V\times V)=\mathcal{ M}^e(T\times T)\cap\operatorname{core}(V\times V)=\mathcal{ M}^{wm}(T)\times \mathcal{ M}^{wm}(T)\cap\operatorname{core}(V\times V)=\{Q\times Q\}$.
\end{cor}
\begin{proof}
	By Theorem \ref{lem:ergodic of core}, there exists a unique $\widetilde{Q}\in\mathcal{ M}^{e}(T\times T)\cap\operatorname{core}(V\times V)$ such that for any $\widetilde{P}\in\operatorname{core}(V\times V)$,
	$\widetilde{P}|_{\widetilde{\mathcal{I}}}=\widetilde{Q}|_{\widetilde{\mathcal{I}}}$. It is well known that $Q\times Q\in\mathcal{ M}^e(T\times T)$, as $Q$ is weakly mixing. Due to   Lemma \ref{lem:invariant meausre }, to prove $\widetilde{Q}=Q\times Q$, we only need to prove that $Q\times Q|_{\mathcal{I}}=\widetilde{Q}|_{\mathcal{I}}$.
	Consider any $\widetilde{A}\in\widetilde{\mathcal{I}}$, then $\widetilde{Q}(\widetilde{A})=0$ or $1$, as  $\widetilde{Q}$ is ergodic. If $\widetilde{Q}(\widetilde{A})=0$, by Corollary \ref{cor:Q<=>V}, one has that $V\times V(\widetilde{A})=0$, which implies that $Q\times Q(\widetilde{A})=0$, as $Q\times Q\in\operatorname{core}(V\times V)$. Similarly, if $\widetilde{Q}(\widetilde{A})=1$, one can prove that $Q\times Q(\widetilde{A})=1$. Thus, $Q\times Q|_{\mathcal{I}}=\widetilde{Q}|_{\mathcal{I}}$, and hence  this corollary is proved.
\end{proof}
\subsection{Asymptotic independence and long time convergence in laws}
The following result shows that each element in the core of a weakly mixing upper probability has  a kind of asymptotic independence.
\begin{theorem}\label{cor:asy}
	Let $(\O,\mathcal{F},V,T)$ be a capacity preserving system, where $V$ is an upper probability. Then the following three statements are equivalent:
	\begin{longlist}
		\item $V$ is weakly mixing;
		\medskip
		\item  	there exists $Q\in\operatorname{core}(V)\cap\mathcal{ M}^{wm}(T)$ such that for any $P\in\operatorname{core}(V)$,
		\begin{equation*}
			\lim_{n\to\infty}\frac{1}{n}\sum_{i=0}^{n-1}|P(B\cap T^{-i}C)-P(B)Q(C)|^2=0\text{ for any }B,C\in\mathcal{F}.
		\end{equation*}
		\medskip
		\item 	there exists $Q\in\operatorname{core}(V)\cap\mathcal{ M}(T)$ such that for any $P\in\operatorname{core}(V)$,
		\begin{equation}\label{eq:521}
			\lim_{n\to\infty}\frac{1}{n}\sum_{i=0}^{n-1}|P(B\cap T^{-i}C)-P(B)Q(C)|^2=0\text{ for any }B,C\in\mathcal{F}.
		\end{equation}
	\end{longlist}
\end{theorem}
\begin{proof}
	(i) $\Rightarrow$ (ii).	Combining Corollary \ref{cor:non} and Corollary \ref{cor:weak Q}, one  deduces that there exists $Q\in\operatorname{core}(V)\cap\mathcal{ M}^{wm}(T)$ such that for any $P\in\operatorname{core}(V)$,
	\begin{equation}\label{eq33}
		\begin{split}
			\lim_{n\to\infty}&\frac{1}{n}\sum_{i=0}^{n-1}P( B\cap T^{-i}C)^2\\
			&=\lim_{n\to\infty}\frac{1}{n}\sum_{i=0}^{n-1}P\times P((B\times B)\cap (T\times T)^{-i}(C\times C))=P(B)^2Q(C)^2
		\end{split}
	\end{equation}
	for any $B,C\in\mathcal{F}.$ Thus, using Corollary \ref{cor:non} again, we have that for any $B,C\in\mathcal{F}$,
	\begin{align*}
		&\lim_{n\to\infty}\frac{1}{n}\sum_{i=0}^{n-1}|P(B\cap T^{-i}C)-P(B)Q(C)|^2\\
		=&\lim_{n\to\infty}\frac{1}{n}\sum_{i=0}^{n-1}P(B\cap T^{-i}C)^2-\lim_{n\to\infty}\frac{2}{n}\sum_{i=0}^{n-1}P(B\cap T^{-i}C)P(B)Q(C)+P(B)^2Q(C)^2\\
		=&P(B)^2Q(C)^2-2P(B)^2Q(C)^2+P(B)^2Q(C)^2=0,
	\end{align*}
	where the last equation is obtained by \eqref{eq33}.
	
	(ii) $\Rightarrow$ (iii).	This is trivial.

	(iii) $\Rightarrow$ (i). Since $Q\in\operatorname{core}(V)$, it follows from \eqref{eq:521} that 
	\[	\lim_{n\to\infty}\frac{1}{n}\sum_{i=0}^{n-1}|Q(B\cap T^{-i}C)-Q(B)Q(C)|^2=0\text{ for any }B,C\in\mathcal{F}.\]
	By \cite[Theorem 1.21]{Walters1982}, we have that $Q$ is weakly mixing. For any  $P\in\operatorname{core}(V)$ and $A\in\mathcal{I}$, taking $B=\O$ and $C=A$, we have that $Q(C)=P(C)$. This shows that 
	\[P|_{\mathcal{I}}=Q|_{\mathcal{I}}\text{ for any }P\in\operatorname{core}(V).\]
	By Lemma \ref{lem:core of w.m.}, $V$ is weakly mixing. The proof is completed.
\end{proof}

Finally, we use an equivalent characterization to end this section. Recall that for a subset $J$ of $\mathbb{N}$, we define the upper density of $J$ by 
\[\overline{D}(J)=\limsup_{n\to\infty}\frac{|J\cap\{1,2,\ldots,n\}|}{n}\]
and the lower density by 
\[\underline{D}(J)=\liminf_{n\to\infty}\frac{|J\cap\{1,2,\ldots,n\}|}{n}.\]
In particular, if $\overline{D}(J)=\underline{D}(J)$, we call the set $J$ has natural density, denoted by $D(J)$.

We recall that from  Theorems 1.20 and  1.23 in \cite{Walters1982}, given any weakly mixing measure preserving system $(\O,\mathcal{F},Q,T)$,   for any $f,g\in L^2(\O,\mathcal{F},Q)$, there exists a subset $J_{f,g}$ of $\mathbb{N}$ with $D(J_{f,g})=0$ such that
\begin{equation}\label{eq:43}
	\lim_{n\notin J_{f,g},n\to\infty}\int (f\circ T^n)\cdot gdQ=\int fdQ\int gdQ. 
\end{equation}
We change this slightly for our needs.
\begin{lemma}\label{lem:change}
	Let $(\O,\mathcal{F},Q,T)$ be a weakly mixing measure preserving system, where $(\O,\mathcal{F})$ is standard. Then for any $f\in L^\infty(\O,\mathcal{F},Q)$, there exists a subset $J=J_f$ of $\mathbb{N}$ such that 
	\[\lim_{n\notin J,n\to\infty}\int (f\circ T^n)\cdot gdQ=\int fdQ\int gdQ,\text{ for any }g\in L^1(\O,\mathcal{F},Q). \]
\end{lemma}
\begin{proof}
Fix $f\in L^\infty(\O,\mathcal{F},Q)$.	Since $(\O,\mathcal{F})$ is standard, we can find $\{g_m\}_{m\in\mathbb{N}}\subset L^2(\O,\mathcal{F},Q)$ such that it is a dense subset of $ L^1(\O,\mathcal{F},Q)$. By \eqref{eq:43}, for each $m\in\mathbb{N}$, there exists $J_m\subset\mathbb{N}$ with $D(J_m)=0$ such that 
	\begin{equation}\label{eq:12345}
		\lim_{n\notin J_{m},n\to\infty}\int (f\circ T^n)\cdot g_mdQ=\int fdQ\int g_mdQ. 
	\end{equation}
	Let $J'_m=\cup_{i=1}^mJ_i$ for each $m\in\mathbb{N}$. Then $D(J_m')=0$ and $J_m'\subset J_{m+1}'$ for each $m\in\mathbb{N}$.
	Thus, for any $m\in\mathbb{N}$, there exists $N_m\in\mathbb{N}$ such that for any $N\ge N_m$, $|J'_m\cap[1,N]|/N<1/m$. Let 
	\[J=\cup_{m=1}^\infty (J'_m\cap [N_{m-1}+1,N_m]).\]
	Then for any $K\in\mathbb{N}$, there exists $m\in\mathbb{N}$ such that $N_{m}+1\le K<N_{m+1}$, and hence $J\cap [1,K]\subset J'_m\cap [1,K]$.Thus, $\frac{|J\cap[1,K]|}{K}<1/m$, which shows that $D(J)=0$. Meanwhile, by the construction of $J$ and \eqref{eq:12345}, one has for any $m\in\mathbb{N}$, 
	\[\lim_{n\notin J,n\to\infty}\int (f\circ T^n)\cdot g_mdQ=\int fdQ\int g_mdQ. \]
	
	Since $f\in L^\infty(\O,\mathcal{F},Q)$,  $L:=\max\{\|f\|_{\infty,Q},1\}<\infty$. For any $g\in L^1(\O,\mathcal{F},Q)$, there exists a subsequence $\{q_m\}_{m\in\mathbb{N}}$ of $\{g_m\}_{m\in\mathbb{N}}$ such that $\lim_{m\to\infty}\|q_m-g\|_{1,Q}=0.$ Thus, for any $\e>0$, there exists $M>0$ such that for any $m\ge M$, $\|q_m-g\|_{1,Q}<\e/(4L).$ Therefore,
	\begin{align*}
		|\int& (f\circ T^n)\cdot gdQ-\int fdQ\int gdQ|\\
		\le&|\int (f\circ T^n)\cdot (g-q_m)dQ|
		+|\int (f\circ T^n)\cdot q_mdQ-\int fdQ\int q_mdQ|\\&+|\int fdQ\int (q_m-g)dQ|\\
		\le& L\cdot \frac{\e}{4L}+|\int (f\circ T^n)\cdot q_mdQ-\int fdQ\int q_mdQ|+L\cdot \frac{\e}{4L}\\
		\le& \e/2+|\int (f\circ T^n)\cdot q_mdQ-\int fdQ\int q_mdQ|,
	\end{align*}
	which shows that 
	\[\limsup_{n\notin J,n\to\infty}	|\int (f\circ T^n)\cdot gdQ-\int fdQ\int gdQ| \le \e/2.\]
	Letting $\e\to0$, one has that
	\begin{equation*}
		\lim_{n\notin J,n\to\infty}\int (f\circ T^n)\cdot gdQ=\int fdQ\int gdQ. 
	\end{equation*}
	The proof is completed, as $g\in L^1(\O,\mathcal{F},Q)$ is arbitrary.
\end{proof}
\begin{rem}\label{rem:weaker}
	If the measurable space $(\O,\mathcal{F})$ is not standard, then by a similar argument, we can prove a weaker result:
	for any $f\in L^\infty(\O,\mathcal{F},Q)$ and $g\in L^1(\O,\mathcal{F},Q)$, there exists a subset $J=J_{f,g}$ of $\mathbb{N}$ such that  
	\[\lim_{n\notin J,n\to\infty}\int (f\circ T^n)\cdot gdQ=\int fdQ\int gdQ. \]
\end{rem}
\begin{theorem}\label{thm:weak, density}
	Let $(\O,\mathcal{F},V,T)$ be a  capacity preserving system, where $(\O,\mathcal{F})$ is standard and $V$ is an upper probability. Then $V$ is weakly mixing if and only if there exists $Q\in\mathcal{ M}^{wm}(T)\cap\operatorname{core}(V)$   such that for any $f\in L^\infty(\O,\mathcal{F},Q)$, there exists a subset $J=J_f$ of $\mathbb{N}$ with  $D(J)=0$ such that for any $g\in L^\infty(\O,\mathcal{F},Q)$,
	\begin{equation}\label{eq:1}
		\lim_{n\notin J,n\to\infty}\int (f\circ T^n)\cdot gdP=\int fdQ\int gdP\text{ for any }P\in\operatorname{core}(V).
	\end{equation}
\end{theorem}
\begin{proof}
	($\Rightarrow$) Since $V$ is weakly mixing, it follows from Lemma \ref{lem:core of w.m.}, there exists $Q\in\mathcal{ M}^{wm}(T)\cap \operatorname{core}(V)$ such that $P|_{\mathcal{I}}=Q|_{\mathcal{I}}$ for all $P\in\operatorname{core}(V)$. By Corollary \ref{cor:Q<=>V}, we have that for any $A\in\mathcal{F}$ with $Q(A)=0$, one has that $V(A)=0$, and hence $P(A)=0$ for any $P\in\operatorname{core}(V)$.
	This means that $P\ll Q$ for any $P\in\operatorname{core}(V)$.

	Fix any $P\in\operatorname{core}(V)$. As $P\ll Q$, let $h=\frac{dP}{dQ}\in L^1(\O,\mathcal{F},Q)$ be the Radon-Nikodym derivative of $P$ with respect to  $Q$. Note that for any $g\in L^\infty(\O,\mathcal{F},Q)$,  $g\cdot h\in L^1(\O,\mathcal{F},Q)$. By Lemma \ref{lem:change}, one has that for any $f\in L^\infty(\O,\mathcal{F},Q)$, there exists a subset $J=J_f$ of $\mathbb{N}$ with $D(J)=0$  such that for any $g\in L^\infty(\O,\mathcal{F},Q)$,
	\begin{align*}
		\lim_{n\notin J,n\to\infty}\int (f\circ T^n)\cdot gdP&=\lim_{n\notin J,n\to\infty}\int (f\circ T^n)\cdot (g\cdot h)dQ\\
		&=\int fdQ\int g\cdot hdQ=\int fdQ\int gdP=\int fdQ\int gdP.
	\end{align*}

	($\Leftarrow$) For any $A\in\mathcal{F}$, consider $f=1_A$ and $g=1$ in \eqref{eq:1}. We have that for any $P\in\operatorname{core}(V)$,
	\[Q(A)=\lim_{n\notin J,n\to\infty}P(T^{-n}A)\le\limsup_{n\notin J,n\to\infty}V(T^{-n}A)= V(A).\]
	This shows that $Q\in \operatorname{core}(V)$.
	From Lemma \ref{lem:core of w.m.}, it suffices to prove that for any $P\in\operatorname{core}(V)$,  $P|_{\mathcal{I}}=Q|_{\mathcal{I}}$.
	This is clear by taking  $f=1_A$ for any $A\in\mathcal{I}$ and  $g=1$ in \eqref{eq:1}.
\end{proof}
\begin{rem}\label{rem:1}
	According to  Remark \ref{rem:weaker}, if the measurable space $(\O,\mathcal{F})$ is not standard, we have the following result:  $V$ is weakly mixing if and only if there exists $Q\in\mathcal{ M}^{wm}(T)\cap\operatorname{core}(V)$   such that for any $f,g\in L^\infty(\O,\mathcal{F},Q)$ and $P\in\operatorname{core}(V)$, there exists a subset $J=J_{f,g,P}$ of $\mathbb{N}$ with  $D(J)=0$ such that
	\begin{equation*}
		\lim_{n\notin J,n\to\infty}\int (f\circ T^n)\cdot gdP=\int fdQ\int gdP.
	\end{equation*}
\end{rem}
\subsection{Applications: asymptotic independence and convergence in laws for non-invariant probabilities}
In this subsection, we investigate  recurrent and ergodic theorems for non-invariant probabilities. Let us begin with the following lemma.
\begin{lemma}\label{lem: uniformly L1}
	Under the same conditions as in Theorem \ref{thm:app1}, for any $\epsilon>0$, there exists $\delta>0$ such that for any $f,g\in L^\infty(\O,\mathcal{F},Q)$ if $\|f-g\|_{1,Q}<\delta$ then 
	\[\|f-g\|_{1,P\circ T^{-m}}<\e,\text{ for any }m\in\mathbb{Z}_+.\]
\end{lemma}
\begin{proof}
	By Remark \ref{rem4.5} (iv), for any $\e>0$ there exists $\d'>0$ such that for any $A\in\mathcal{F}$ with $Q(A)<\delta'$,
	\[P (T^{-m}A)<\frac{\e}{2(\|f\|_{\infty,Q}+\|g\|_{\infty,Q})},\text{ for any }m\in\mathbb{Z}_+.\]
	Let $\d=\frac{\d'\e}{4}$. If $\|f-g\|_{1,Q}<\d$, then 
	\begin{align*}
		Q(A_\e)\le(2/\e)\cdot\int_{A_\e} |f-g|dQ\le (2/\e)\cdot\|f-g\|_{1,Q}<(2/\e)\cdot\d<\d',
	\end{align*}
	where $A_\e=\{\o\in\O:|f(\o)-g(\o)|>\e/2\}$.
	Thus, 
	\[P(T^{-m}A_\e)<\frac{\e}{2(\|f\|_{\infty,Q}+\|g\|_{\infty,Q})},\text{ for any }m\in\mathbb{Z}_+.\]
	This implies that for each $m\in\mathbb{Z}_+$,
	\begin{align*}
		\|f-g\|_{1,P\circ T^{-m}}=&\int_{A_\e}|f-g|d(P\circ T^{-m})+\int_{A_\e^c}|f-g|d(P\circ T^{-m})\\
		<& (\|f\|_{\infty,Q}+\|g\|_{\infty,Q})\cdot\frac{\e}{2(\|f\|_{\infty,Q}+\|g\|_{\infty,Q})}+\e/2\\
		\le&\e,
	\end{align*}
	proving the lemma.
\end{proof}
\begin{theorem}\label{thm:app2}
	Under the same conditions in Theorem \ref{thm:app1}, the following statements are equivalent:
	\begin{longlist}
		\item 	$Q$ is weakly mixing;
		\medskip
		\item $\lim_{\substack{n+m\to\infty\\m,n\ge0}}\frac{1}{m+n+1}\sum_{i=-m}^n|P(B\cap T^{-i} C)-P(B)Q(C)|^2=0\text{ for any }B,C\in\mathcal{F};$
		\medskip
		\item for any $f,g\in L^\infty(\O,\mathcal{F},Q)$, there exists a subset $J=J_{f,g}$ of $\mathbb{N}$ with $D(J)=0$ such that 
		\begin{equation}\label{eq:522}
			\lim_{n\notin J,n\to\infty}\int (f\circ T^n)\cdot gdP=\int fdQ\int gdP.
		\end{equation}
	\end{longlist}
	Moreover, these equivalent statements can imply that letting $p(x)$ be a polynomial with integer coefficients,  then for any $f \in L^r(\Omega,\mathcal{F},Q)$, $r>1$, 
	$$
	\lim _{n \rightarrow \infty} \frac{1}{n} \sum_{i=1}^{n} f(T^{p(i)} \o)=\int fdQ \text{ for }P\text{-a.s. }\o\in\O.
	$$
\end{theorem}
\begin{proof}
	(i) $\Rightarrow$ (ii).	Let $V$ be defined by \eqref{eq:8=}. Under the assumptions of this theorem, $Q\in \mathcal{ M}^{wm}(T)\cap \operatorname{core}(V)$ such that $Q|_{\mathcal{I}}=V|_{\mathcal{I}}$, and so $V$ is also weakly mixing.  Applying a similar argument of (iii)  in Remark \ref{rem4.5} on Theorem  \ref{cor:asy}, we obtain (ii).
	
	(ii) $\Rightarrow$ (iii). It follows from \cite[Theorem 1.20]{Walters1982} that (ii) is equivalent to that for any $B,C\in\mathcal{F}$, there exists a subset $J_{B,C}$ of $\mathbb{N}$ such that  $\lim_{n\not\in J_{B,C},n\to\infty}P(B\cap T^{-n} C)=P(B)Q(C)$. It turns out that for any simple functions $f,g$ there exists a subset $J_{f,g}$ of $\mathbb{N}$ such that  \[\lim_{n\not\in J_{f,g},n\to\infty}\int (f\circ T^n)\cdot gdP=\int fdQ\int gdP.\]
	Now we prove the above equation holds for any $f,g\in L^\infty(\O,\mathcal{F},Q)$. Given $f,g\in L^\infty(\O,\mathcal{F},Q)$, there exists two increasing sequences of simple functions $\{f_k\}_{k\in\mathbb{N}}$ and $\{g_k\}_{k\in\mathbb{N}}$ such that 
	\[\lim_{k\to\infty}\|f_k-f\|_{1,Q}=\lim_{k\to\infty}\|g_k-g\|_{1,Q}=0.\]
	By Lemma \ref{lem: uniformly L1}, for any $\epsilon>0$, there exists $K>0$ such that 
	\begin{equation}\label{eq520}
		\|f_K-f\|_{1,P\circ T^{-m}}\le \frac{\e}{8\|g\|_{\infty,Q}},\text{ }\|g_K-g\|_{1,P}\le \frac{\e}{8\|f\|_{\infty,Q}}\text{ for any }m\in\mathbb{Z}_+,
	\end{equation}
	and 
	\begin{equation}\label{eq521}
		\|f_K-f\|_{1,Q}\le \frac{\e}{8\|g\|_{\infty,Q}},\text{ }\|g_K-g\|_{1,Q}\le \frac{\e}{8\|f\|_{\infty,Q}}.
	\end{equation}
	Since $f_K$ and $g_K$ are simple functions, we have that 
	\[\lim_{n\not\in J_{f_K,g_K},n\to\infty}\int (f_K\circ T^n)\cdot g_KdP=\int f_KdQ\int g_KdP.\]
	Thus, there exists $N_K>0$ such that for any $n\notin J_{f_K,g_K}$ and $n\ge N_K$
	\begin{equation}\label{eq522}
		|\int (f_K\circ T^n)\cdot g_KdP-\int f_KdQ\cdot\int g_KdP|<\e/8.
	\end{equation}
	Since $P\ll Q$, one has that $\|h\|_{\infty,P}\le \|h\|_{\infty,Q}$ for any $\mathcal{F}$-measurable function $h$. Indeed, as $Q(\{\o\in\O:h(\o)> \|h\|_{\infty,Q}\})=0$, it follows that $P(\{\o\in\O:h(\o)> \|h\|_{\infty,Q}\})=0$, and hence  $\|h\|_{\infty,P}\le \|h\|_{\infty,Q}$. Thus, we have that  for any $n\notin J_{f_K,g_K}$ and $n\ge N_K$,
	\begin{align*}
		|\int&(f\circ T^n)\cdot gdP-\int fdQ\cdot\int gdP|\\
		\le& |\int(f\circ T^n)\cdot gdP-\int(f\circ T^n)\cdot g_KdP|\\
		&+|\int(f\circ T^n)\cdot g_KdP-\int(f_K\circ T^n)\cdot g_KdP|
		\\
		&+|\int(f_K\circ T^n)\cdot g_KdP-\int f_KdQ\cdot\int g_KdP|\\
		&+|\int f_KdQ\cdot\int g_KdP-\int f_KdQ\cdot\int gdP|
		\\
		&+|\int f_KdQ\cdot\int gdP-\int fdQ\cdot\int gdP|\\
		\le&  \|f\|_{\infty,Q}\cdot\|g-g_K\|_{1,P}+\|g\|_{\infty,Q}\cdot\|f-f_K\|_{1,P\circ T^{-n}}\\
		&+\e/8+\|f\|_{\infty,Q}\cdot\|g-g_K\|_{1,P}+\|g\|_{\infty,Q}\cdot\|f-f_K\|_{1,Q}\\
		\le&\e/8+\e/8+\e/8+\e/8+\e/8<\e.
	\end{align*}
	Similar to the construction of the subset with zero density in Lemma \ref{lem:change}, we can prove (iii).
	
	(iii) $\Rightarrow$ (i). Firstly, we prove that $Q$ is ergodic. Indeed, for any $A\in\mathcal{I}$, if $P(A)>0$, letting $f=1_{A^c}$ and $g=1_A$ in \eqref{eq:522}, we have that 
	$0=P(A\cap A^c)=Q(A^c)P(A),$
	which implies that $Q(A^c)=0$; if $P(A)=0$, letting $f=1_{A}$ and $g=1_{A^c}$ in \eqref{eq:522}, we have that 
	$Q(A)=0$. Thus, $Q(A)\in\{0,1\}$, and hence $Q$ is ergodic. 
	
	Now we prove $Q$ is weakly mixing. Fix any $\mathcal{F}$-measurable function $h$ with $\int |h|dQ>0$ such that $h\circ T=\l h$, $Q$-a.s. for some $\l\in\mathbb{C}\setminus\{0\}$. Since $Q$ is ergodic and $|\l|=1$, it follows that $|h|$ is constant, $Q$-a.s. In particular, $h\in L^\infty(\O,\mathcal{F},Q)$. Taking $f=\frac{\bar{h}}{|h|^2}$ and $g=h$ in \eqref{eq:522}, one has that 
	\begin{equation}\label{eq:123456}
		\lim_{n\to\infty, n\notin J_{f,g}}\l^n =\int \frac{\bar{h}}{|h|^2}dQ\int hdP.
	\end{equation}
	Since $|\l|=1$, we write $\l=e^{i\t}$ for some $\t\in\mathbb{R}$. If $\t\in \mathbb{Q}\setminus\{0\}$, then $\{\l^n\}_{n\in\mathbb{Z}_+}$ is a periodic sequence with a period of $t>1$, and hence for any constant $c\in\mathbb{C}$ there is no subset $J$ of $\mathbb{N}$ with $D(J)=0$ such that 
	$\lim_{n\to\infty, n\notin J}\l^n =c,$
	which is a contradiction with \eqref{eq:123456}. If $\t\notin\mathbb{Q}$, it is well known that $\{\l^n\}_{n\in\mathbb{Z}_+}$ is equidistributed in $[0,1)$ (see for example \cite[Theorem 1.4]{Ward2011}). Thus, for any constant $c\in\mathbb{C}$ there is no subset $J$ of $\mathbb{N}$ with $D(J)=0$ such that 
	$\lim_{n\to\infty, n\notin J}\l^n =c,$
	which is also a contradiction with \eqref{eq:123456}.  Thus, $\l=1$, and hence $h\circ T= f$, $Q$-a.s., which implies that $f$ is constant, $Q$-a.s., as $Q$ is ergodic. Therefore, $Q$ is weakly mixing. 
	
	The last statement is a direct corollary of Corollary \ref{thm:Birkhoff along sequence}.
\end{proof}
\begin{rem}
	Mirroring Remark \ref{rem:app1}, in the case that $P$ is an invariant probability, all the three results in the above theorem are the exactly same as the results in the classical ergodic theory with no extra condition. So our results in the case that $P$ is an invariant probability are sharp and hold true for a class of non-invariant probabilities.
\end{rem}
\section{Further applications to invariant probabilities}\label{sec:ex and open}
In this section, we see examples of ergodic and weakly mixing capacity preserving systems, and provide applications for invariant probabilities.
\subsection{Distortions of invariant probabilities and their applications in characterization of weakly mixing and periodic probabilities}
The first example shows that  each ergodic (resp. weakly mixing) measure preserving system can give rise to a non-trivial  ergodic (resp. weakly mixing) upper probability. Then the ergodic theory of upper probabilities leads to some new results on the classical ergodic theory for measure preserving systems.
\begin{example}\label{ex:1}
	Let $(\O,\mathcal{F},P,T)$ be a measure preserving system. Given a  concave strictly increasing continuous function $f:[0,1]\to [0,1]$ with $f(0)=0$ and $f(1)=1$, define a concave distortion of $P$ with respect to $f$ by $V_f(A):=f(P(A))$ for any $A\in\mathcal{F}$ (see \cite{CarlierDana} for more properties about this type of capacities). Note that for any $A,B\in\mathcal{F}$, $P(A\cap B)\le P(A),P(B)\le P(A\cup B)$. By the property of concave functions, 
	\[f(P(A\cap B))+f(P(A\cup B))\le f(P(A))+f(P(B)),\]
	which shows that $V_f$ is concave capacity. It is easy to see that $V_f$ is also $T$-invariant.
	Meanwhile, it is continuous by the continuity of the function $f$ and the probability $P$.  By the choice of $f$, we have that $f(P(A))\ge P(A)$ for any $A\in\mathcal{F}$, and hence $P\in\operatorname{core}(V_f)$.
	If the system $(\O,\mathcal{F},P,T)$ is  ergodic (resp. weakly mixing), it is easy to see from the definition of $V_f$ that $V_f|_{\mathcal{I}}=P|_{\mathcal{I}}$. So by Theorem \ref{lem:ergodic of core} (resp. Lemma \ref{lem:core of w.m.}),  we have that $V_f$ is also ergodic (resp. weakly mixing).
	
	Thus, all the results can be applied on the capacity preserving system $(\O,\mathcal{F},V_f,T)$. Some of them are new and striking to the classical theory. We can see some of them in the following.
\end{example}
In particular, we suppose that $P$ is an ergodic probability and consider $f(x)=\sqrt{x}$. Applying (iv) of  Theorem \ref{thm:} on $V_f$, we have
\[\liminf_{n\to\infty}\frac{1}{n}\sum_{i=0}^{n-1}P^{1/2}(B\cap T^{-i}C)\ge P^{1/2}(B)P(C)\text{ for any }B,C\in\mathcal{F}.\]
Meanwhile, by the concavity of the function $f(x)=x^{1/2}$, $x\ge0$, one has that  for any $B,C\in\mathcal{F}$,
\begin{align*}
	\limsup_{n\to\infty}\frac{1}{n}\sum_{i=0}^{n-1}P^{1/2}(B\cap T^{-i}C)
	\le\left(\lim_{n\to\infty}\frac{1}{n}\sum_{i=0}^{n-1}P(B\cap T^{-i}C)\right)^{1/2}= P^{1/2}(B)P^{1/2}(C).
\end{align*}
Thus, for any $B,C\in\mathcal{F},$
\begin{equation}\label{eq:bound}
	\begin{split}
		P^{1/2}(B)P(C)&\le\liminf_{n\to\infty}\frac{1}{n}\sum_{i=0}^{n-1}P^{1/2}(B\cap T^{-i}C)\\
		&\le\limsup_{n\to\infty}\frac{1}{n}\sum_{i=0}^{n-1}P^{1/2}(B\cap T^{-i}C)\le P^{1/2}(B)P^{1/2}(C).
	\end{split}
\end{equation}
\begin{rem}
	Under the assumption that $P$ is ergodic,  the limit $\lim_{n\to\infty}\frac{1}{n}\sum_{i=0}^{n-1}P(B\cap T^{-i}C)$ exists, but the limit $\lim_{n\to\infty}\frac{1}{n}\sum_{i=0}^{n-1}P^{1/2}(B\cap T^{-i}C)$ may not exist. In fact, for a general sequence $\{a_i\}_{i\in\mathbb{N}}\subset[0,1]$, even if the limit $\lim_{n\to\infty}\frac{1}{n}\sum_{i=1}^{n}a_i$ exists, we may not be able to obtain the limit $\lim_{n\to\infty}\frac{1}{n}\sum_{i=1}^{n}a_i^{1/2}$ exists. For example, we consider the sequence $\{a_i\}_{i\in\mathbb{N}}$ defined via 
	$a_i=1/4$ if $i\in(2^{2k-1},2^{2k}]$; $a_i=1/2$ if $i\in(2^{2k},2^{2k+1}]$ is even; $a_i=0$ if $i\in(2^{2k},2^{2k+1}]$ is odd for each $k\in\mathbb{N}$. In the following, we consider $\frac{1}{n}\sum_{i=1}^{n}a_i$ and $\frac{1}{n}\sum_{i=1}^{n}a_i^{1/2}$  instead for convenience of notations in this special example. By computation, $\lim_{n\to\infty}\frac{1}{n}\sum_{i=1}^{n}a_i=1/4,$ but the limit $\lim_{n\to\infty}\frac{1}{n}\sum_{i=1}^{n}a_i^{1/2}$ does not exist, as 
\begin{align*}
\lim_{k\to\infty}\frac{1}{2^{2k}}\sum_{i=1}^{2^{2k}}a_i^{1/2}&=\lim_{k\to\infty}\frac{1}{2^{2k}}\left(\sum_{i=1}^k\frac{1}{2}(2^{2i}-2^{2i-1})+\frac{1}{2}\sum_{i=1}^k\frac{1}{\sqrt{2}}(2^{2i-1}-2^{2i-2})\right)\\&=\frac{1}{3}+\frac{1}{6\sqrt{2}},
\end{align*}
	but
\begin{align*}
\lim_{k\to\infty}\frac{1}{2^{2k+1}}\sum_{i=1}^{2^{2k+1}}a_i^{1/2}&=\lim_{k\to\infty}\frac{1}{2^{2k+1}}\left(\sum_{i=1}^k\frac{1}{2}(2^{2i}-2^{2i-1})+\frac{1}{2}\sum_{i=0}^k\frac{1}{\sqrt{2}}(2^{2i+1}-2^{2i})\right)\\
&=\frac{1}{6}+\frac{1}{3\sqrt{2}}.
\end{align*}
	Thus, the $\liminf$ and $\limsup$ in \eqref{eq:bound} may not be equal.
\end{rem}

The following two propositions show that  the limit of  Ces\`aro summation $
\frac{1}{n}\sum_{i=0}^{n-1}P^{1/2}(B\cap T^{-i}C),$ for $n\in\mathbb{N},B,C\in\mathcal{F}$,
is closely related to the complexity of the dynamical system. Firstly, we characterize the weak mixing measure preserving systems by the limit is equal to the upper bound in \eqref{eq:bound}.
\begin{prop}\label{prop1}
	Let $(\O,\mathcal{F},P,T)$  be a measure preserving system. Then the following three statements are equivalent:
	\begin{longlist}
		\item $P$ is weakly mixing;
		\medskip
		\item for any $r>0$, $\lim_{n\to\infty}\frac{1}{n}\sum_{i=0}^{n-1}P^{r}(B\cap T^{-i}C)=P^{r}(B)P^{r}(C)$ for any $B,C\in\mathcal{F}$;
		\medskip
		\item    for any $B,C\in\mathcal{F}$, $
		\lim_{n\to\infty}\frac{1}{n}\sum_{i=0}^{n-1}P(B\cap T^{-i}C)=P(B)P(C),$	 and there exists $r=r_{B,C}\in(0,1/2]$ such that $\lim_{n\to\infty}\frac{1}{n}\sum_{i=0}^{n-1}P^{r}(B\cap T^{-i}C)=P^{r}(B)P^{r}(C)$.
	\end{longlist}
	In particular, $P$ is weakly mixing  if and only if $P$ is ergodic and $\lim_{n\to\infty}\frac{1}{n}\sum_{i=0}^{n-1}P^{1/2}(B\cap T^{-i}C)=P^{1/2}(B)P^{1/2}(C)$ for any $B,C\in\mathcal{F}$.
\end{prop}
\begin{proof}
	(i) $\Rightarrow$ (ii). Given any $B,C\in\mathcal{F}$, it is well known that (see \cite[Theorem 1.21]{Walters1982} for example) there exists a subset $J=J_{B,C}\subset\mathbb{N}$ with nature density $D(J)=0$ such that 
	\[\lim_{n\notin J,n\to\infty}P(B\cap T^{-n}C)=P(B)P(C),\]
	which implies that for any $r>0$,
	\[\lim_{n\notin J,n\to\infty}P^{r}(B\cap T^{-n}C)=P^{r}(B)P^{r}(C),\]
	which implies that (ii).
	
	(ii) $\Rightarrow$ (iii). This is trivial.

	(iii) $\Rightarrow$ (i). Given  any $B,C\in\mathcal{F},$ suppose that there exists $r=r_{B,C}\in(0,1/2]$ such that \[\lim_{n\to\infty}\frac{1}{n}\sum_{i=0}^{n-1}P^{r}(B\cap T^{-i}C)=P^{r}(B)P^{r}(C).\] Then we have the following claim.
	\begin{claim}\label{claim1}
		$\lim\limits_{n\to\infty}\frac{1}{n}\sum_{i=0}^{n-1}P^{1/2}(B\cap T^{-i}C)=P^{1/2}(B)P^{1/2}(C)$.
	\end{claim}
	\begin{proof}[Proof of  the Claim]
		By H\"older inequality and the first statement of (iii), we have 
		\[\limsup_{n\to\infty}\frac{1}{n}\sum_{i=0}^{n-1}P^{1/2}(B\cap T^{-i}C)\le \lim_{n\to\infty}\left(\frac{1}{n}\sum_{i=0}^{n-1}P(B\cap T^{-i}C)\right)^{1/2}=P^{1/2}(B)P^{1/2}(C).\]
		Using H\"older inequality again, we have that 
		\begin{align*}
			P^{r}(B)P^{r}(C)&=\lim_{n\to\infty}\frac{1}{n}\sum_{i=0}^{n-1}P^{r}(B\cap T^{-i}C)\\
			&\le\liminf_{n\to\infty}\frac{1}{n}\left(\sum_{i=0}^{n-1}P^{1/2}(B\cap T^{-i}C)\right)^{2r}\cdot n^{1-2r},
		\end{align*}
		which implies that 	\[\liminf_{n\to\infty}\frac{1}{n}\sum_{i=0}^{n-1}P^{1/2}(B\cap T^{-i}C)\ge P^{1/2}(B)P^{1/2}(C).\]
		Now we finish the proof of claim.
	\end{proof}
	The above claim,  together with the first statement of (iii), implies that 
	\begin{align*}
		&\lim_{n\to\infty}\frac{1}{n}\sum_{i=0}^{n-1}(P^{1/2}(B\cap T^{-i}C)-P^{1/2}(B)P^{1/2}(C))^2\\
		=&\lim_{n\to\infty}\left[\frac{1}{n}\sum_{i=0}^{n-1}P(B\cap T^{-i}C)-2\frac{1}{n}\sum_{i=0}^{n-1}P^{1/2}(B\cap T^{-i}C)P^{1/2}(B)P^{1/2}(C)+P(B)P(C)\right]\\
		=&0.
	\end{align*}
	Thus, 
	\begin{align*}
		&\lim_{n\to\infty}\frac{1}{n}\sum_{i=0}^{n-1}(P(B\cap T^{-i}C)-P(B)P(C))^2\\
		\le&4\lim_{n\to\infty}\frac{1}{n}\sum_{i=0}^{n-1}(P^{1/2}(B\cap T^{-i}C)-P^{1/2}(B)P^{1/2}(C))^2
		=0.
	\end{align*}
	As $B,C\in\mathcal{F}$ are arbitrary,  by \cite[Theorem 1.17]{Walters1982}, we finish the proof.
\end{proof}
Conversely, if the Ces\`aro summation above-mentioned converges to the lower bound in \eqref{eq:bound}, then the system is simple. Namely,
\begin{prop}\label{prop2}
	Let $(\O,\mathcal{F},P,T)$  be an ergodic measure preserving system, where $(\O,\mathcal{F})$ is a standard measurable space. Then the following two statements are equivalent:
	\begin{longlist}
		\item there exists $B\in\mathcal{F}$ with $P(B)>0$ such that for any $C\in\mathcal{F}$ with $C\subset B$,
		\begin{equation}\label{eq:500}
			\lim_{n\to\infty}\frac{1}{n}\sum_{i=0}^{n-1}P^{1/2}(B\cap T^{-i}C)=P^{1/2}(B)P(C);
		\end{equation}
		\item $P$ is a periodic probability, i.e., there exist $r\in\mathbb{N}$ and distinct points $\o_1,\ldots,\o_r\in\O$ such that $P(\{\o_i\})=\frac{1}{r}$, $i=1,2,\ldots,r$.
	\end{longlist}
	In this case, $r=\frac{1}{P(B)}\in\mathbb{N}$. 
\end{prop}
\begin{proof}
	(ii) $\Rightarrow$ (i). Let $B=\{\o_1\}$. Then we only need to check \eqref{eq:500}.
	Note that $P^{1/2}(B\cap T^{-i}B)=1/\sqrt{r}$, if $i=kr$ for some $k\in\mathbb{N}$, otherwise, $P^{1/2}(B\cap T^{-i}B)=0$. Thus, 
	\[\lim_{n\to\infty}\frac{1}{n}\sum_{i=0}^{n-1}P^{1/2}(B\cap T^{-i}B)=\frac{1}{r}\cdot \frac{1}{\sqrt{r}}=P^{1/2}(B)P(B).\]
	
	(i) $\Rightarrow$ (ii). 
	Let $B$ be as in assumption (i). Fix any $C\in\mathcal{F}$ with $C\subset B$. Let
	\[a_i=\left(\frac{P(B\cap T^{-i}C)}{P(B)}\right)^{1/2}\text{ for each }i\in\mathbb{Z}_+.\]
	Then $0\le a_i\le 1$ for any $i\in\mathbb{Z}_+$. From the assumption \eqref{eq:500} and $P$ being ergodic,
	\[\lim_{n\to\infty}\frac{1}{n}\sum_{i=0}^{n-1}a_i=P(C)=\lim_{n\to\infty}\frac{1}{n}\sum_{i=0}^{n-1}a_i^2.\]
	So $\lim_{n\to\infty}\frac{1}{n}\sum_{i=0}^{n-1}a_i(1-a_i)=0$. Thus, by Theorem 1.20 in \cite{Walters1982}, there exists a subset $J\subset\mathbb{Z}_+$ with natural density $D(J)=1$ such that 
	\[\lim_{n\in J,n\to\infty}a_n(1-a_n)=0.\]
	If $P(C)>0$ then there exists a subset $J_1\subset J$ with positive natural density such that $\lim\limits_{n\in J_1,n\to\infty}a_n=1$, that is, 
	\[\lim_{n\in J_1,n\to\infty}P(B\cap T^{-n}C)=P(B).\]
	This implies that $P(C)=P(B)$. If this is not true, note that
	\[P(B\cap T^{-n}C)\le P(C)<P(B).\]
	This is a contradiction. Thus, $P(C)=P(B)$. As $C\in\mathcal{F}$ with $C\subset B$ is arbitrary, it follows that $B$ is an atom. Since $(\O,\mathcal{F})$ is standard, it follows that each atom is singleton, denoted by  $B=\{\o\}$ for some $\o\in\O$. Since $P(\cup_{i=0}^\infty T^i\{\o\})\le1$, and $T$ is measure-preserving, so there exists a smallest integer $r\in\mathbb{N}$ such that $T^r\o=\o$. Since $P$ is ergodic, and $\{\o,T\o,\ldots,T^{r-1}\o\}$ is an invariant set, it follows that 
	\[P(\{\o,T\o,\ldots,T^{r-1}\o\})=1.\]
	Let $\o_i=T^{i-1}\o$ for $i=1,2,\ldots,r$. Then $P(\o_i)=\frac{1}{r}=P(\o)=P(B)$ for $i=1,2,\ldots,r$. In particular, $r=\frac{1}{P(B)}$.
\end{proof}

More generally, we have the following. The proof is similar to that of \eqref{eq:bound} by using concavity of the function $f$, so it is omitted.
\begin{prop}
	Let $(\O,\mathcal{F},P,T)$ be a measure preserving system and $f:[0,1]\to [0,1]$  be a concave increasing continuous function with $f(0)=0$ and $f(1)=1$. Then for any $B,C\in\mathcal{F}$,
	\begin{align*}
		f(P(B))P(C)\le\liminf_{n\to\infty}&\frac{1}{n}\sum_{i=0}^{n-1}f(P(B\cap T^{-i}C))\\
		&\le\limsup_{n\to\infty}\frac{1}{n}\sum_{i=0}^{n-1}f(P(B\cap T^{-i}C))\le f(P(B)P(C)).
	\end{align*}
\end{prop}

\subsection{Counter example of ergodic capacity preserving system in number theory without Birkhoff's law of large numbers}\label{sec:number}
In this subsection, we will prove the law of large numbers does not hold for any invariant capacity on $(\mathbb{Z},2^\mathbb{Z})$ with respect to $T:\mathbb{Z}\to\mathbb{Z},n\to n+1$, due it lack of continuity of the capacities on   $(\mathbb{Z},2^\mathbb{Z})$. For convenience, we denote by $[m,n]=\{m,m+1,\ldots,n\}$ for any $m<n\in\mathbb{Z}$.

Given any  capacity preserving system $(\mathbb{Z},2^\mathbb{Z},\mu,T)$, we prove  Birkhoff's ergodic theorem does not hold for $\mu$. By contradiction, if it holds, then for any $A\subseteq\mathbb{Z}$,  there exists $B\subset\mathbb{Z}$ with $\mu(B^c)=0$ (i.e., $B\neq \emptyset$) and $c\ge0$  such that for any $m\in B$,
\[\lim_{n\to\infty}\frac{1}{2n+1}|A\cap[m-n,m+n]|=\lim_{n\to\infty}\frac{1}{n}\sum_{i=0}^{n-1}1_A(T^im)=c.\]
Thus, 
\[\lim_{n\to\infty}\frac{1}{2n+1}|A\cap[-n,n]|=c.\]
This means that the natural density of any subset $A$ of $\mathbb{Z}$ exists.
However, it is not true, for instance, the set $A=\bigcup_{n=0}^\infty [2^{2n},2^{2n+1}]$ does not have a natural density. This can be seen as
\[\limsup_{n\to\infty}\frac{|A\cap[-n,n]|}{2n+1}\ge\lim_{n\to\infty}\frac{|A\cap[-2^{2n+1},2^{2n+1}]|}{2^{2n+2}+1}>\lim_{n\to\infty}\frac{2^{2n+1}-2^{2n}}{2^{2n+2}+1} =\frac{1}{4},\]
but 
\[\liminf_{n\to\infty}\frac{|A\cap[-n,n]|}{2n+1}\le\lim_{n\to\infty}\frac{|A\cap[-2^{2n},2^{2n}]|}{2^{2n+1}+1}\le\lim_{n\to\infty}\frac{2^{2n-1}}{2^{2n+1}+1} =\frac{1}{4}.\]
\begin{rem}
	As a corollary of Theorem \ref{thm:main Birk} and the above consequence, there is no $T$-ergodic upper probability on $(\mathbb{Z},2^\mathbb{Z})$.
\end{rem}
Now we study some well-known concrete examples of capacities  on $(\mathbb{Z},2^\mathbb{Z})$, but indeed they are not upper probabilities. Firstly, we prove $\bar{d}$ in Example \ref{ex:number theory}, introduced in the introduction, is an ergodic subadditive capacity.
\begin{example}\label{ex:recall}
	We recall the capacity preserving system $(\mathbb{Z},2^\mathbb{Z},\bar{d},T)$ defined in Example \ref{ex:number theory}. Now we prove that $\bar{d}$ is ergodic with respect to $T$. Note that $T^{-1}A=\{n\in\mathbb{Z}:n+1\in A\}$. There are only two possible sets  $A=\emptyset$ and $A=\mathbb{Z}$ satisfying $T^{-1}A=A$. In particular, $\bar{d}(A)=0$ or $\bar{d}(A)=1$ and $\bar{d}(A^c)=0$, proving the ergodicity of $\bar{d}$. But consider $A_k=[k,\infty)$ for all $k\in\mathbb{N}$. Then $A_k\downarrow\emptyset$, but for any fixed $k\in\mathbb{N}$, $\bar{d}(A_k)=1/2$. Thus, $\bar{d}$ is not continuous from above. Meanwhile, we consider $B_k=[-k,k]$ for each $k\in\mathbb{N}$. Then $A_k\uparrow\mathbb{Z}$, but for any fixed $k\in\mathbb{N}$, $\bar{d}(A_k)=0$. Thus, $\bar{d}$ is not continuous from below. So the subadditive capacity $\bar{d}$ is not an upper probability.
\end{example}
Recall that a continuous concave capacity must be an upper probability, and hence if it is ergodic then the Birkhoff's ergodic theorem holds. The following example shows that there exists an ergodic concave capacity which is continuous from below such that Birkhoff's ergodic theorem does not hold.
\begin{example}\label{ex:concave capacity}
	Let $T:\mathbb{Z}\to\mathbb{Z}$, $x\mapsto x+1$, and $2^\mathbb{Z}$ be the family consisting of all subsets of $\mathbb{Z}$. Define 
	\[\mu=\max_{n\in\mathbb{Z}}\delta_n,\]
	where $\delta_n$ is the Dirac measure for $n\in\mathbb{Z}$. Then it is easy to check that $\mu$ is an invariant concave  capacity continuous from below, and in particular $(\mathbb{Z},2^\mathbb{Z},\mu,T)$ is a capacity preserving system. The ergodicity of $\mu$ can be obtained by the same argument in Example \ref{ex:recall}. However, by the same argument as in  Example \ref{ex:recall}, $\mu$ is not continuous from above, and hence $\mu$ is also not an upper probability.
\end{example}

\subsection{Examples: ergodicity but not weak mixing}
We recall that a weakly mixing subadditive capacity must be ergodic. In the case for probabilities, there are many examples to show that an ergodic probability may not be weakly mixing (for example irrational rotations on the torus). Now we provide some examples for more general capacities.
\begin{example}
	Let $(\mathbb{Z},2^\mathbb{Z},\mu,T)$ be the ergodic system in Example \ref{ex:concave capacity}. We claim that $\mu$ is not weak mixing. To see this, we consider the measurable function $f(n)=\lambda^n$ on $\mathbb{Z}$, where $\lambda\in\mathbb{C}$ with $|\l|=1$. Then $f(Tn)=f(n+1)=\l f(n)$ for each $n\in\mathbb{Z}$. Since $f$ is not  constant, it follows that $\mu$ is not weakly mixing.
\end{example}

\begin{example}
	Let  $\O_i=[i-1,i)$, $P_i$ be the Lebesgue measure on $\O_i$  for $i=1,2$, and  $\O=\O_1\cup\O_2$. Let $S:\O_1\to \O_1,x\mapsto 2x\pmod1$ be the doubling map. It is well known that $(\O,\mathcal{B}(\O_1),P_1,S)$ is a weakly mixing measure preserving system (see \cite{Quas2011} for example). 
	Define $T:\O \rightarrow \O$ by
	$$
	T(\o)= \begin{cases}(2\o\bmod 1)+1, & \o \in [0,1), \\ \o-1, & \o\in [1,2) .\end{cases}
	$$
	For each $i=1,2$, let
	$$
	\bar{P}_i(A):=P_i\left(A \cap\O_i\right), \quad \text { for any } A \in \mathcal{B}(\O),
	$$
	and define the upper probability
	$$
	V=\max\{\bar{P}_1,\bar{P}_2\}.
	$$
	
	Firstly, we prove $V$ is $T$-invariant. Indeed, for any $A\in\mathcal{B}(\O)$,
	\begin{align*}
		V(T^{-1}A)=&\max\{{P}_1(T^{-1}A\cap\O_1),{P}_2(T^{-1}A\cap\O_2)\}\\
		=&\max\{{P}_2(A\cap\O_2),{P}_1(A\cap\O_1)\}\\
		=&V(A).
	\end{align*}
	
	Next, we prove that $V$ is ergodic. Since $(\O,\mathcal{B}(\O_1),P_1,S)$ is weakly mixing, it follows that the system $(\O_i,\mathcal{B}(\O_i),P_i,T^2|_{\O_i})$ is weakly mixing for $i=1,2$. Given any measurable function $f$ with $f\circ T= f$, $V$-a.s.,  then for $i=1,2$,
	\[f\circ T^2|_{\O_i}=f \text{, }P_i\text{-a.s.}\]
	Thus, for each $i=1,2$ there exists $A_i\subset \O_i$ with $P_i(A_i)=1$ and a constant $c_i\in\mathbb{C}$ such that for any $\o\in A_i$, $f(\o)=c_i$. As $P_1(A_2-1)=1$, where $A_2-1=\{\o\in \O_1:\o+1\in A_2\}$, there exists $\o\in A_1\cap S^{-1}(A_2-1)$, and hence 
	\[c_1=f(\o)=f(T\o)=f(S\o+1)=c_2.\]
	Let $A=A_1\cup A_2$. Then $V(A^c)= \max\{{P}_1(\O_1\setminus A_1),{P}_2(\O_2\setminus A_2)\}=0$ and  $f(\o)=c_1$ for any $\o\in A$. Thus, $V$ is ergodic.
	
	Finally, we prove it is not weakly mixing. Let 	$$
	f(\o)= \begin{cases}1, & \o\in [0,1), \\ -1, & \o \in [1,2) .\end{cases}
	$$
	Then $f\circ T=-f$, and $f$ is not constant, $V$-a.s., which shows that $V$ is not weakly mixing.
\end{example}
\section{Subadditive ergodic theorem for capacities}\label{sec:subadditive}
In the last section of this paper, we apply the common conditional expectation and invariant skeleton to  study the subadditive ergodic theorem on upper probability spaces, which provides a way to understand the long-term behavior of subadditive functions. Subadditive functions have applications in a variety of fields, including probability theory, information theory, and statistical physics.  
\subsection{Proof of subadditive ergodic theorem for invariant upper probabilities}
We recall that a sequence of $\mathcal{F}$-measurable functions $\{f_n\}_{n\in\mathbb{N}}$ on the capacity space $(\O,\mathcal{F},\mu)$ is said to be subadditive (resp. superadditive) if for each $k,n\in\mathbb{N}$, $f_{n+k}\le f_n+f_k\circ T^n$ (resp. $f_{n+k}\ge f_n+f_k\circ T^n$), $\mu$-a.s. If a sequence is subadditive and superadditive, then it is said to be additive. Given an $\mathcal{F}$-measurable function $g$, let 
$f_n=\sum_{i=0}^{n-1}g\circ T^i \text{ for each }n\in\mathbb{N}.$
Then $f_{n+k}=f_n+f_k\circ T^n$ for each $k,n\in\mathbb{N}$, and hence it is additive. However, for example, $\{|f_n|\}_{n\in\mathbb{N}}$ is only subadditive.
Thus, subadditive ergodic theorem can be viewed as an extension of Birkhoff's ergodic theorem.
We remark that as the proofs of  subadditive and superadditive sequences are similar, we only state and prove the results for subadditive sequences.

In the following, a standard setup is a capacity preserving system $(\O,\mathcal{F},V,T)$, where $(\O,\mathcal{F})$ is a standard measurable space, $T:\O\to \O$ is a measurable transformation and $V$ is a $T$-invariant upper probability.
\begin{theorem}\label{thm:main theorem}
	Suppose that $\{f_n\}_{n\in\mathbb{N}}$ is a sequence of $\mathcal{F}$-measurable functions satisfying the following conditions:
	\begin{longlist}
		\item there exists $\l >0$ such that $-\l n\le f_n(\o)\le \l n$ for any $n\in\mathbb{N}$, and $\o\in\O$;
		\medskip
		\item for each $k,n\in\mathbb{N}$, $f_{n+k}\le f_n+f_k\circ T^n$, $V$-a.s.
	\end{longlist} 
	Then there exists $f^*\in B(\O,\mathcal{I})$ such that 
	\[\lim_{n\to\infty}\frac{1}{n}f_n(\o)=f^*(\o)\text{ for }V\text{-a.s. }\o\in\O.\]
	
	If, in addition, $V$ is ergodic, then $f^*$ is a constant $V$-a.s.
\end{theorem}
\begin{proof}
	Define \[S_n=\frac{1}{n}f_n\text{ for each }n\in\mathbb{N}.\]
	By the assumption (i), one has $S_n\in B(\O,\mathcal{F})$ for each $n\in\mathbb{N}$. By Lemma \ref{lem:common conditional expectation}, for each $n\in\mathbb{N}$, there exists $\hat{S}_n\in B(\O,\mathcal{I})$ such that for any $\eta\in\mathcal{M}(T)$,
	\[\hat{S}_n=\mathbb{E}_\eta(S_n\mid \mathcal{I}), \quad \eta\text{-a.s.}\]
	Let $f^*=\inf_{n\in\mathbb{N}} \hat{S}_n$ and $\O^*=\{\o\in\O:\lim_{n\to\infty}S_n(\o)=f^*(\o)\}$. Then by Theorem \ref{thm:standard subadditive}, one has 
	$
	\eta(\O^*)=1\text{ for any }\eta\in\mathcal{M}(T)\cap\operatorname{core}(V).
	$
	In particular, for any $P\in\operatorname{core}(V)$, its invariant skeleton satisfies that
	$
	\hat{P}((\O^*)^c)=0.
	$
	By Lemma \ref{lem:core of invariant lower probabilities} (ii), we deduce that
	$V((\O^*)^c)=0.$
	
	Suppose that $V$ is ergodic. Applying Lemma \ref{lem:char. for ergodic} on the $T$-invariant function $f^*$, we have $f^*$ is  constant, $V$-a.s.
\end{proof}
\begin{rem}
	\begin{longlist}
		\item 		Cerreia-Vioglio, Maccheroni and Marinacci \cite{CMM2016} have considered subadditive ergodic theorem on upper probability spaces $(\O,\mathcal{F},V)$. However, they imposed the additional condition that there exists a compact subset $\L$ of $\mathcal{ M}(T)$ such that $V=\max_{P\in\L}P$. We recall Example \ref{ex:upper}, in which $V$ does not satisfy this condition.
		\medskip
		\item When $V$ is ergodic, $\operatorname{core}(V)\cap\mathcal{ M}(T)$ has only one element. So we do not need to use the common conditional expectation. Therefore, in this case, subadditive ergodic theorem holds for general probability spaces, not necessarily standard ones.
	\end{longlist}
\end{rem}
Continuing with the viewpoint of Theorem \ref{thm:app1}, the following result provides the subadditive ergodic theorem for a class of non-invariant probabilities.
\begin{theorem}\label{thm:sub for noninvariant}
	Let $(\O,\mathcal{F},P)$ be a probability space (not necessarily standard), and $T:\O\to \O$ be an invertible measurable transformation. Suppose that the limit $\lim_{n\to\infty}\frac{1}{n}\sum_{i=0}^{n-1}P\circ T^{i}$ exists, denoted by $Q$. 	Suppose that $\{f_n\}_{n\in\mathbb{N}}$ is a sequence of $\mathcal{F}$-measurable functions satisfying the following conditions:
	\begin{longlist}
		\item there exists $\l >0$ such that $-\l n\le f_n(\o)\le \l n$ for any $n\in\mathbb{N}$, and $\o\in\O$;
		\medskip
		\item for each $k,n\in\mathbb{N}$, $f_{n+k}\le f_n+f_k\circ T^n$.
	\end{longlist} 
	If $Q$ is ergodic, then there exists a constant $c\in\mathbb{R}$ such that 
	\[\lim_{n\to\infty}\frac{1}{n}f_n=c,\text{ }P\text{-a.s.}\]
\end{theorem}
\subsection{The multiplicative ergodic theorem for capacities}
In this subsection,	 we apply Theorem \ref{thm:main theorem} to obtain an extension of Furstenberg-Kesten theorem (\cite{FurstenbergKesten1960}, see also \cite[Corollary 10.1.1]{Walters1982} or \cite[Theorem 3.3.3]{Arnold199}) to upper probability spaces, and further using it to prove the multiplicative ergodic theorem \cite{Oseledec1968} on upper probability spaces. Let us begin notations. Denote by $\mathbb{R}^{d\times d}$ the space of all linear operators on $\mathbb{R}^d$. By choosing a basis of $\mathbb{R}^d$, we can view $\mathbb{R}^{d\times d}$ as the space of all $d\times d$ matrices. For any $A\in \mathbb{R}^{d\times d}$, denote by $A^*$ its transpose matrix. 

Let $(\Omega, \mathcal{F},\mu, T)$ be a capacity preserving system. Define
\begin{equation}\label{eq:star}
	\Phi(n, \omega)=L(T^{n-1} \omega) \cdots L(\omega), \quad n \geq 1 \text{ and }\o\in\O,
\end{equation}
where $L(\omega):=\Phi(1, \omega) \in \mathbb{R}^{d\times d}$ is the generator of $\Phi$. Let $\wedge^k \mathbb{R}^d$, $1 \leq k \leq d$, be the $k$-fold exterior power of $\mathbb{R}^d$ (see \cite[Chapter V]{Temam1988} for more details).
The cocycle property
$$
\Phi(n+m, \omega)=\Phi\left(m, T^n \omega\right) \Phi(n, \omega)
$$
lifts to $\wedge^k \mathbb{R}^d$, $1 \leq k \leq d$ (see \cite[Lemma 3.2.6]{Arnold199})
\begin{equation}\label{eq:10}
	\wedge^k \Phi(n+m, \omega)=(\wedge^k \Phi)(m, T^n \omega)(\wedge^k \Phi)(n, \omega).
\end{equation}
In the following, we always suppose that $\|\cdot\|$ is a matrix norm on $\mathbb{R}^{d\times d}$, i.e., a norm on $\mathbb{R}^{d\times d}$ with additional property $\|AB\|\le\|A\|\|B\|$.

Now we prove the Furstenberg-Kesten theorem on upper probability spaces.
\begin{theorem}[Furstenberg-Kesten theorem for upper probabilities]\label{thm:FK theorem}
	Let $\Phi$ be defined, by \eqref{eq:star}, on the capacity preserving system $(\Omega, \mathcal{F}, V,T)$, where $V$ is an upper probability.  If the generator $L:\O\to \mathbb{R}^{d\times d}$ of $\Phi$ is a measurable function such that $\log\|L(\o)\|\in B(\O,\mathcal{F})$, then 
	\begin{longlist}
		\item	for each $k=1, \ldots, d$ the sequence $\{	f_n^{(k)}(\omega)\}_{n\in\mathbb{N}}$ defined by
		$$
		f_n^{(k)}(\omega):=\log \left\|\wedge^k \Phi(n, \omega)\right\| \text{ for each }n \in \mathbb{N}
		$$
		is subadditive;
		\item there exist $\gamma^{(k)} \in B(\O,\mathcal{F})$ for $k=1,2,\ldots,d$ such that 
		\begin{equation}\label{eq:jing}
			\lim _{n \rightarrow \infty} \frac{1}{n} \log \left\|\wedge^k \Phi(n, \cdot)\right\|=\gamma^{(k)}
			\text{, }V\text{-a.s.}
		\end{equation}
	\end{longlist}
\end{theorem}
\begin{proof}	Fix $k\in\{1,2,\ldots,d\}$.
	Since $\|(\wedge^kL_1)  \cdot(\wedge^kL_2)\|\le \|\wedge^kL_1\|\cdot\|\wedge^kL_2\|$ for any $L_1,L_2\in \mathbb{R}^{d\times d}$ (see \cite[Lemma 3.2.6 (vi)]{Arnold199}) it follows from \eqref{eq:10} that $\{	f_n^{(k)}(\omega)\}$ is subadditive.
	Meanwhile, since $\log\|L(\o)\|\in B(\O,\mathcal{F})$, there exists $M>0$ such that 
	\[\left|\log{\|L(\o)\|}\right|\le M\text{ for any }\o\in\O,\]
	which shows that for each $n\in\mathbb{N}$,
	\begin{equation}\label{eq:12}
		|f^{(k)}_n(\o)|\le knM\text{ for any }n\in\mathbb{N}.
	\end{equation}
	
	The proof of (ii) is completed by applying Theorem \ref{thm:main theorem} on the sequence $\{	f_n^{(k)}(\omega)\}_{n\in\mathbb{N}}$ for each $k=1,2,\ldots,d$.
\end{proof}

\begin{theorem}[Multiplicative ergodic theorem for upper probabilities]\label{thm:MET}
	Under the same conditions in Theorem \ref{thm:FK theorem}, there exists $\tilde{\O}\in\mathcal{I}$ with $V(\tilde{\O}^c)=0$ such that for any $\o\in\tilde{\O}$,
	\begin{longlist}
		\item	The limit $\lim_{n\to\infty}(\Phi(n,\o)^*\Phi(n,\o))^{1/2n}:=\Psi(\o)$ exists.
		\item Let $e^{\l_{p(\o)}(\o)}<\ldots<e^{\l_{1}(\o)}$ be the different eigenvalues of $\Psi(\o)$, where $p(\o)$ is the number of  the different eigenvalues of $\Psi(\o)$, and let $U_{p(\o)}(\o),\ldots,U_{1}(\o)$ be the corresponding eigenspaces with multiplicities $d_i(\o):=\operatorname{dim}U_i(\o)$. Then
		\[p(T\o)=p(\o),\l_i(T\o)=\l_i(\o),\text{ and }d_i(T\o)=d_i(\o)\text{ for all }i\in\{1,2,\ldots,p(\o)\}.\]
		\item Put $V_{p(\o)+1}(\o):=\{0\}$ and for $i=1,2,\ldots,p(\o)$
		\[V_i(\o):=U_{p(\o)}(\o)\oplus\ldots\oplus U_i(\o)\]
		such that 
		$$
		V_{p(\o)}(\o) \subset \ldots \subset V_i(\o) \subset \ldots \subset V_1(\o)=\mathbb{R}^d.
		$$
		Then for each $x \in \mathbb{R}^d \backslash\{0\}$ the Lyapenov exponent
		$$
		\lambda(\o,x):=\lim _{n \rightarrow \infty} \frac{1}{n} \log \left\|\Phi(n,\o) x\right\|
		$$
		exists as a limit and
		$$
		\lambda(\o,x)=\lambda_i(\o) \Leftrightarrow x \in V_i (\o)\backslash V_{i+1}(\o)
		$$
		equivalently 
		\[V_i(\o)=\{x\in\mathbb{R}^d:\l(\o,x)\le\l_i(\o)\}.\]
		\item For all $x\in\mathbb{R}^d\setminus\{0\}$
		\[\l(T\o,L(\o)x)=\l(\o,x),\]
		whence 
		\[L(\o)V_i(\o)\subset V_i(T \o)\text{ for all }i\in\{1,\ldots,p(\o)\}.\]
		\item If $V$ is ergodic then the function $p$ is constant on $\tilde{\O}$, and the functions $\l_i$ and $d_i$ are constant on $\{\o\in\tilde{\O}:p(\o)\ge i\}$, $i=1,2,\ldots,d$.
	\end{longlist}
\end{theorem}
\begin{proof}
	Under the assumption that $\log\|L(\o)\|\in B(\O,\mathcal{F})$, it is easy to see that for any $\o\in\O$, $\limsup_{n\to\infty}\frac{1}{n}\log\|L(T^n\o)\|\le0$. Moreover, $\Phi_n$ defined as in \eqref{eq:star} also satisfies \eqref{eq:jing} on a subset $\O_1$ of $\O$ with $V(\O_1^c)=0$. Thus,  (i), (ii) and (iii) are true on $\O_1$ from Oseledec's deterministic multiplicative ergodic theorem (see \cite[Proposition 3.4.2]{Arnold199}).
	
	Now we check (iv) holds for any $\o\in\O_1$. Note that $\Phi(n,T\o)L(\o)=\Phi(n+1,\o)$ for any $n\in\mathbb{N}$ and $\o\in\O$. Thus, for any $\o\in\O_1$, by (iii), one has 
	\[\l(T\o,L(\o)x)=\lim_{n\to\infty}\frac{1}{n}\log\|\Phi(n,T\o)L(\o)x\|=\lim_{n\to\infty}\frac{1}{n}\log\|\Phi(n+1,\o)x\|=\l(\o,x).\]
	Using (iii) again, for any $x\in V_i(\o)$, $\l(\o,x)\le \l_i(\o)$, which shows that 
	\[\l(T\o,L(\o)x)=\l(\o,x)\le \l_i(\o)\overset{(ii)}{=}\l_i(T\o).\]
	Therefore, $L(\o)x\in V_i(T\o)$, which by the arbitrariness of $x\in V_i(\o)$, implies that $L(\o)V_i(\o)\subset V_i(T\o)$.
	
	Since $p$ is $T$-invariant on $\O_1$ and $V(\O_1^c)=0$, it follows from Lemma \ref{lem:char. for ergodic} that $p$ is constant on some measurable set $\O_2\subset\O_1$ with $V(\O_2^c)=0$. Using Lemma \ref{lem:char. for ergodic} again, there exists $\widetilde{\O}\subset \O_2$ with $V(\widetilde{\O}^c)=0$ such that the functions $\l_i$ and $d_i$ are constant on $\{\o\in\tilde{\O}:p(\o)\ge i\}$, $i=1,2,\ldots,d$. The proof is completed.
\end{proof}
As a direct corollary of Theorems \ref{thm:sub for noninvariant} and \ref{thm:MET}, one has that multiplicative ergodic theorem holds for a class of non-invariant probabilities. 
\begin{theorem}\label{thm:met for non}
	Let $(\O,\mathcal{F},P)$ be a probability space, and $T:\O\to \O$ be an invertible measurable transformation. Suppose that the limit $\lim_{n\to\infty}\frac{1}{n}\sum_{i=0}^{n-1}P\circ T^{i}$ exists, denoted by $Q$.  Then (i)-(v) in Theorem \ref{thm:MET} holds, $P$-a.s.
\end{theorem}
\begin{acks}[Acknowledgments]
	The authors would also like to thank Jie Li,
Baoyou Qu and Maoru Tan for their helpful discussions 
concerning this paper. 
\end{acks}

\begin{funding}
The second author was partially supported by NNSF of China (12090012, 12090010, 12031019).

The third author was supported by CSC of China (No. 202206340035), and partially supported by NNSF of China (12090012, 12090010).

The fourth author was supported  by the Royal Society Newton Fund (ref. NIF \textbackslash R1\textbackslash 221003) and the EPSRC  
(ref. EP/S005293/2).
\end{funding}

\end{document}